\pdfoutput=1
\documentclass[11pt]{article}

\usepackage{amsfonts}
\usepackage{amsmath}
\usepackage{amsthm}
\usepackage{amssymb}

\usepackage{mdframed} 
\usepackage{thmtools}
\usepackage{mathtools}

\declaretheorem[style=shaded,within=section]{definition}
\declaretheorem[style=shaded,sibling=definition]{theorem}

\declaretheorem[style=shaded,sibling=definition]{assumption}
\declaretheorem[style=shaded,sibling=definition]{corollary}
\declaretheorem[sibling=definition]{remark}
\declaretheorem[sibling=definition]{example}
\declaretheorem[style=shaded,sibling=definition]{lemma}

\usepackage{colordvi,epsfig}
\usepackage[table]{xcolor}   
\usepackage[noend]{algpseudocode}
\usepackage{verbatim}
\usepackage[ruled,vlined]{algorithm2e}

\usepackage{verbatim,tikz}
\usepackage{caption,subcaption}   

\DeclareMathOperator*{\argmin}{argmin}

\graphicspath{{./figures/}}  

\newcommand{\ignore}[1]{}
\newcommand{\R}{\mathbb{R}}

\newcommand{\proxop}{\mathop{\mathrm{prox}}\nolimits}
\newcommand{\prox}{\proxop_{\alpha R}}

\newcommand{\eqdef}{\overset{\text{def}}{=}} 
\newcommand{\diag}{{\rm Diag}} 
\newcommand{\trace}{{\rm Trace}} 
 
\newcommand{\TD}{{\rm \eta}}

\newcommand{\cD}{{\cal D}}

\newcommand{\cN}{{\cal N}}
\newcommand{\cO}{{\cal O}}

\newcommand{\mA}{{\bf A}}
\newcommand{\mB}{{\bf B}}
\newcommand{\mC}{{\bf C}}
\newcommand{\mD}{{\bf D}}

\newcommand{\mG}{{\bf G}}
\newcommand{\mH}{{\bf H}}
\newcommand{\mI}{{\bf I}}

\newcommand{\mM}{{\bf M}}
\newcommand{\mmM}{{\bf Q}}

\newcommand{\mP}{{\bf P}}
\newcommand{\mPdiag}{{\bf \hat{P}}}
\newcommand{\mQ}{{\bf Q}}

\newcommand{\mS}{\mathbf{S}}

\newcommand{\mU}{{\bf U}}

\newcommand{\mZ}{{\bf Z}}

\newcommand{\mVdiag}{{\bf \hat{V}}}
\newcommand{\Probmat}{{\mP}}

\newcommand{\Lgen}{{\Phi}}
\newcommand{\Lnacc}{{\Psi}}
\newcommand{\Lacc}{{\Upsilon}}

\newcommand{\E}[1]{\mathbb{E}\left[#1\right] } 
\newcommand{\ED}[1]{\mathbb{E}_{\cD}\left[#1\right] } 
\newcommand{\Prob}[1]{\mathbb{P}\left(#1\right) }

\providecommand{\Range}[1]{\mbox{Range}\left( #1\right)}

\usepackage[colorinlistoftodos,bordercolor=orange,backgroundcolor=orange!20,linecolor=orange,textsize=scriptsize]{todonotes}

\usepackage{hyperref}

\SetCommentSty{mycommfont}

\usepackage{booktabs}       

\usepackage{fullpage}

\title{SEGA: Variance Reduction via Gradient Sketching\footnote{Accepted to NIPS 2018.}}

\author{Filip Hanzely\thanks{King Abdullah University of Science and Technology, Kingdom of Saudi Arabia}   \qquad  Konstantin Mishchenko\thanks{King Abdullah University of Science and Technology, Kingdom of Saudi Arabia}  \qquad Peter Richt\'{a}rik\thanks{King Abdullah University of Science and Technology, Kingdom of Saudi Arabia --- School of Mathematics, University of Edinburgh, United Kingdom ---  Moscow Institute of Physics and Technology, Russia}}

\graphicspath{{experiments/images/}}

\begin{document}
\maketitle

\begin{abstract} 
We propose a randomized first order optimization method---\texttt{SEGA} (SkEtched GrAdient)---which progressively throughout its iterations builds a variance-reduced estimate of the gradient from random linear measurements (sketches) of the gradient obtained from an oracle. In each iteration, \texttt{SEGA} updates the current estimate of the gradient through a sketch-and-project operation using the information provided by the latest sketch, and this is subsequently used to compute an unbiased estimate of the true gradient through a random relaxation procedure. This unbiased estimate is then used to perform a gradient step. Unlike standard subspace descent methods, such as coordinate descent, \texttt{SEGA} can be used for optimization problems with a {\em non-separable} proximal term. We provide a general convergence analysis and prove linear convergence for strongly convex objectives. In the special case of coordinate sketches, \texttt{SEGA} can be enhanced with various techniques such as {\em importance sampling}, {\em minibatching} and {\em acceleration}, and its rate is up to a small constant factor identical to the best-known rate of coordinate descent. 
\end{abstract}

\newpage
{
\tableofcontents
}

\newpage
\section{Introduction}

Consider the  optimization problem
\begin{equation}\label{eq:main} \min_{x\in \R^n} F(x)\eqdef f(x) + R(x),\end{equation}
where $f:\R^n\to \R$ is smooth and $\mu$--strongly convex, and $R:\R^n\to \R\cup \{+\infty\}$ is a closed convex regularizer. In some applications, $R$ is either the indicator function of a convex set or a sparsity-inducing non-smooth penalty such as group $\ell_1$-norm. We assume that, as in these two examples, the \emph{proximal operator} of $R$, defined as
\begin{eqnarray*}
    \prox(x) \eqdef \argmin_{y\in \R^n} \left\{R(y) + \frac{1}{2\alpha}\|y - x\|^2_\mB  \right\},
\end{eqnarray*}
 is easily computable (e.g., in closed form). Above we use the weighted Euclidean norm $\| x\|_\mB \eqdef \langle x,x\rangle_\mB ^{1/2}$, where $\langle x,y\rangle_\mB  \eqdef \langle \mB x, y\rangle$ is a weighted inner product associated with a positive definite weight matrix $\mB$. Strong convexity of $f$ is defined  with respect to the geometry induced by this inner product and norm\footnote{$f$ is $\mu$--strongly convex if $f(x)\geq f(y)+\langle \nabla f(y),x-y\rangle_\mB +\tfrac{\mu}{2}\|x-y\|_\mB^2$ for all $x,y\in \R^n$.}.
 
 \subsection{Gradient sketching}
 
 In this paper we design proximal gradient-type methods for solving \eqref{eq:main} without assuming that the true gradient  of $f$ is available.  Instead, we assume that an oracle provides a {\em random linear transformation (i.e., a sketch) of the gradient}, which is the information available to drive the iterative process. 
 In particular, given a fixed distribution $\cD$  over matrices $\mS \in \R^{n\times b}$ ($b\geq 1$ can but does not need to be fixed), and  a query point $x\in \R^n$, our oracle provides us the random linear transformation  of the gradient given by
\begin{equation}
\label{eq:sketched_grad}
\zeta(\mS, x) \eqdef \mS^\top \nabla f(x) \in \R^{b}, \qquad \mS \sim \cD.
\end{equation}

Information of this type is available/used in a variety of scenarios. For instance, randomized coordinate descent (\texttt{CD}) methods use oracle~\eqref{eq:sketched_grad}  with $\cD$ corresponding to a distribution over standard  basis vectors. Minibatch/parallel variants of \texttt{CD} methods utilize  oracle \eqref{eq:sketched_grad} with $\cD$ corresponding to a distribution over random column submatrices of the identity matrix. If one is prepared to use difference of function values to approximate  directional derivatives, then one can apply our oracle model  to   zeroth-order optimization~\cite{conn2009introduction}.  Indeed, the directional derivative of $f$ in a random direction $\mS=s \in \R^{n\times 1}$ can be approximated by $\zeta(s, x)  \approx \tfrac{1}{\epsilon}(f(x+ \epsilon s) - f(x))$, where $\epsilon>0$ is sufficiently small. 

We now illustrate this concept using two examples.

\begin{example}[Sketches] 
 {\em (i) Coordinate sketch.} Let $\cD$ be the uniform distribution over standard unit basis vectors $e_1,e_2,\dots,e_n$ of $\R^n$. Then
$\zeta(e_i,x) = e_i^\top \nabla f(x)$, i.e., the $i^{\text{th}}$ {\em partial derivative} of $f$ at $x$. {\em (ii)  Gaussian sketch.} Let $\cD$ be the standard Gaussian distribution in $\R^n$. Then for $s\sim \cD$ we have
$\zeta(s, x) = s^\top \nabla f(x) $, i.e., the {\em directional derivative} of $f$ at $x$ in direction $s$.
\end{example}

\subsection{Related work}

In the last decade, stochastic gradient-type methods for solving problem~\eqref{eq:main} have received unprecedented attention by theoreticians and practitioners alike. Specific examples of such methods  are stochastic gradient descent (\texttt{SGD})~\cite{RobbinsMonro:1951}, variance-reduced variants of \texttt{SGD} such as \texttt{SAG}~\cite{SAG}, \texttt{SAGA}~\cite{SAGA}, \texttt{SVRG}~\cite{SVRG}, and their accelerated counterparts~\cite{lin2015universal, allen2017katyusha}. While these methods are specifically designed for objectives formulated as an expectation or a finite sum, we do not assume such a structure. Moreover, these methods  utilize a fundamentally different stochastic gradient information: they have access to an unbiased estimator of the gradient. In contrast, we do not assume that \eqref{eq:sketched_grad} is an unbiased estimator of $\nabla f(x)$. In fact,  $\zeta(\mS, x)\in \R^b$ and $\nabla f(x)\in \R^n$ do not even necessarily  belong to the same space.  Therefore, our algorithms and results should be seen as complementary to the above line of research.

While the gradient sketch $\zeta(\mS, x)$  does not immediatey lead to an unbiased estimator of the gradient,  \texttt{SEGA} uses  the information provided in the sketch  to {\em construct} an unbiased estimator of the gradient  via a {\em sketch-and-project} process. Sketch-and-project iterations were introduced in \cite{SIMAX2015} in the contex of linear feasibility problems. A dual view  uncovering a direct relationship with  stochastic subspace ascent methods was developed in \cite{SDA}. The latest and most in-depth treatment of sketch-and-project for linear feasibility is based on the idea of stochastic reformulations  \cite{ASDA}. Sketch-and-project can be combined with Polyak~\cite{SMOMENTUM, SHB-NIPS} and Nesterov momentum~\cite{ASBFGS}, extended to convex~ feasibility problems \cite{SPM}, matrix inversion \cite{inverse, PSEUDOINVERSE, ASBFGS}, and empirical risk minimization \cite{SBFGS, gower2018stochastic}. Connections to gossip algorithms for average consensus  were made in \cite{NEW-PERSPECTIVE, agossip}.

The line of work most closely related to our setup is that on randomized coordinate/subspace descent methods~\cite{Nesterov:2010RCDM, SDA}. Indeed, the information available to these methods is compatible with our oracle for specific distributions $\cD$. However, the main disadvantage of these methods is that they are not able to handle non-separable regularizers $R$. In contrast, the algorithm we propose---\texttt{SEGA}---works for any regularizer $R$.    In particular,  \texttt{SEGA} can handle non-separable constraints even with coordinate sketches, which is out of range of current coordinate descent methods. Hence, our work could be understood as extending the reach of coordinate and subspace descent methods from separable to arbitrary regularizers, which allows for a plethora of new applications.  Our method is able to work with an arbitrary regularizer due to its ability to {\em build an unbiased variance-reduced estimate of the gradient} of $f$ throughout the iterative process from the random linear measurements thereof provided by the oracle.  Moreover, and unlike coordinate descent,  \texttt{SEGA} allows for general sketches from essentially any distribution $\cD$. 

Another stream of work on designing gradient-type methods without assuming perfect access to the gradient is represented by the {\em inexact gradient descent} methods~\cite{d2008smooth,Devolder:2011inexact,schmidt2011convergence}. However, these methods deal with deterministic estimates of the gradient and are not based on linear transformations of the gradient. Therefore, this second line of research is also significantly different from what we do here. 

\subsection{Outline}

We describe \texttt{SEGA} in Section~\ref{sec:SEGA}. Convergence results for general sketches are described in Section~\ref{sec:analysis}. Refined results for coordinate sketches are presented in Section~\ref{sec:CD}, where we also describe and analyze an accelerated variant of \texttt{SEGA}. Experimental results can be found in Section~\ref{sec:experiments}. We also include here experiments with a {\em subspace} variant of \texttt{SEGA}, which is described and analyzed in Appendix~\ref{sec:subSEGA}. Conclusions are drawn and potential extensions outlined in Section~\ref{sec:conclusion}. A simplified analysis of \texttt{SEGA} in the case of coordinate sketches and for $R\equiv 0$ is developed in Appendix~\ref{sec:simple_SEGA} (under standard assumptions as in the main paper) and~\ref{sec:analysis-samenorm} (under alternative assumptions). Extra experiments for additional insights are included in Appendix~\ref{sec:extra_exp}.

\subsection{Notation}

We introduce notation when and where needed. For convenience, we provide a table of frequently used notation in Appendix~\ref{sec:notation}.



\section{The \texttt{SEGA} Algorithm} \label{sec:SEGA}

In this section we introduce a learning process for estimating the gradient from the sketched information provided by \eqref{eq:sketched_grad}; this will be used as a subroutine of \texttt{SEGA}.

Let $x^k$ be the current iterate, and let $h^k$ be the current estimate of the gradient of $f$. We then query the oracle, and receive new gradient information in the form of the sketched gradient \eqref{eq:sketched_grad}. At this point, we would like to update $h^k$ based on this new information. We do this using a {\em sketch-and-project} process~\cite{SIMAX2015, SDA, ASDA}: we set $h^{k+1}$ to be the closest vector to $h^k$ (in a certain Euclidean norm) satisfying~\eqref{eq:sketched_grad}:
\begin{eqnarray}
h^{k+1} &=& \arg \min_{h\in \R^{n}} \| h -  h^k\|_{\mB}^2 \notag \\
&& \text{subject to} \quad \mS_{k}^\top h =  \mS_{k}^\top \nabla f(x^k). \label{eq:sketch-n-project}
\end{eqnarray}

The closed-form solution of \eqref{eq:sketch-n-project} is
\begin{equation} h^{k+1} = h^k - \mB^{-1} \mZ_{k} (h^k - \nabla f(x^k)) = (\mI-\mB^{-1}\mZ_{k} ) h^k + \mB^{-1}\mZ_{k} \nabla f(x^k),\label{eq:h^{k+1}}
\end{equation}
where $\mZ_{k} \eqdef  \mS_{k} \left(\mS_{k}^\top \mB^{-1} \mS_{k}\right)^\dagger\mS_{k}^\top$. Notice that $h^{k+1}$ is a \emph{biased} estimator of $\nabla f(x^k)$. In order to obtain an unbiased gradient estimator, we introduce a random variable\footnote{Such a random variable may not exist. Some sufficient conditions are provided later.} $\theta_k=\theta(\mS_{k})$ for which

\begin{equation} \label{eq:unbiased} \ED{\theta_k  \mZ_{k}} = \mB.
\end{equation}
If $\theta_k$ satisfies~\eqref{eq:unbiased}, it is straightforward to see that the random vector
\begin{equation} \label{eq:g^k} g^k \eqdef (1-\theta_k) h^k + \theta_k h^{k+1} \overset{\eqref{eq:h^{k+1}}}{=} h^k + \theta_k \mB^{-1}\mZ_{k} (\nabla f(x^k) - h^k)
\end{equation}
is an {\em unbiased estimator} of the gradient:
\begin{eqnarray}
\ED{g^k} &\overset{\eqref{eq:unbiased} +\eqref{eq:g^k}}{=}&  \nabla f(x^k). \label{eq:unbiased_estimator}
\end{eqnarray}

Finally, we use $g^k$ instead of the true gradient, and perform a proximal step with respect to $R$. This leads to a new randomized optimization method, which we call {\em SkEtched GrAdient Method (\texttt{SEGA})}. The method is formally described in Algorithm~\ref{alg:gs}. We stress again that the method does not need the access to the full gradient. 

\begin{figure}
\begin{minipage}{0.60\textwidth}
\centering
\begin{algorithm}[H]\label{alg:gs}
\caption{\texttt{SEGA}: SkEtched GrAdient Method}
    \SetKwInOut{Init}{Initialize}
        \SetKwInOut{Output}{Output}
    \nl\Init {$x^0, h^0\in \R^n$;  $\mB\succ 0$; distribution $\cD$; stepsize $\alpha>0$}
    \nl \For{$k=1,2,\dots$}
    {
    \nl{Sample $\mS_{k} \sim \cD$} \\
    \nl{$g^{k} = h^k + \theta_k  \mB^{-1}\mZ_{k} (\nabla f(x^k) - h^k)$} \label{eq:g_update}\\
 \nl{$x^{k+1} = \prox(x^k - \alpha g^k) $}  \label{eq:x_update} \\
   \nl{$h^{k+1} = h^k + \mB^{-1}\mZ_{k} (\nabla f(x^k) - h^k) $}\label{eq:h_update}\\
   }
\end{algorithm}
\end{minipage}
\hskip 0.4cm
\begin{minipage}{0.40\textwidth}
\centering
\includegraphics[width = 0.9\textwidth ]{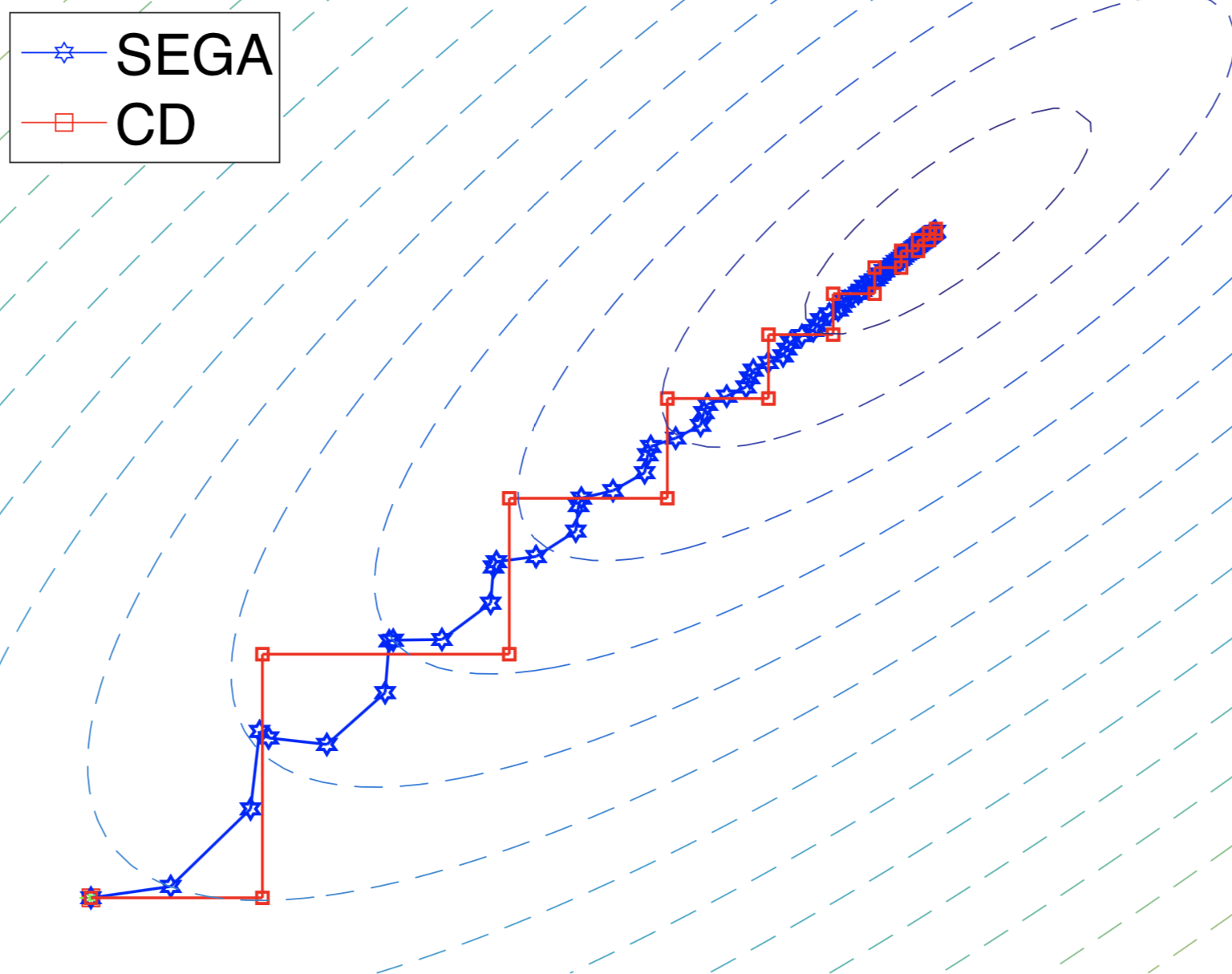}
\caption{Iterates of \texttt{SEGA} and \texttt{CD}}\label{fig:trajectory}
\end{minipage}
\end{figure}

\subsection{\texttt{SEGA} as a variance-reduced method}

As we shall show, both $h^k$ and $g^k$ are becoming better at approximating $\nabla f(x^k)$ as the iterates $x^k$ approach the optimum. Hence, the variance of $g^k$ as an estimator of the gradient tends to zero, which means that  \texttt{SEGA} is a \emph{variance-reduced} algorithm. The structure of \texttt{SEGA} is inspired by the \texttt{JackSketch} algorithm  introduced in~\cite{gower2018stochastic}. However,  as \texttt{JackSketch} is aimed at solving a finite-sum optimization problem with many components, it does not make much sense to apply it to \eqref{eq:main}. Indeed, when applied to \eqref{eq:main} (with $R=0$, since  \texttt{JackSketch} was analyzed for smooth  optimization only), \texttt{JackSketch}  reduces to gradient descent. While \texttt{JackSketch} performs {\em Jacobian} sketching (i.e., multiplying the Jacobian by a random matrix from the right, effectively sampling a subset of the gradients forming the finite sum), \texttt{SEGA} multiplies the Jacobian by a random matrix from the left. In doing so, \texttt{SEGA} becomes oblivious to the finite-sum structure and transforms into the gradient sketching mechanism described in~\eqref{eq:sketched_grad}.



\subsection{\texttt{SEGA} versus coordinate descent}

We now illustrate the above general setup on the simple example when $\cD$ corresponds to a distribution over standard unit basis vectors in $\R^n$. 

\begin{example}\label{ex:coord_setup} Let $\mB=\diag(b_1,\dots,b_n)\succ 0$ and let $\cD$ be defined as follows. We choose $\mS_{k} = e_{i}$ with probability $p_i>0$, where  $e_1,e_2,\dots, e_n$ are the unit basis vectors in $\R^n$. Then
 \begin{equation} \label{eq:988fgf} h^{k+1} \overset{\eqref{eq:h^{k+1}}}{=} h^k + e_{i}^\top (\nabla f(x^k) - h^k) e_{i},\end{equation}
which can equivalently be written as $h^{k+1}_i = e_{i}^\top \nabla f(x^k)$ and $h^{k+1}_j = h^k_j$ for $j\neq i$. Note that $h^{k+1}$ does not depend on $\mB$. If we choose $\theta_k=\theta(\mS_{k}) = 1/p_i$, then 
\[\ED{\theta_k \mZ_{k}} = \sum_{i=1}^n p_i \frac{1}{p_i} e_i (e_i^\top \mB^{-1} e_i)^{-1} e_i^\top = \sum_{i=1}^n \frac{e_i e_i^\top}{1/b_i} = \mB\]
which means that $\theta_k$ is a bias-correcting random variable. We then get
\begin{equation} \label{eq:8h0h09ffs}g^k \overset{\eqref{eq:g^k} }{=}   h^k + \frac{1}{p_{i}}  e_{i}^\top (\nabla f(x^k) - h^k) e_{i} . \end{equation}
\end{example}

In the setup of Example~\ref{ex:coord_setup}, both \texttt{SEGA} and \texttt{CD}  obtain  new gradient information in the form of a random partial derivative of $f$. However, the two methods process this information differently, and perform a different update: 
\begin{itemize}
\item[(i)]  While \texttt{SEGA} allows for arbitrary proximal term, \texttt{CD} allows for separable proximal term only~\cite{ProxSDCA,lin2014accelerated,APPROX}.  
\item[(ii)] While \texttt{SEGA} updates all coordinates in every iteration,  \texttt{CD} updates  a single coordinate only. 
\item[(iii)] If we force $h^k=0$ in  \texttt{SEGA} and use coordinate sketches, the method transforms  into \texttt{CD}.
\end{itemize}

 Based on the above observations, we conclude that \texttt{SEGA} can be applied in more general settings for the price of potentially more expensive iterations\footnote{Forming vector $g$ and computing the prox.}.  For intuition-building illustration of how \texttt{SEGA} works, Figure~\ref{fig:trajectory} shows  the evolution of iterates of both \texttt{SEGA} and \texttt{CD} applied to minimizing a simple quadratic function in 2 dimensions. For more figures of this type, including the composite case where \texttt{CD} does not work, see Appendix~\ref{sec:evolution_extra}.

In Section~\ref{sec:CD} we show that \texttt{SEGA} enjoys the same theoretical iteration complexity rates as \texttt{CD}, up to a small constant factor. This remains true when comparing state-of-the-art variants of \texttt{CD} utilizing importance-sampling, parallelism/mini-batching and acceleration with the appropriate corresponding  variants of \texttt{SEGA}.

\begin{remark}
Nontrivial sketches $\mS$ and metric $\mB$ might, in some applications, bring a substantial speedup against the baseline choices mentioned in Example~\ref{ex:coord_setup}. Appendix~\ref{sec:subSEGA} provides one setting where this can happen: there are problems where the gradient of $f$ always lies in a particular $d$-dimensional subspace of $\R^n$. In such a case, suitable choice of $\mS$ and $\mB$ leads to $\cO\left(\tfrac{n}{d}\right)$--times faster convergence compared to the setup of Example~\ref{ex:coord_setup}. In Section~\ref{sec:exp_aggressive} we numerically demonstrate this claim. 
\end{remark}

\section{Convergence of \texttt{SEGA} for General Sketches \label{sec:analysis}}

In this section we state a linear convergence result for \texttt{SEGA} (Algorithm~\ref{alg:gs}) for general sketch distributions $\cD$ under  smoothness and strong convexity assumptions. 

\subsection{Smoothness assumptions}

We will use the following general version of smoothness. 
\begin{assumption}[$\mQ$-smoothness]  \label{ass:M_smooth_inv} Function $f$ is $\mQ$-smooth with respect to $\mB$, where $\mQ\succ 0$ and $\mB\succ 0$. That is,
for all $x,y$, the following inequality is satisfied:
    \begin{align}\label{eq:M_smooth_inv}
         f(x) - f(y) - \langle \nabla f(y),x - y \rangle_{\mB } \ge \frac{1}{2}\|\nabla f(x) - \nabla f(y)\|_{\mmM}^2,
    \end{align}
\end{assumption}
Assumption~\ref{ass:M_smooth_inv} is not standard in the literature. However, as Lemma~\ref{lem:relate} states, for $\mB=\mI$ and $\mQ=\mM^{-1}$, Assumption~\ref{ass:M_smooth_inv} is equivalent to $\mM$-smoothness (see Assumption~\ref{ass:M_smooth}), which is a common assumption in modern analysis of \texttt{CD} methods. Hence, our assumption is more general than the commonly used assumption.

\begin{assumption}[$\mM$-smoothness]\label{ass:M_smooth} Function $f$ is $\mM$-smooth for some matrix $\mM\succ 0$. That is, 
for all $x,y$, the following inequality is satisfied:
\begin{equation}\label{eq:M_smooth}
f(x)\leq f(y)+\langle \nabla f(y),x-y \rangle +\frac12 \| x-y\|^2_{\mM}.
\end{equation}
\end{assumption}

Assumption~\ref{ass:M_smooth} is fairly standard in the \texttt{CD} literature. It appears naturally in various application such as empirical risk minimization with linear predictors and is a baseline in the development of minibatch \texttt{CD} methods~\cite{NSync, ESO, ALPHA, SDNA}. We will adopt this notion in Section~\ref{sec:CD}, when comparing \texttt{SEGA} to coordinate descent. Until then, let us consider the more general Assumption~\ref{ass:M_smooth_inv}. 

\subsection{Main result}

We are now ready to present one of the key theorems of the paper, which states that the iterates of  \texttt{SEGA} converge linearly to the optimal solution.  

\begin{theorem}\label{thm:main}
    Assume that $f$ is $\mmM$--smooth with respect to $\mB$, and $\mu$--strongly convex. Choose stepsize $\alpha>0$ and Lyapunov parameter $\sigma>0$ so that 
\begin{equation}
        \alpha\left(2(\mC - \mB) +\sigma \mu \mB\right) \le \sigma\ED{\mZ},\qquad \alpha \mC \le \frac{1}{2}\left(\mmM - \sigma \ED{\mZ}\right), \label{eq:general_bound_on_stepsize}        
    \end{equation} where $\mC\eqdef \ED{\theta_k^2 \mZ_k}$. Fix $x^0,h^0\in {\rm dom} (F)$ and let $x^k,h^k$ be the
 random iterates produced by  \texttt{SEGA}.     Then
\[
        \E{\Lgen^{k}} \le (1 - \alpha\mu)^k \Lgen^0,
\]
where  $\Lgen^k \eqdef \|x^k - x^*\|^2_{\mB} + \sigma \alpha \|h^k - \nabla f(x^*)\|^2_{\mB}$ is a Lyapunov function and $x^*$ is the solution of~\eqref{eq:main}.
\end{theorem}

Note that the convergence of the Lyapunov function $\Lgen^k$ implies both $x^k \rightarrow x^*$ and $h^k \rightarrow \nabla f(x^*)$. The latter means that \texttt{SEGA} is {\em variance reduced}, in contrast to \texttt{CD} in the proximal setup with  non-separable $R$, which does not converge to the solution.

To clarify on the assumptions, let us mention that if $\sigma$ is small enough so that $\mmM - \sigma \ED{\mZ}\succ 0$, one can always choose stepsize $\alpha$ satisfying
\begin{align}\label{eq:alfa_0}
    \alpha \leq \min\left\{
\frac {\lambda_{\text{min}}(\ED{\mZ})}{\lambda_{\max} (2\sigma^{-1}(\mC - \mB) + \mu\mB)}, \frac{\lambda_{\min}(\mmM - \sigma \ED{\mZ})}{2\lambda_{\max}(\mC)} \right\}
\end{align}
and inequalities~\eqref{eq:general_bound_on_stepsize} will hold. Therefore, we get the next corollary. 
\begin{corollary}\label{cor:general}
    If $\sigma < \frac{\lambda_{\min}(\mmM)}{\lambda_{\max}(\ED{\mZ})}$, $\alpha$ satisfies~\eqref{eq:alfa_0} and $k\ge\frac{1}{\alpha \mu}\log \frac{\Lgen^0}{\epsilon}$, then $\E{\|x^k - x^*\|^2_{\mB}} \le \epsilon$.
\end{corollary}
As Theorem~\ref{thm:main} is rather general, we also provide a simplified version thereof, complete with a simplified analysis (Theorem~\ref{thm:simple} in Appendix~\ref{sec:simple_SEGA}). In the simplified version we remove the proximal setting (i.e., we set $R=0$), assume $L$--smoothness\footnote{The standard $L$--smoothness assumption is a special case of $\mM$--smoothness for $\mM =L \mI$, and hence is less general than both $\mM$--smoothness and $\mmM$--smoothness with respect to $\mB$.}, and only consider coordinate sketches with uniform probabilities. The result is provided as Corollary~\ref{cor:simple}.
\begin{corollary} \label{cor:simple} Let $\mB=\mI$ and choose $\cD$ to be the uniform distribution over unit basis vectors in $\R^n$. If the stepsize satisfies 
\[0<\alpha \leq \min\left\{ \frac{1-\frac{L\sigma}{n}}{2Ln}, \frac{1}{n\left(\mu + \tfrac{2(n-1)}{\sigma}\right)} \right\},\]
then
$\ED{\Lgen^{k+1}} \leq (1-\alpha \mu) \Lgen^{k}$, therefore the iteration complexity is $\tilde{\cO}(nL/\mu)$.
\end{corollary}

\begin{remark}\label{rem:aggressive}
In the fully general setting, one might choose $\alpha$ to be bigger than bound~\eqref{eq:alfa_0}, which depends on eigen properties of matrices $\ED{\mZ}, \mC, \mQ, \mB$, leading to a better overall complexity according to Corollary~\ref{cor:general}. However, in the simple case with $\mB=\mI$, $\mmM=\mI$ and $\mS_k = e_{i_k}$ with uniform probabilities, bound~\eqref{eq:alfa_0} is tight. 
\end{remark}

\section{Convergence of \texttt{SEGA} for Coordinate Sketches\label{sec:CD}}

In this section we compare \texttt{SEGA} with  coordinate descent. We demonstrate that, specialized to a particular  choice of the distribution $\cD$ (where $\mS$ is a random column submatrix of the identity matrix), which makes \texttt{SEGA} use the same random gradient information as that used in modern state-of-the-art randomized \texttt{CD} methods,  \texttt{SEGA} attains, up to a small constant factor, the same convergence rate as \texttt{CD} methods.

Firstly, in Section~\ref{sec:nonacc} we develop \texttt{SEGA} with arbitrary ``coordinate sketches'' (Theorem~\ref{t:imp_nacc}). Then, in Section~\ref{s:acc} we develop an {\em accelerated  variant of \texttt{SEGA}} in a very general setup known as {\em arbitrary sampling} (see Theorem~\ref{t:imp_acc}) \cite{NSync, quartz, alfa, ESO, scp}. Lastly, Corollary~\ref{cor:imp_nacc} and Corollary~\ref{cor:acc_imp} provide us with \emph{importance sampling} for both nonaccelerated and accelerated method, which matches up to a constant factor cutting-edge coordinate descent rates \cite{NSync,allen2016even} under the same oracle and assumptions\footnote{There was recently introduced a notion of importance minibatch sampling for coordinate descent~\cite{AccMbCd}. We state, without a proof, that \texttt{SEGA} with block coordinate sketches allows for the same importance sampling as developed in the mentioned paper. 
}. Table~\ref{tab:CDcmp} summarizes the results of this section. We provide a dedicated analysis for the methods from this section in Appendix~\ref{sec:proofs_CD}.


\begin{table}[t]
\centering
\begin{tabular}{|c|c|c|c|c|}
\hline
& \texttt{CD}  &   \texttt{SEGA}  \\
\hline
\hline
\begin{tabular}{c}Nonaccelerated method\\ importance sampling, $b=1$ \end{tabular} & 
$\frac{\trace(\mM)}{ \mu} \log \frac{1}{\epsilon}$ \cite{Nesterov:2010RCDM}
& 
$8.55\cdot \frac{\trace(\mM)}{ \mu} \log \frac{1}{\epsilon}$ 
\\
\hline
\begin{tabular}{c}Nonaccelerated method\\ arbitrary sampling  \end{tabular} & 
$ \left(\max_i \frac{v_i}{p_i \mu}\right) \log \frac{1}{\epsilon}$ \hfill \cite{NSync} 
& 
$8.55 \cdot \left(\max_i \frac{v_i}{p_i \mu}\right) \log \frac{1}{\epsilon}$
\\
\hline
\begin{tabular}{c}Accelerated method\\  importance sampling, $b=1$  \end{tabular} & 
$1.62\cdot\frac{\sum_i \sqrt{\mM_{ii}}}{\sqrt{\mu}}  \log \frac{1}{\epsilon} $ \cite{allen2016even} 
&
$9.8 \cdot \frac{\sum_i \sqrt{\mM_{ii}}}{\sqrt{\mu}}  \log \frac{1}{\epsilon}$ 
\\
\hline
\begin{tabular}{c}Accelerated method\\ arbitrary sampling  \end{tabular} & 
$1.62 \cdot\sqrt{ \max_i \frac{v_i}{p_i^2 \mu} }   \log \frac{1}{\epsilon} $  \cite{AccMbCd}
&
$9.8 \cdot \sqrt{ \max_i \frac{v_i}{p_i^2 \mu} }   \log \frac{1}{\epsilon}$
\\
\hline
\end{tabular}
\caption{Complexity results for coordinate descent (\texttt{CD}) and our sketched gradient method (\texttt{SEGA}), specialized to coordinate sketching, for $\mM$--smooth and $\mu$--strongly convex functions. }
\label{tab:CDcmp}
\end{table}

We now describe the setup and technical assumptions for this section. In order to facilitate a direct comparison with \texttt{CD} (which 
does not work with non-separable regularizer $R$), for simplicity we consider problem~\eqref{eq:main} in the simplified setting with $R\equiv 0$.  Further, function $f$ is assumed to be $\mM$--smooth (Assumption~\ref{ass:M_smooth}) and $\mu$--strongly convex. 


\subsection{Defining $\cD$: samplings}

In order to draw a direct comparison with general variants of \texttt{CD} methods (i.e., with those analyzed in the {\em arbitrary sampling} paradigm), we consider sketches in~\eqref{eq:sketch-n-project} that are column submatrices of the identity matrix:
$\mS  = \mI_S,$
where $S$ is a random subset (aka {\em sampling}) of $[n]\eqdef \{1,2,\dots,n\}$.  Note that the columns of $\mI_S$  are the standard basis vectors $e_i$ for $i\in S$ and hence \[\Range{\mS} = \Range{e_i\;:\; i\in S}.\] So, distribution $\cD$ from which we draw matrices  is uniquely determined  by the distribution of sampling $S$. Given a sampling $S$, define $p = (p_1,\dots,p_n)\in \R^n$ to be the vector satisfying $p_i=\Prob{e_i\in \Range{\mS}} = \Prob{i\in S}$, and $\Probmat$ to be the matrix for which 
\[\Probmat_{ij}=\Prob{\{i,j\}\subseteq S}.\]

Note that $p$ and $\Probmat$ are the {\em probability vector} and {\em probability matrix} of sampling $S$, respectively~\cite{ESO}. We assume throughout the paper that $S$ is proper, i.e., we assume that $p_i>0$ for all $i$. State-of-the-art minibatch \texttt{CD} methods (including the ones we compare against~\cite{NSync, AccMbCd}) utilize large stepsizes related to the so-called ESO \emph{Expected Separable Overapproximation (ESO)}~\cite{ESO} parameters $v=(v_1,\dots,v_n)$. ESO parameters play a key role in \texttt{SEGA} as well, and are defined next. 
\begin{assumption}[ESO]\label{ass_ESO}
There exists a vector $v$ satisfying the following inequality
\begin{equation}\label{eq:ESO}
\Probmat \circ \mM \preceq \diag(p) \diag(v) ,
\end{equation}
where $\circ$ denotes the Hadamard (i.e., element-wise) product of matrices.
\end{assumption}
In case of single coordinate sketches, parameters $v$ are equal to coordinate-wise smoothness constants of $f$. An extensive study on how to choose them in general was performed in~\cite{ESO}. 
For notational brevity, let us set $\mPdiag \eqdef\diag(p)$ and $\mVdiag\eqdef \diag(v)$ throughout this section.

%

\subsection{Non-accelerated method \label{sec:nonacc}}

We now state the convergence rate of (non-accelerated) \texttt{SEGA} for coordinate sketches with {\em arbitrary sampling} of subsets of coordinates. The corresponding \texttt{CD} method was developed in~\cite{NSync}.

\begin{theorem}\label{t:imp_nacc}
Assume that $f$ is $\mM$--smooth and $\mu$--strongly convex. Denote $\Lnacc^{k} \eqdef f(x^{k})-f(x^*)+ \sigma \|h^{k} \|^2_{\mPdiag^{-1}}$. Choose $\alpha, \sigma>0$ such that 
\begin{equation}\label{eq:assumption}
\sigma \mI-\alpha^2(\mVdiag\mPdiag^{-1}-\mM) \succeq \gamma \mu \sigma \mPdiag^{-1},
\end{equation}
where $\gamma \eqdef \alpha - \alpha^2\max_{i}\{\tfrac{v_{i}}{p_{i}}\}-\sigma$. Then the iterates of \texttt{SEGA} satisfy
$
\E{\Lnacc^{k}}\leq (1-\gamma\mu )^k \Lnacc^0.
$
\end{theorem}

We now give an importance sampling result for a coordinate version of  \texttt{SEGA}. We recover, up to a constant factor, the same convergence rate as standard \texttt{CD}~\cite{Nesterov:2010RCDM}. The probabilities we chose are optimal in our analysis and are proportional to the diagonal elements of matrix $\mM$. 

\begin{corollary}\label{cor:imp_nacc}
Assume that $f$ is $\mM$--smooth and $\mu$--strongly convex. Suppose that $\cD$ is such that at each iteration standard unit basis vector $e_i$ is sampled with probability $p_i\propto \mM_{ii}$.
If we choose $ \alpha=\frac{0.232}{\trace(\mM)}, \sigma=\frac{0.061}{ \trace(\mM)}$, then
$
\E{\Lnacc^{k}}\leq \left(1-\frac{0.117 \mu}{\trace(\mM)} \right)^k \Lnacc^0.
$
\end{corollary}

The iteration complexities provided in Theorem~\ref{t:imp_nacc} and Corollary~\ref{cor:imp_nacc} are summarized in Table~\ref{tab:CDcmp}. We also state that $\sigma, \alpha$ can be chosen so that~\eqref{eq:assumption} holds, and the rate from Theorem~\ref{t:imp_nacc} coincides with the rate from Table~\ref{tab:CDcmp}.

 \begin{remark}
 Theorem~\ref{t:imp_nacc} and Corollary~\ref{cor:imp_nacc} hold even under a non-convex relaxation of strong convexity -- Polyak-\L{}ojasiewicz inequality: $\mu (f(x)- f(x^*))\leq \tfrac{1}{2}\|\nabla f(x) \|_2^2$. Therefore, \texttt{SEGA} also converges for a certain class of non-convex problems. For an overview on different relaxations of strong convexity, see~\cite{karimi2016linear}.
 \end{remark}

\subsection{Accelerated method \label{s:acc}}
 
 In this section, we propose an accelerated (in the sense of Nesterov's method~\cite{nesterov1983method,NesterovBook}) version of \texttt{SEGA}, which we call \texttt{ASEGA}. The analogous accelerated \texttt{CD} method, in which a single coordinate is sampled in every iteration, was developed and analyzed in~\cite{allen2016even}. The general variant utilizing arbitrary sampling was developed and analyzed in~\cite{AccMbCd}.

\begin{algorithm}[H]\label{alg:acc}
\caption{\texttt{ASEGA}: Accelerated \texttt{SEGA}}
    \SetKwInOut{Init}{Initialize}
        \SetKwInOut{Output}{Output}
    \nl\Init {$x^0=y^0=z^0\in \R^n$; $h^0\in \R^n$; 	 $S$; parameters $\alpha, \beta,\tau, \mu>0$}
    \nl \For{$k=1,2,\dots$}
    {
    \nl{$x^{k}=(1-\tau)y^{k-1}+\tau z^{k-1}$} \label{eq:x_update}\\
\nl {Sample $\mS_k = \mI_{S_k}$, where $S_k\sim S$, and compute $g^{k},h^{k+1}$ according to~\eqref{eq:h^{k+1}},~\eqref{eq:g^k}}
\\
 \nl{$y^{k}=x^{k}-\alpha\mPdiag^{-1} g^k $}\label{eq:y_update}  \\
   \nl{$z^{k}=\frac{1}{1+\beta \mu}(z^k+\beta \mu x^{k}-\beta g^k) $}\label{eq:z_update}\\
   }
\end{algorithm}

The method and analysis is inspired by~\cite{allen2014linear}. Due to space limitations and technicality of the content, we state the main theorem of this section in Appendix~\ref{sec:acc_thm}. Here, we provide  Corollary~\ref{cor:acc_imp}, which shows that Algorithm~\ref{alg:acc} with single coordinate sampling enjoys, up to a constant factor, the same convergence rate as state-of-the-art accelerated coordinate descent method \texttt{NUACDM} of Allen-Zhu et al.~\cite{allen2016even}. 

 \begin{corollary}\label{cor:acc_imp}
Let the sampling be defined as follows: $S = \{i\}$ with probability $p_i \propto \sqrt{\mM_{ii}}$, for $i \in [n]$. Then there exist acceleration parameters and a Lyapunov function $\Lacc^k$ such that $f(y^k) - f(x^*) \le \Lacc^k$ and
 \[
\E{\Lacc^{k}}\leq (1-\tau)^k\Lacc^0=\left(1-\cO\left(\frac{ \sqrt{\mu}}{\sum_{i} \sqrt{\mM_{ii}}} \right) \right)^k\Lacc^0.
\]
 \end{corollary}

The iteration complexity guarantees provided by Theorem~\ref{t:imp_acc} and Corollary~\ref{cor:acc_imp} are summarized in Table~\ref{tab:CDcmp}.

 \section{Experiments} \label{sec:experiments}
  In this section we perform numerical experiments to illustrate the potential of \texttt{SEGA}. Firstly, in Section~\ref{sec:exp_pgd}, we compare it  to projected gradient descent (\texttt{PGD}) algorithm. Then in Section~\ref{sec:exp_zero}, we study the performance of zeroth-order \texttt{SEGA} (when sketched gradients are being estimated through function value evaluations) and compare it to the analogous zeroth-order method. Lastly, in Section~\ref{sec:exp_aggressive} we verify the claim from Remark~\ref{rem:aggressive} that in some applications, particular sketches and metric might lead to a significantly faster convergence. In the experiments where theory-supported stepsizes were used, we obtained them by precomputing strong convexity and smoothness measures.
  
\subsection{Comparison to projected gradient\label{sec:exp_pgd}}

In this experiment, we illustrate the potential superiority of our method to \texttt{PGD}. We consider the $\ell_2$ ball constrained problem ($R$ is the indicator function of the unit ball) with the oracle providing the sketched gradient in the random Gaussian direction. As we mentioned in the introduction, a method moving in the gradient direction (analogue of \texttt{CD}), will not converge due to the proximal nature of the problem. Therefore, we can only compare against the projected gradient. However, in order to obtain the full gradient, one needs to gather $n$ sketched gradients and solve a linear system to recover the gradient. To illustrate this, we choose 4 different quadratic problems, according to Table~\ref{tab:problem} in the appendix. We stress that these are synthetic problems generated for the purpose of illustrating the potential of our method against a natural baseline.  Figure~\ref{fig:pgd_comp} compares \texttt{SEGA} and \texttt{PGD} under various relative cost scenarios of solving the linear system compared to the cost of the oracle calls. The results show that \texttt{SEGA}  significantly outperforms \texttt{PGD} as soon as solving the linear system is expensive, and is as fast as \texttt{PGD} even if solving the linear system comes for free.

\begin{figure}
\centering
\begin{minipage}{0.25\textwidth}
  \centering
\includegraphics[width =  \textwidth ]{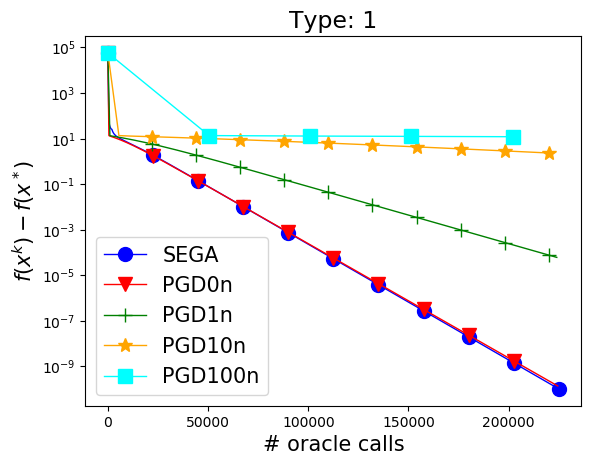}
\end{minipage}%
\begin{minipage}{0.25\textwidth}
  \centering
\includegraphics[width =  \textwidth ]{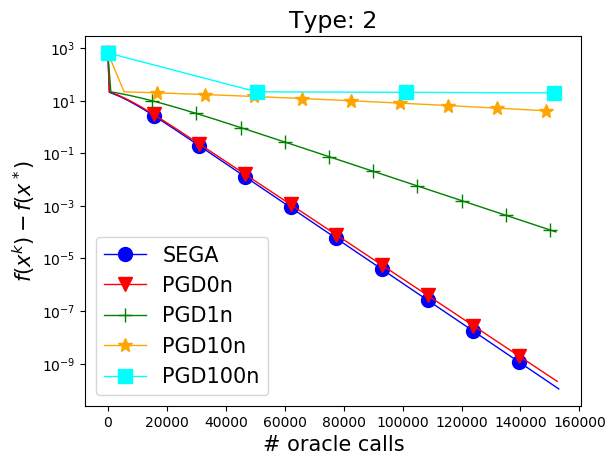}
\end{minipage}%
\begin{minipage}{0.25\textwidth}
  \centering
\includegraphics[width =  \textwidth ]{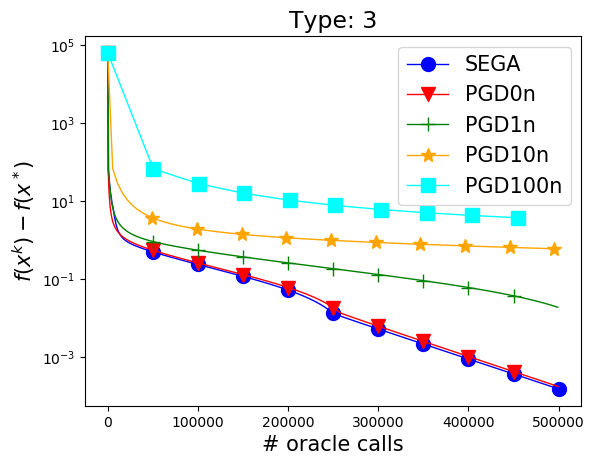}
\end{minipage}%
\begin{minipage}{0.25\textwidth}
  \centering
\includegraphics[width =  \textwidth ]{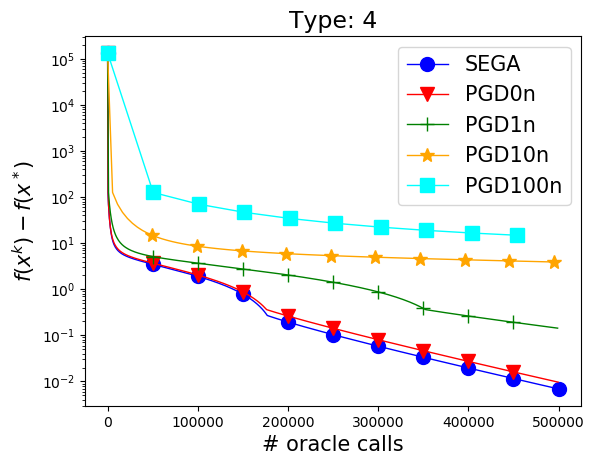}
\end{minipage}%
\caption{\footnotesize Convergence of \texttt{SEGA} and \texttt{PGD} on synthetic problems with  $n=500$. The indicator ``Xn'' in the label indicates  the setting where the cost of solving linear system is $Xn$ times higher comparing to the oracle call. Recall that a linear system is solved after each $n$ oracle calls. Stepsizes $1/\lambda_{\max}(\mM)$ and $1/(n\lambda_{\max}(\mM))$ were used for \texttt{PGD} and \texttt{SEGA}, respectively. }\label{fig:pgd_comp}
\end{figure}

\subsection{Comparison to zeroth-order optimization methods\label{sec:exp_zero}}
In this section, we compare \texttt{SEGA} to the {\em random direct search} (\texttt{RDS}) method~\cite{RDS} under a zeroth-order oracle for unconstrained optimization. For \texttt{SEGA}, we estimate the sketched gradient using finite differences. Note that \texttt{RDS} is a randomized version of the classical direct search method~\cite{hooke1961direct, kolda2003optimization, konevcny2014simple}. At iteration $k$, \texttt{RDS} moves  to $\argmin \left(f(x^k+\alpha^k s^k),f(x^k-\alpha^k s^k),f(x^k)\right)$ for a random direction $s^k\sim \cD$ and a suitable stepszie $\alpha^k$. For illustration, we choose $f$ to be a quadratic problem based on Table~\ref{tab:problem} and compare both Gaussian and coordinate directions. Figure~\ref{fig:DFO} shows that \texttt{SEGA} outperforms \texttt{RDS}. 

\begin{figure}
\centering
\begin{minipage}{0.25\textwidth}
  \centering
\includegraphics[width =  \textwidth ]{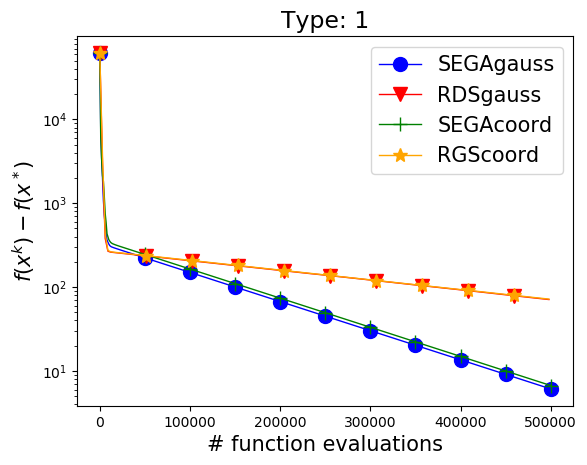}
\end{minipage}%
\begin{minipage}{0.25\textwidth}
  \centering
\includegraphics[width =  \textwidth ]{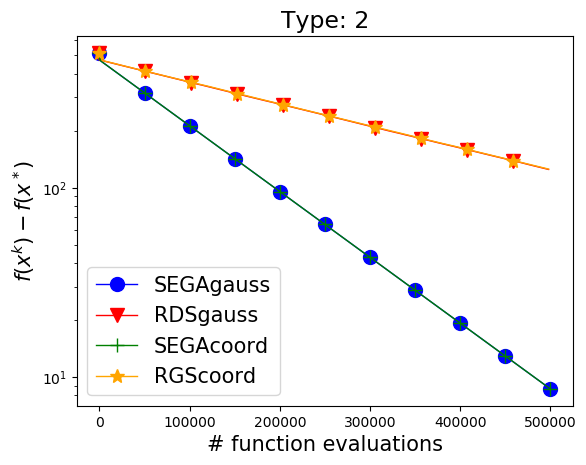}
\end{minipage}%
\begin{minipage}{0.25\textwidth}
  \centering
\includegraphics[width =  \textwidth ]{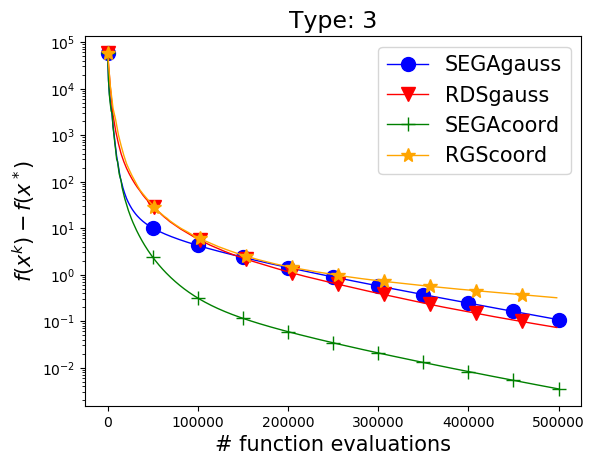}
\end{minipage}%
\begin{minipage}{0.25\textwidth}
  \centering
\includegraphics[width =  \textwidth ]{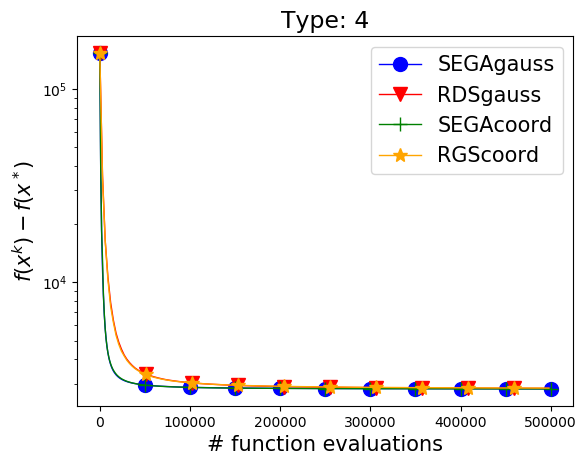}
\end{minipage}%
\caption{\footnotesize Comparison of \texttt{SEGA} and randomized direct search for various problems. Theory supported stepsizes were chosen for both methods. 500 dimensional problem. }\label{fig:DFO}
\end{figure}

\subsection{Subspace \texttt{SEGA}: a more aggressive approach\label{sec:exp_aggressive}}
As mentioned in Remark~\ref{rem:aggressive}, well designed sketches are capable of exploiting  structure of $f$ and lead to a better rate. We address this in detail Appendix~\ref{sec:subSEGA} where we develop and analyze a subspace variant of \texttt{SEGA}. 

To illustrate this phenomenon in a simple setting, we perform experiments for problem \eqref{eq:main} with $f(x)=\| \mA x -b\|^2,$ where $b\in \R^{d}$ and $\mA\in \R^{d\times n}$ has orthogonal rows, and with $R$ being the indicator function of the unit ball in $\R^n$.  That is, we solve the problem
\[\min_{\|x\|_2 \leq 1} \| \mA x -b\|^2.\]
We assume that $n\gg d$.  We compare two methods: \texttt{naiveSEGA}, which uses coordinate sketches, and \texttt{subspaceSEGA}, where sketches are chosen as rows of $\mA$. Figure~\ref{fig:aggressive} indicates that \texttt{subspaceSEGA} outperforms  \texttt{naiveSEGA} roughly by the factor $\tfrac{n}{d}$, as claimed in Appendix~\ref{sec:subSEGA}.

\begin{figure}
\centering
\begin{minipage}{0.25\textwidth}
  \centering
\includegraphics[width =  \textwidth ]{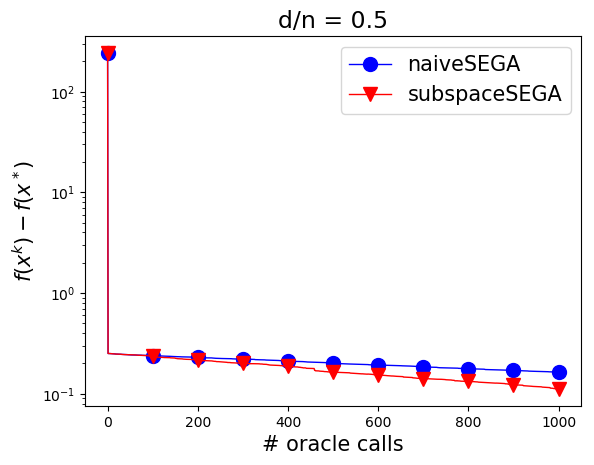}
\end{minipage}%
\begin{minipage}{0.25\textwidth}
  \centering
\includegraphics[width =  \textwidth ]{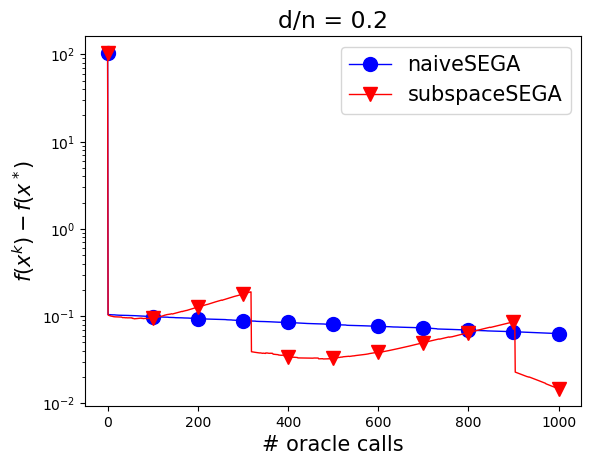}
\end{minipage}%
\begin{minipage}{0.25\textwidth}
  \centering
\includegraphics[width =  \textwidth ]{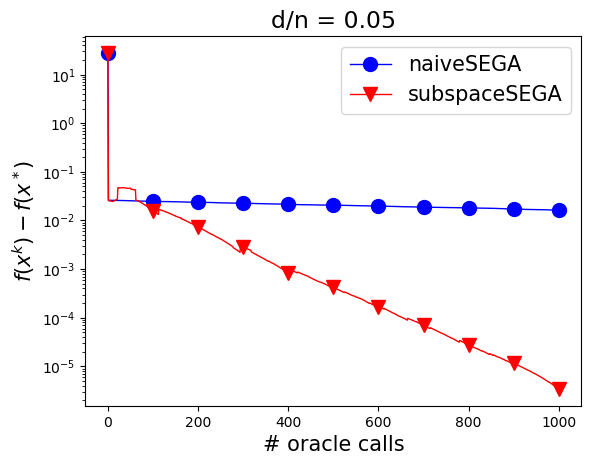}
\end{minipage}%
\begin{minipage}{0.25\textwidth}
  \centering
\includegraphics[width =  \textwidth ]{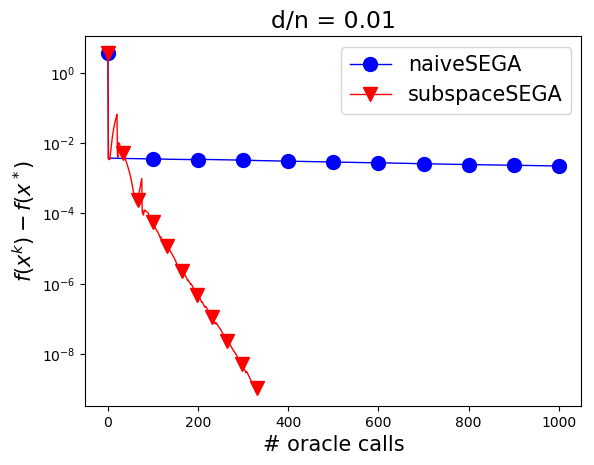}
\end{minipage}%
\caption{\footnotesize Comparison of \texttt{SEGA} with sketches from a correct subspace versus coordinate sketches \texttt{naiveSEGA}. Stepsize chosen according to theory. 1000 dimensional problem. }\label{fig:aggressive}
\end{figure}

\section{Conclusions and Extensions} \label{sec:conclusion}

\subsection{Conclusions}

We proposed \texttt{SEGA}, a  method for solving composite optimization problems under a novel stochastic linear first order  oracle. \texttt{SEGA} is variance-reduced, and this is achieved via  sketch-and-project updates of gradient estimates. We provided an analysis for smooth and strongly convex functions and general sketches, and  a refined analysis for coordinate sketches. For coordinate sketches we  also proposed an accelerated variant of \texttt{SEGA}, and  our theory matches that of state-of-the-art \texttt{CD} methods. However, in contrast to  \texttt{CD}, \texttt{SEGA} can be used for optimization problems with a {\em non-separable} proximal term. We develop a more aggressive subspace variant of the method---\texttt{subspaceSEGA}---which leads to improvements in the $n\gg d$ regime. In the Appendix we give several further results, including simplified and alternative analyses of \texttt{SEGA}  in the coordinate setup from Example~\ref{ex:coord_setup}.  Our experiments are encouraging and substantiate our theoretical predictions.

\subsection{Extensions}

We now point to several potential  extensions of our work.

\paragraph{Speeding up the general method.}  We believe that it should be possible to extend \texttt{ASEGA} to the general setup from Theorem~\ref{thm:main}. In such a case, it might be possible to design metric $\mB$ and distribution of sketches $\cD$ so as to outperform accelerated proximal gradient methods~\cite{Nesterov05:smooth, beck2009fista}. 

\paragraph{Biased gradient estimator.} Recall that \texttt{SEGA} uses unbiased gradient estimator $g^k$ for updating the iterates $x^k$ in a similar way \texttt{JacSketch}~\cite{gower2018stochastic} or \texttt{SAGA}~\cite{SAGA} do this for the stochastic finite sum optimization. Recently, a  stochastic method for finite sum  optimization using biased gradient estimators was proven to be more efficient~\cite{nguyen2017sarah}. Therefore, it might be possible to establish better properties for a biased variant of \texttt{SEGA}. To demonstrate the potential of this approach, in Appendix~\ref{sec:evolution_extra} we plot the evolution of iterates for the very simple biased method which uses $h^k$ as an update for line~\ref{eq:x_update} in Algorithm~\ref{alg:gs}.

\paragraph{Applications.} We believe that \texttt{SEGA} might work well in applications where a zeroth-order approach is inevitable, such as reinforcement learning. We therefore believe that \texttt{SEGA} might be an efficient proximal method in some reinforcement learning applications. We also believe that communication-efficient variants of \texttt{SEGA} can be used for  distributed training of machine learning models. This is because \texttt{SEGA} can be adapted to communicate sparse model updates only.

 \bibliographystyle{plain} 
 \bibliography{literature}
 
\newpage
\appendix

\part*{Appendix}

\section{Proofs for Section~\ref{sec:analysis}}
\begin{lemma}\label{lem:relate}
Suppose that $\mB=\mI$ and $f$ is twice differentiable. 
Assumption~\ref{ass:M_smooth_inv} is equivalent to Assumption~\ref{ass:M_smooth} for $\mmM=\mM^{-1}$. 
\end{lemma}
\begin{proof}
We first establish that Assumption~\ref{ass:M_smooth_inv} implies Assumption~\ref{ass:M_smooth}. 
Summing up~\eqref{eq:M_smooth_inv} for $(x,y)$ and $(y,x)$ yields
\[
 \langle  \nabla f(x) - \nabla f(y) ,x - y \rangle \geq \|\nabla f(x) - \nabla f(y)\|_{\mmM}^2. 
\]
Using Cauchy Schwartz inequality we obtain
\[
 \| x-y\|_{\mQ^{-1}} \geq \|\nabla f(x) - \nabla f(y)\|_{\mmM}. 
\]
By the mean value theorem, there is $z \in [x,y]$ such that $\nabla f(x) - \nabla f(y) = \nabla^2 f(z) (x-y)$. Thus
\[
 \| x-y\|_{\mQ^{-1}} \geq \|x-y\|_{\nabla^2 f(z) \mmM \nabla^2 f(z) }. 
\]
The above is equivalent to
\[
\left(\nabla^2 f(z)\right)^{-\frac12}\mQ^{-1}\left(\nabla^2 f(z)\right)^{-\frac12} \succeq 
\left(\nabla^2 f(z)\right)^{\frac12} \mQ\left(\nabla^2 f(z)\right)^{\frac12}
\]
Note that for any $\mM'\succ 0$ we have $\mM' \succeq \mM^{-1}$ if and only if $\mM \succeq \mI$. Thus 
\[
\left(\nabla^2 f(z)\right)^{-\frac12}\mQ^{-1}\left(\nabla^2 f(z)\right)^{-\frac12}\succeq \mI,
\]
which is equivalent to $\mQ^{-1} \succeq \nabla^2 f(z)$.
To establish the other direction, denote $\phi(y)=f(y)-\langle \nabla f(x),y\rangle$. Clearly, $x$ is minimizer of $\phi$ and therefore we have
\[
\phi(x)\leq \phi(x-\mM^{-1} \nabla f(y)) \leq \phi(y)-\frac12 \|\nabla f(y) \|^2_{\mM^{-1}}, 
\]
which is exactly~\eqref{eq:M_smooth_inv} for $\mmM=\mM^{-1}$.  
\qed
\end{proof}

\begin{lemma}
\label{lem:zkbzk}
    For $\mB\succ 0$ and $\mZ_k \eqdef  \mS_k (\mS_k^\top \mB^{-1} \mS_k)^\dagger \mS_k^\top$, then
    \begin{align}
    \label{eq:zkbzk}
        \mZ_k^\top \mB^{-1} \mZ_k = \mZ_k.
    \end{align}
\end{lemma}
\begin{proof}
	It is a property of pseudo-inverse that for any matrices $\mA, \mB$ it holds $((\mA\mB)^\dagger)^\top = (\mB^\top\mA^\top)^\dagger$, so $\mZ_k^\top = \mZ_k$. Moreover, we also know for any $\mA$ that $\mA^\dagger \mA \mA^\dagger = \mA^\dagger$ and, thus,
\[
        \mZ_k^\top \mB^{-1}\mZ_k 
        =  \mS_k (\mS_k^\top \mB^{-1} \mS_k)^\dagger \mS_k^\top \mB^{-1}  \mS_k(\mS_k^\top \mB^{-1} \mS_k)^\dagger \mS_k^\top=  \mS_k (\mS_k^\top \mB^{-1} \mS_k)^\dagger \mS_k^\top= \mZ_k.
\]\qed
\end{proof}

\subsection{Proof of Theorem~\ref{thm:main}}
We first state two lemmas which will be crucial for the analysis. They characterize key properties of the gradient learning process~\eqref{eq:h^{k+1}},~\eqref{eq:g^k} and will be used later to bound expected distances of both $h^{k+1}$ and $g^k$ from $\nabla f(x^*)$. The proofs are provided in Appendix~\ref{sec:prl1} and~\ref{sec:prl2} respectively

\begin{lemma} \label{lem:2} For all $v\in \R^n$ we have
\begin{equation}\label{eq:h_decomp}
\ED{\|h^{k+1} - v\|_{\mB}^2} =  \|h^k - v\|_{\mB - \ED{\mZ}}^2 + \|\nabla f(x^k) - v\|_{\ED{\mZ}}^2.
\end{equation}
\end{lemma}

\begin{lemma}
\label{lem:gk_general}
        Let $\mC\eqdef \ED{\theta^2\mZ}$. Then for all $v\in \R^n$ we have
\[
        \ED{\|g^k - v\|_\mB^2} \le 2\|\nabla f(x^k) - v\|^2_\mC + 2\|h^k - v\|^2_{\mC-\mB}.
\]
\end{lemma}
For notational simplicity, it will be convenient to define Bregman divergence between $x$ and $y$: \[D_f(x,y)\eqdef f(x) - f(y) - \langle \nabla f(y)),x - y \rangle_{\mB }\]
We can now proceed with the proof of Theorem~\ref{thm:main}. 
    Let us start with bounding the first term in the expression for $\Lgen^{k+1}$. From Lemma~\ref{lem:gk_general} and strong convexity it follows that
    \begin{eqnarray*}
        \ED{\|x^{k+1} - x^*\|^2_\mB}
        &=& \ED{\|\prox(x^k - \alpha g^k) - \prox(x^* - \alpha \nabla f(x^*))\|^2_\mB}\\
        &\le& \ED{\|x^k - \alpha g^k - (x^* - \alpha \nabla f(x^*))\|^2_\mB}\\ 
        &=& \|x^k - x^*\|^2_\mB - 2\alpha\ED{ (g^k - \nabla f(x^*))^\top \mB (x^k - x^*)}\\
        &&\qquad + \alpha^2\ED{\|g^k - \nabla f(x^*)\|^2_\mB}\\
        &\le & \|x^k - x^*\|^2_\mB - 2\alpha (\nabla f(x^k) - \nabla f(x^*))^\top \mB (x^k - x^*) \\
        &&\qquad + 2\alpha^2\|\nabla f(x^k) - \nabla f(x^*)\|^2_\mC + 2\alpha^2\|h^k - \nabla f(x^*)\|^2_{\mC-\mB}\\
        &\le &\|x^k - x^*\|^2_\mB - \alpha\mu\|x^k - x^*\|^2_\mB  - 2\alpha D_f(x^k, x^*) \\
        &&\qquad + 2\alpha^2\|\nabla f(x^k) - \nabla f(x^*)\|^2_\mC + 2\alpha^2\|h^k - \nabla f(x^*)\|^2_{\mC-\mB}.
    \end{eqnarray*}
    Using Assumption~\ref{ass:M_smooth_inv} we get
    \begin{eqnarray*}
        -2\alpha D_f(x^k, x^*) \le -\alpha \|\nabla f(x^k) - \nabla f(x^*)\|^2_\mmM.
    \end{eqnarray*}
    As for the second term in $\Lgen^{k+1}$, we have by Lemma~\ref{lem:2}
    \begin{eqnarray*}
        \alpha\sigma\ED{\|h^{k+1} - \nabla f(x^*)\|^2_\mB} = \alpha\sigma\|h^k - \nabla f(x^*)\|^2_{\mB - \ED{\mZ}}  + \alpha\sigma\|\nabla f(x^k) - \nabla f(x^*)\|^2_{\ED{\mZ}}
    \end{eqnarray*}
    Combining it into Lyapunov function $\Lgen^k$,
    \begin{eqnarray*}
        \Lgen^{k+1} 
        &\le & (1  - \alpha\mu)\|x^k - x^*\|^2_\mB + \alpha\sigma\|h^k - \nabla f(x^*)\|^2_{\mB - \ED{\mZ}} + 2\alpha^2\|h^k - \nabla f(x^*)\|^2_{\mC-\mB} \\
        &&  + \alpha\sigma\|\nabla f(x^k) - \nabla f(x^*)\|^2_{\ED{\mZ}} + 2\alpha^2\|\nabla f(x^k) - \nabla f(x^*)\|^2_\mC -\alpha \|\nabla f(x^k) - \nabla f(x^*)\|^2_\mmM.
    \end{eqnarray*}
    To see that this gives us the theorem's statement, consider first
    \begin{eqnarray*}
        \alpha\sigma \ED{\mZ} +2\alpha^2 \mC - \alpha\mmM 
        = 2\alpha (\alpha\mC - \tfrac{1}{2}(\mmM - \sigma \ED{\mZ}))
        \le 0,
    \end{eqnarray*}
    so we can drop norms related to $\nabla f(x^k) - \nabla f(x^*)$. Next, we have
    \begin{eqnarray*}
        \alpha\sigma (\mB - \ED{\mZ}) + 2\alpha^2(\mC - \mB) 
        &=& \alpha\left( \alpha(2(\mC - \mB) + \sigma\mu\mB) - \ED{\mZ}\right) + \sigma \alpha(1 - \alpha\mu)\mB \\
        &\le& \sigma\alpha (1 - \alpha\mu)\mB,
    \end{eqnarray*}
    which follows from our assumption on $\alpha$.\hfill \qed

\subsection{Proof of Lemma~\ref{lem:2} \label{sec:prl1}}

\begin{proof}
Keeping in mind that $\mZ_k^\top = \mZ_k$ and $(\mB^{-1})^\top = \mB^{-1}$, we first write
\begin{eqnarray*}
    \ED{\| h^{k+1}-v\|_\mB^2} 
    &\overset{\eqref{eq:988fgf}}{=}& \ED{ \left\|h^k + \mB^{-1}\mZ_k (\nabla f(x^k) - h^k) - v \right\|_\mB^2 }\\
 &=&  \ED{ \left\| \left(\mI - \mB^{-1}\mZ_k \right)(h^k - v) + \mB^{-1}\mZ_k (\nabla f(x^k)-v)  \right\|_\mB^2 }\\
 &=& \ED{ \left\| \left(\mI - \mB^{-1}\mZ_k \right)(h^k-v) \right\|_\mB^2} + \ED{\left\|\mB^{-1}\mZ_k  ( \nabla f(x^k) - v)  \right\|_\mB^2 }\\
 && \qquad + 2 (h^k-v)^\top\ED{\left(\mI - \mB^{-1}\mZ_k \right)^\top \mB \mB^{-1}\mZ_k}  (\nabla f(x^k) - v) \\
 &=& (h^k-v)^\top \ED{ \left(\mI - \mB^{-1}\mZ_k \right)^\top \mB  \left(\mI - \mB^{-1}\mZ_k \right)} (h^k-v) \\
 && \qquad + (\nabla f(x^k)-v)^\top \ED{\mZ_k\mB^{-1}\mB\mB^{-1}\mZ_k} (\nabla f(x^k) -v)\\
 && \qquad + 2 (h^k-v)^\top\ED{\mZ_k - \mZ_k\mB^{-1}\mZ_k}  (\nabla f(x^k) - v).
 \end{eqnarray*}
By Lemma \ref{lem:zkbzk} we have $\mZ_k\mB^{-1}\mZ_k=\mZ_k$, so the last term in the expression above is equal to 0. As for the other two, expanding the matrix factor in the first term leads to
\begin{eqnarray*}
     \ED{ \left(\mI - \mB^{-1}\mZ_k \right)^\top \mB  \left(\mI - \mB^{-1}\mZ_k \right)}
     &=& \ED{ \left(\mI - \mZ_k\mB^{-1} \right) \mB  \left(\mI - \mB^{-1}\mZ_k \right)}\\
     &=&  \ED{ \mB - \mZ_k\mB^{-1}\mB - \mB\mB^{-1}\mZ_k + \mZ_k \mB^{-1}\mB \mB^{-1}\mZ_k}\\
     &=& \mB - \ED{\mZ_k}.
\end{eqnarray*}
We, thereby, have derived
 \begin{eqnarray*}
 \ED{\| h^{k+1}-v\|_\mB^2} 
 &=&(h^k-v)^\top \left(\mB - \ED{\mZ_k}\right) (h^k - v) \\
 && \quad + (\nabla f(x^k)-v)^\top \ED{\mZ_k\mB^{-1}\mZ_k} (\nabla f(x^k) - v)\\
 &=& \|h^k - v\|_{\mB - \ED{\mZ}}^2 + \|\nabla f(x^k) - v\|_{\ED{\mZ}}^2.
\end{eqnarray*}
\hfill \qed
\end{proof}

\subsection{Proof of Lemma~\ref{lem:gk_general}\label{sec:prl2}}
\begin{proof}
    Throughout this proof, we will use without any mention that $\mZ_k^\top = \mZ_k$.
    
    Writing $g^k - v = a + b$, where $a\eqdef (\mI - \theta_k\mB^{-1}\mZ_k)(h^k - v)$ and $b\eqdef \theta_k \mB^{-1}\mZ_k (\nabla f(x^k) - v)$, we get $\|g^k\|_\mB^2\le 2(\|a\|_\mB^2 + \|b\|_\mB^2)$. Using Lemma~\ref{lem:zkbzk} and the definition of $\theta_k$ yields
    \begin{eqnarray*}
        \ED{\|a\|_\mB^2} 
        &=& \ED{\|\left(\mI - \theta_k\mB^{-1}\mZ_k\right)(h^k - v)\|_\mB^2}\\
        &=& (h^k - v)^\top \ED{\left( \mI - \theta_k\mZ_k\mB^{-1}\right)\mB\left( \mI - \theta_k\mB^{-1}\mZ_k\right)} (h^k - v)\\
        &=&  (h^k - v)^\top \ED{\left( \mB - \theta_k\mZ_k\mB^{-1}\mB - \mB\theta_k\mB^{-1}\mZ_k + \theta_k^2\mZ_k\mB^{-1}\mB\mB^{-1}\mZ_k \right)} (h^k - v)\\
        &=&(h^k - v)^\top \ED{\left( \mB - 2\mB + \theta_k^2\mZ_k\right)} (h^k - v)\\
        &=& \|h^k - v\|^2_{\ED{\theta^2\mZ} - \mB}.
    \end{eqnarray*}
    Similarly, the second term in the upper bound on $g^k$ can be rewritten as
    \begin{eqnarray*}
        \ED{\|b\|_\mB^2}
        &=& \ED{\|\theta_k \mB^{-1}\mZ_k (\nabla f(x^k) - v)\|^2_\mB}\\
        &=& (\nabla f(x^k) - v)^\top \ED{\theta_k^2\mZ_k \mB^{-1}\mB\mB^{-1}\mZ_k} (\nabla f(x^k) - v)\\
        &=& \|\nabla f(x^k) - v\|^2_{\mC}.
    \end{eqnarray*}
    Combining the pieces, we get the claim. \hfill \qed
\end{proof}

\section{Proofs for Section~\ref{sec:CD} \label{sec:proofs_CD}}
\subsection{Technical Lemmas\label{sec:techlem}}
We first start with an analogue of Lemma~\ref{lem:gk_general} allowing for a norm different from $\|\cdot\|_\mB$. We remark that 
matrix $\mQ'$ in the lemma is  not to be confused with the smoothness matrix $\mQ$ from Assumption~\ref{ass:M_smooth_inv}.

\begin{lemma} \label{lem:1xxxx} Let $\mQ'\succ 0$. The variance of $g^k$ as an estimator of $\nabla f(x^k)$ can be bounded as follows:
\begin{equation} \label{eq:almost_eso_g}
\frac12\ED{\|g^k\|_{\mQ'}^2} \leq  \|h^k \|_{\mPdiag^{-1} (\Probmat \circ \mQ')\mPdiag^{-1}- \mQ'}^2 +  \|\nabla f(x^k)\|^2_{\mPdiag^{-1} (\Probmat \circ \mQ')\mPdiag^{-1}}.
\end{equation}
\end{lemma}

\begin{proof}  
Denote $\mS_{k}$ to be a matrix with columns $e_i$ for $i\in \Range{\mS_{k}}$.  
We first write \[g^k = \underbrace{h^k - \mPdiag^{-1}  \mS_{k} \mS_{k}^\top h^k}_{a} +  \underbrace{ \mPdiag^{-1}    \mS_{k}\mS_{k}^\top \nabla f(x^k)}_{b}.\]
Let us bound the expectation of each term individually. The first term is equal to
\begin{eqnarray*}
\ED{\|a\|_{\mQ'}^2} &=&  
\ED{\left\| \left(\mI -  \mPdiag^{-1}  \mS_{k} \mS_{k}^\top\right) h^k  \right\|_{\mQ'}^2}
\\
&=&(h^k)^\top \ED{ \left(\mI -\mPdiag^{-1}  \mS_{k} \mS_{k}^\top \right)^\top \mQ' \left(\mI - \mPdiag^{-1}  \mS_{k} \mS_{k}^\top \right)} h^k
\\
&=&
(h^k)^\top  \ED{\left(\mQ' -\mPdiag^{-1}  \mS_{k} \mS_{k}^\top \mQ'- \mQ'  \mS_{k} \mS_{k}^\top\mPdiag^{-1}\right)}h^k  
\\
&&
\qquad \qquad  +(h^k)^\top  \ED{ \left(\mPdiag^{-1}  \mS_{k} \mS_{k}^\top \mQ'  \mS_{k} \mS_{k}^\top \mPdiag^{-1} \right)}h^k 
\\
&=& 
(h^k)^\top \left( \mPdiag^{-1} (\Probmat \circ \mQ')\mPdiag^{-1}- \mQ'  \right)h^k.
\end{eqnarray*}
The second term can be bounded as
\begin{eqnarray*}
\ED{\|b\|_{\mQ'}^2} &=&  \ED{\left\|\mPdiag^{-1} \mS_{k}^\top \nabla f(x^k)  \mS_{k} \right\|_{\mQ'}^2}=
\ED{ \|\nabla f(x^k) \|^2_{\mPdiag^{-1}  \mS_{k} \mS_{k}^\top \mQ'  \mS_{k} \mS_{k}^\top\mPdiag^{-1} }}
\\
&=&
\|\nabla f(x^k) \|^2_{\mPdiag^{-1} (\Probmat \circ \mQ')\mPdiag^{-1}}
\end{eqnarray*}
 It remains to combine the two bounds. \qed
\end{proof}

We also state the analogue of Lemma~\ref{lem:2}, which allows for a different norm as well. 
\begin{lemma} \label{lem:2xxxxxx} For all diagonal $\mD\succ 0$ we have
\begin{equation}\label{eq:h_dec_general}
\ED{\|h^{k+1} \|_{\mD}^2} = \|h^k \|_{\mD-\mPdiag \mD}^2 + \|\nabla f(x^k) \|_{\mPdiag \mD}^2.
\end{equation}
\end{lemma}

\begin{proof}  
Denote $\mS_{k}$ to be a matrix with columns $e_i$ for $i\in \mS_{k}$.  
We first write \[h^{k+1} = h^k - \mS_{k} \mS_{k}^\top  h^k + \mS_{k} \mS_{k}^\top  \nabla f(x^k)  .\]
Therefore
\begin{eqnarray*}
\ED{\|h^{k+1}\|_{\mD}^2} &=& \ED{\left\|(\mI- \mS_{k}\mS_{k}^\top )h^k + \mS_{k} \mS_{k}^\top  \nabla f(x^k) \right\|_{\mD}^2}
\\
&=& 
\ED{\left\|(\mI- \mS_{k}\mS_{k}^\top )h^k\right\|_{\mD}^2} +\ED{\left\| \mS_{k} \mS_{k}^\top  \nabla f(x^k) \right\|_{\mD}^2} 
\\
&& \qquad 
+ 2\ED{{h^k}^\top (\mI- \mS_{k}\mS_{k}^\top ) \mD  \mS_{k} \mS_{k}^\top  \nabla f(x^k)} 
\\
&=&
\|h^k \|_{\mD-\mPdiag \mD}^2 + \|\nabla f(x^k) \|_{\mPdiag \mD}^2.
\end{eqnarray*}
 \qed
\end{proof}

\subsection{Proof of Theorem~\ref{t:imp_nacc}}
\begin{proof}
Throughout the proof, we will use the following Lyapunov function: 
\begin{equation*}
	\Lnacc^k \eqdef f(x^{k})-f(x^*)+ \sigma \|h^{k} \|^2_{\mP^{-1}}.
\end{equation*}
Following similar steps to what we did before, we obtain
\begin{eqnarray*}
\E{\Lnacc^{k+1}}
&\stackrel{\eqref{eq:M_smooth}}{\leq}& 
f(x^{k})-f(x^*)+\alpha \E{\langle\nabla f(x^k), g^k\rangle}+\frac{\alpha^2}{2} \E{\|g^k\|_\mM^2} +\sigma\E{\|h^{k+1} \|^2_{\mPdiag^{-1}}}
\\
&=&
f(x^{k})-f(x^*)-\alpha \|\nabla f(x^k) \|_2^2+\frac{\alpha^2}{2} \E{\|g^k\|_\mM^2} +\sigma\E{\|h^{k+1} \|^2_{\mPdiag^{-1}}}
\\
&\stackrel{\eqref{eq:almost_eso_g}}{\leq} &
f(x^{k})-f(x^*)-\alpha \|\nabla f(x^k)\|_2^2
+\alpha^2  \|\nabla f(x^k) \|^2_{\mPdiag^{-1} (\Probmat \circ \mM)\mPdiag^{-1}} + \alpha^2\|h^k\|^2_{\mPdiag^{-1} (\Probmat \circ \mM)\mPdiag^{-1}-\mM}
\\
&& \qquad
+\sigma\E{\|h^{k+1} \|^2_{\mPdiag^{-1}}}.
\end{eqnarray*}
This is the place where the ESO~assumption comes into play. By applying it to the right-hand side of the bound above, we obtain
\begin{eqnarray*}
\E{\Lnacc^{k+1}} 
&\stackrel{\eqref{eq:ESO}}{\leq} &
f(x^{k})-f(x^*)-\alpha \|\nabla f(x^k)\|_2^2
+\alpha^2  \|\nabla f(x^k) \|^2_{\mVdiag\mPdiag^{-1}} + \alpha^2\|h^k\|^2_{\mVdiag\mPdiag^{-1}-\mM}
\\
&& \qquad
+\sigma\E{\|h^{k+1} \|^2_{\mPdiag^{-1}}}
\\
&\stackrel{\eqref{eq:h_dec_general}}{=} &
f(x^{k})-f(x^*)-\alpha \|\nabla f(x^k)\|_2^2
+\alpha^2  \|\nabla f(x^k) \|^2_{\mVdiag\mPdiag^{-1}}  + \alpha^2\|h^k\|^2_{\mVdiag\mPdiag^{-1}-\mM}
\\
&& \qquad
+\sigma\| \nabla f(x^k)\|_2^2+\sigma\|h^k \|^2_{\mPdiag^{-1}-\mI}
\\
&= &
f(x^{k})-f(x^*)-\left(\alpha - \alpha^2\max_{i}\frac{v_{i}}{p_{i}}- \sigma \right)\|\nabla f(x^k)\|_2^2
\\
&& \qquad
+  \|h^k\|^2_{\alpha^2(\mVdiag\mPdiag^{-1}-\mM)+\sigma(\mPdiag^{-1}-\mI)}.
\end{eqnarray*}
Due to Polyak-\L{}ojasiewicz inequality, we can further upper bound the last expression by
\begin{eqnarray*}
\left(1-\left(\alpha - \alpha^2\max_{i}\frac{v_{i}}{p_{i}}- \sigma \right)\mu \right)(f(x^{k})-f(x^*))
+  \|h^k\|^2_{\alpha^2(\mVdiag\mP^{-1}-\mM)+\sigma(\mP^{-1}-\mI)}.
\end{eqnarray*}
To finish the proof, it remains to use~\eqref{eq:assumption}. \qed
\end{proof}
\subsection{Proof of Corollary~\ref{cor:imp_nacc}}
The claim was obtained by choosing carefully $\alpha$ and~$\sigma$ using numerical grid search. Note that by strong convexity we have  $\mI\succeq \mu \diag(\mM)^{-1}$, so we can satisfy assumption~\eqref{eq:assumption}. Then, the claim follows immediately noticing that we can also set $\mVdiag=\diag(\mM)$ while maintaining
\[
\left(\alpha - \alpha^2\max_{i}\frac{\mM_{ii}}{p_{i}}- \sigma \right) \geq \frac{0.117 }{\trace(\mM)}.
\]

\subsection{Accelerated \texttt{SEGA} with arbitrary sampling\label{sec:acc_thm}} 
Before establishing the main theorem, we first state two technical lemmas which will be crucial for the analysis. First one, Lemma~\ref{lem:mirror} provides a key inequality following from~\eqref{eq:z_update}. The second one, Lemma~\ref{lem:acc_grad}, analyzes update~\eqref{eq:y_update} and was technically established throughout the proof of Theorem~\ref{t:imp_nacc}. We include a proof of lemmas in Appendix~\ref{app:mirror} and~\ref{app:grad} respectively. 

\begin{lemma}\label{lem:mirror}
For every $u\in \R^n$ we have
\begin{eqnarray}
&& \beta\langle \nabla f(x^{k+1}), z^{k}-u  \rangle -\frac{\beta \mu}{2}\|x^{k+1}-u \|_2^2\nonumber
\\
&& \qquad  \qquad \leq 
\beta^2\frac12\E{\|g^k\|_2^2}+\frac{1}{2}\|z^k-u\|_2^2
-\frac{1+\beta \mu}{2}\E{\| z^{k+1}-u\|_2^2} \label{eq:md_imp}
\end{eqnarray}
 \end{lemma}

 \begin{lemma} \label{lem:acc_grad} Letting $ \TD(v,p)\eqdef\max_i \frac{\sqrt{v_i}}{p_i}$, we have
\begin{eqnarray}
f(x^{k+1})-\E{f(y^{k+1})} +\|h^k\|^2_{\alpha^2(\mVdiag\mPdiag^{-3}-\mPdiag^{-1}\mM\mPdiag^{-1})} 
\geq  \left(\alpha - \alpha^2\TD(v,p)^2 \right)
 \| \nabla f(x^k)\|^2_{\mPdiag^{-1}}.  \label{eq:gd}
\end{eqnarray}
 \end{lemma}
 
Now we state the main theorem of Section~\ref{s:acc}, providing a convergence rate of \texttt{ASEGA} (Algorithm~\ref{alg:acc}) for arbitrary minibatch sampling.
As we mentioned, the convergence rate is, up to a constant factor, same as state-of-the-art minibatch accelerated coordinate descent~\cite{AccMbCd}. 

  \begin{theorem}\label{t:imp_acc}
Assume $\mM$--smoothness and $\mu$--strong convexity and that $v$ satisfies~\eqref{eq:ESO}. 
 Denote 
 \[
 \Lacc^{k} \eqdef\frac{2}{75} \frac{  \TD(v,p)^{-2}}{ \tau^2 } \left(\E{f(y^{k})}-f(x^{*})\right)
 + \frac{1+\beta\mu}{2}\E{\| z^{k}-x^*\|_2^2}+
 \sigma \E{\|h^{k} \|^2_{\mPdiag^{-2}}}
 \] 
 and choose 
\begin{eqnarray}
c_1&=&\max \left( 1,  \TD(v,p)^{-1} \frac{\sqrt{\mu}}{\min_i p_i} \right) \label{eq:c4_choice}
\\
\alpha &=&   \frac{1}{5 \TD(v,p)^{2}} \label{eq:alfa_choice}
\\
\beta &=& \frac{2}{75 \tau  \TD(v,p)^{2}} \label{eq:beta_choice}
\\
\sigma &=& 5\beta^2  \label{eq:sigma_choice}
\\
\tau &=&\frac{\sqrt{\frac{4}{9\cdot 5^4}  \TD(v,p)^{-4} \mu^2+\frac{8}{75}  \TD(v,p)^{-2} \mu }-\frac{2}{75} \TD(v,p)^{-2} \mu}{2}   \label{eq:tau_choice}
\end{eqnarray}
Then, we have
\[
\E{\Lacc^{k}}\leq \left(1-c_1^{-1}\tau \right)^k\Lacc^0.
\]
\end{theorem}

\begin{proof}  The proof technique is inspired by \cite{allen2014linear}. First of all, let us see what strong convexity of $f$ gives us:
\begin{eqnarray*}
\beta \left( f(x^{k+1})- f(x^*)\right)
\leq
\beta \langle \nabla f(x^{k+1}),x^{k+1}-x^*\rangle -\frac{\beta \mu}{2} \|x^*-x^{k+1}\|_2^2.
\end{eqnarray*}
Thus, we are interested in finding an upper bound for the scalar product that appeared above. We have
 \begin{eqnarray*}
&& \beta\langle \nabla f(x^{k+1}), z^{k}-u  \rangle -\frac{\beta \mu}{2}\|x^{k+1}-u \|_2^2+ \sigma \E{\|h^{k+1} \|^2_{\mPdiag^{-2}}}
\\
&& \qquad  \qquad 
\stackrel{\eqref{eq:md_imp}}{\leq} 
\beta^2\frac12\E{\|g^k\|_2^2}+\frac{1}{2}\|z^k-u\|_2^2
-\frac{1+\beta \mu}{2}\E{\| z^{k+1}-u\|_2^2} +  \sigma \E{\|h^{k+1} \|^2_{\mPdiag^{-2}}}.
\end{eqnarray*}
Using the Lemmas introduced above, we can upper bound the norms of $g^k$ and $h^{k+1}$ by using norms of $h^k$ and $\nabla f(x^k)$ to get the following:
\begin{eqnarray*}
&&\beta^2\frac12\E{\|g^k\|_2^2} +  \sigma \E{\|h^{k+1} \|^2_{\mPdiag^{-2}}}\\
&&\qquad\qquad\stackrel{\eqref{eq:h_dec_general}}{\leq}
\beta^2\frac12\E{\|g^k\|_2^2}
+\sigma\|h^{k} \|^2_{\mPdiag^{-2}-\mPdiag^{-1}} +\sigma  \|\nabla f(x^k)\|^2_{\mPdiag^{-1}}
\\
&&\qquad\qquad
\stackrel{\eqref{eq:almost_eso_g}}{\leq}
\beta^2\|h^k \|_{\mPdiag^{-1}- \mI}^2 + \beta^2 \|\nabla f(x^k)\|^2_{\mPdiag^{-1}}
+\sigma\|h^{k} \|^2_{\mPdiag^{-2}-\mPdiag^{-1}} +\sigma  \|\nabla f(x^k)\|^2_{\mPdiag^{-1}}.
\end{eqnarray*}
Now, let us get rid of $\nabla f(x^k)$ by using the gradients property from Lemma~\ref{lem:acc_grad}:
\begin{eqnarray*}
&&\beta^2\frac12\E{\|g^k\|_2^2} +  \sigma \E{\|h^{k+1} \|^2_{\mPdiag^{-2}}}\\
&&\qquad\qquad
\stackrel{\eqref{eq:gd}}{\leq} 
\beta^2\|h^k \|_{\mPdiag^{-1}- \mI}^2+ \left( \beta^2 +\sigma \right) \frac{f(x^{k+1})-f(y^{k+1}) +\|h^k\|^2_{\alpha^2(\mVdiag\mPdiag^{-3}-\mPdiag^{-1}\mM\mPdiag^{-1})}}{ \alpha - \alpha^2\TD(v,p)^2} 
+\sigma\|h^{k} \|^2_{\mPdiag^{-2}-\mPdiag^{-1}}
\\
&& \qquad\qquad
=
\|h^k \|^2_{\beta^2(\mPdiag^{-1}- \mI)+ \frac{(\beta^2+\sigma) \alpha^2}{ \alpha - \alpha^2\TD(v,p)^2} (\mVdiag\mPdiag^{-3}-\mPdiag^{-1}\mM\mPdiag^{-1}) +\sigma (\mPdiag^{-2}-\mPdiag^{-1})}
\\
&&\qquad\qquad\qquad
+  \frac{\beta^2 + \sigma }{ \alpha - \alpha^2\TD(v,p)^2}( f(x^{k+1})-\E{f(y^{k+1})}) 
\\
&& \qquad\qquad
\leq
\|h^k \|^2_{\beta^2\mPdiag^{-1}+ \frac{(\beta^2+\sigma) \alpha^2}{ \alpha - \alpha^2\TD(v,p)^2} \mVdiag\mPdiag^{-3} +\sigma (\mPdiag^{-2}-\mPdiag^{-1})}
+  \frac{\beta^2 + \sigma }{ \alpha - \alpha^2\TD(v,p)^2}( f(x^{k+1})-\E{f(y^{k+1})}) .
\end{eqnarray*}
Plugging this into the bound with which we started the proof, we deduce
\begin{eqnarray*}
&& \beta\langle \nabla f(x^{k+1}), z^{k}-u  \rangle -\frac{\beta \mu}{2}\|x^{k+1}-u \|_2^2+ \sigma \E{\|h^{k+1} \|^2_{\mPdiag^{-2}}}\\
&& \qquad  \qquad 
\leq
\|h^k \|^2_{\beta^2\mPdiag^{-1}+ \frac{(\beta^2+\sigma) \alpha^2}{ \alpha - \alpha^2\TD(v,p)^2} \mVdiag\mPdiag^{-3} +\sigma (\mPdiag^{-2}-\mPdiag^{-1})}
\\
&& \qquad \qquad \qquad 
+  \frac{\beta^2 + \sigma }{ \alpha - \alpha^2\TD(v,p)^2}( f(x^{k+1})-\E{f(y^{k+1})}) 
+\frac{1}{2}\|z^k-u \|_2^2
-\frac{1+\beta \mu}{2}\E{\| z^{k+1}-u\|_2^2 }.
\end{eqnarray*} 
Recalling our first step, we get with a few rearrangements
\begin{eqnarray*}
& & \beta \left( f(x^{k+1})- f(x^*)\right)
 \\
& & \qquad \qquad 
\leq
\beta \langle \nabla f(x^{k+1}),x^{k+1}-x^*\rangle -\frac{\beta \mu}{2} \|x^*-x^{k+1}\|_2^2
\\
& & \qquad \qquad 
=
\beta \langle \nabla f(x^{k+1}),x^{k+1}-z^{k}\rangle +\beta \langle \nabla f(x^{k+1)},z^{k}-x^*\rangle -\frac{\beta \mu}{2} \|x^*-x^{k+1}\|_2^2
\\
& & \qquad \qquad 
=
\frac{(1-\tau)\beta}{\tau} \langle \nabla f(x^{k+1}),y^k-x^{k+1}\rangle +\beta \langle \nabla f(x^{k+1}),z^{k}-x^*\rangle -\frac{\beta \mu}{2} \|x^*-x^{k+1}\|_2^2
\\
& &\qquad \qquad  
\leq 
\frac{(1-\tau)\beta}{\tau}\left( f(y^k)-f(x^{k+1})\right)
+
\|h^k \|^2_{\beta^2\mPdiag^{-1}+ \frac{(\beta^2+\sigma) \alpha^2}{ \alpha - \alpha^2\TD(v,p)^2} \mVdiag\mPdiag^{-3} +\sigma (\mPdiag^{-2}-\mPdiag^{-1})}
\\
& & \qquad \qquad\qquad+
  \frac{\beta^2 + \sigma }{ \alpha - \alpha^2\TD(v,p)^2}( f(x^{k+1})-\E{f(y^{k+1})}) 
   +
\frac{1}{2}\|z^k-x^*\|_2^2
\\
& & \qquad \qquad\qquad
-\frac{1+\beta\mu}{2}\E{\| z^{k+1}-x^*\|_2^2} - \sigma \E{\|h^{k+1} \|^2_{\mPdiag^{-2}}}.
\end{eqnarray*}
Let us choose $\sigma$, $\beta$ such that for some constant $c_2$ (which we choose at the end) we have
 \[
c_2\sigma=\beta^2, \qquad \beta=\frac{\alpha - \alpha^2\TD(v,p)^2}{(1+c_2^{-1}) \tau }.
 \]
 Consequently, we have
 \begin{eqnarray*}
&&
\frac{\alpha - \alpha^2\TD(v,p)^2}{(1+c_2^{-1}) \tau^2 } \left(\E{f(y^{k+1})}-f(x^{*})\right)
 + \frac{1+\beta\mu}{2}\E{\| z^{k+1}-x^*\|_2^2}+
 \sigma \E{\|h^{k+1} \|^2_{\mPdiag^{-2}}}
 \\
  &&
 \qquad \qquad \leq 
(1-\tau)\frac{\alpha - \alpha^2\TD(v,p)^2}{(1+c_2^{-1}) \tau^2 }\left( f(y^k)-f(x^{*})\right)+
  \frac{1}{2}\|z^k-x^*\|_2^2
  \\
&& 
\qquad \qquad \qquad+
  \|h^k \|^2_{\left(\mPdiag^{-1} -(1-c_2)\mI+  \frac{(1+c_2)  \alpha^2}{ \alpha - \alpha^2\TD(v,p)^2} \mVdiag\mPdiag^{-2} \right) \sigma \mPdiag^{-1}}
 \end{eqnarray*}
 Let us make a particular choice of $\alpha$, so that for some constant $c_3$ (which we choose at the end) we can obtain the equations below:
 \[
  \alpha = \frac{1}{c_3\TD(v,p)^2} 
  \quad
  \Rightarrow 
  \quad
  \alpha - \alpha^2\TD(v,p)^2
  =
   \frac{c_3-1}{c_3^2}\TD(v,p)^{-2}, \quad
  \frac{\alpha^2}{\alpha - \alpha^2\TD(v,p)^2}
  =
  \frac{1}{(c_3-1)\TD(v,p)^2}
  .
 \]
 Thus
  \begin{eqnarray*}
&&
\frac{ \frac{c_3-1}{c_3^2}\TD(v,p)^{-2}}{(1+c_2^{-1}) \tau^2 } \left(\E{f(y^{k+1})}-f(x^{*})\right)
 + \frac{1+\beta\mu}{2}\E{\| z^{k+1}-x^*\|_2^2}+
 \sigma \E{\|h^{k+1} \|^2_{\mPdiag^{-2}}}
 \\
  &&
 \qquad \qquad \leq 
(1-\tau)\frac{ \frac{c_3-1}{c_3^2}\TD(v,p)^{-2}}{(1+c_2^{-1}) \tau^2 }\left( f(y^k)-f(x^{*})\right)+
  \frac{1}{2}\|z^k-x^*\|_2^2
  \\
&& 
\qquad \qquad \qquad+
  \|h^k \|^2_{\left(\mPdiag^{-1} -(1-c_2)\mI+  \frac{(1+c_2) }{ (c_3-1)\TD(v,p)^{2}} \mVdiag\mPdiag^{-2} \right) \sigma \mPdiag^{-1}}.
 \end{eqnarray*}
 Using the definition of $\TD(v,p)$, one can see that the above gives

  \begin{eqnarray*}
&&
\frac{ \frac{c_3-1}{c_3^2} \TD(v,p)^{-2}}{(1+c_2^{-1}) \tau^2 } \left(\E{f(y^{k+1})}-f(x^{*})\right)
 + \frac{1+\beta\mu}{2}\E{\| z^{k+1}-x^*\|_2^2}+
 \sigma \E{\|h^{k+1} \|^2_{\mPdiag^{-2}}}
 \\
  &&
 \qquad \qquad \leq 
(1-\tau)\frac{ \frac{c_3-1}{c_3^2} \TD(v,p)^{-2}}{(1+c_2^{-1}) \tau^2 }\left( f(y^k)-f(x^{*})\right)+
  \frac{1}{2}\|z^k-x^*\|_2^2
+
  \|h^k \|^2_{\left(\mPdiag^{-1} -(1-c_2)\mI+  \frac{1+c_2 }{ c_3-1} \mI \right) \sigma \mPdiag^{-1}}.
 \end{eqnarray*}
 To get the convergence rate, we shall establish
 \begin{equation}\label{eq:h_suff_acc_imp}
 \left(1-c_2-  \frac{1+c_2 }{ c_3-1}\right) c_1 \mI \succeq \tau\mPdiag^{-1}
 \end{equation}
 and
 \begin{equation}\label{eq:acc_tau_implicit}
1+\beta \mu \geq \frac{1}{1-\tau}.
 \end{equation}
 To this end, let us recall that
 \[
 \beta=\frac{c_3-1}{c^2_2}  \TD(v,p)^{-2}\tau^{-1} \frac{1}{1+c_2^{-1}}.
 \]
Now we would like to set equality in~\eqref{eq:acc_tau_implicit}, which yields
\[
0=\tau^2+\frac{c_3-1}{c^2_2} \TD(v,p)^{-2}\frac{1}{1+c_2^{-1}}  \mu  \tau- \frac{c_3-1}{c^2_2}  \TD(v,p)^{-2}\frac{1}{1+c_2^{-1}} \mu =0.
\]
 This, in turn, implies
 \begin{eqnarray*}
\tau
&=&
\frac{\sqrt{\left(\frac{c_3-1}{c^2_2}\right)^2  \TD(v,p)^{-4}\frac{1}{\left(1+c_2^{-1}\right)^2} \mu^2+4\frac{c_3-1}{c^2_2}  \TD(v,p)^{-2} \frac{1}{1+c_2^{-1}}\mu }-\frac{c_3-1}{c^2_2}  \TD(v,p)^{-2} \frac{1}{1+c_2^{-1}}\mu}{2}
\\
&=&
\cO\left(\sqrt{\frac{c_3-1}{c^2_2}}\frac{1}{\sqrt{1+c_2^{-1}}} \TD(v,p)^{-1}\sqrt{\mu} \right).
  \end{eqnarray*}
 Notice that for any $c\leq 1$ we have $\frac{\sqrt{c^2+4c}-c}{2}\leq \sqrt{c}$ and therefore
\begin{equation}\label{eq:tau_bound}
\tau \leq \sqrt{\frac{c_3-1}{c^2_2}}  \TD(v,p)^{-1} \frac{1}{\sqrt{1+c_2^{-1}}} \sqrt{\mu} .
\end{equation}
 Using this inequality and a particular choice of constants, we can upper bound $\mP^{-1}$ by a matrix proportional to identity as shown below:
 \begin{eqnarray*}
 \tau \mPdiag^{-1} &\stackrel{\eqref{eq:tau_bound}}{\preceq}&
 \sqrt{\frac{c_3-1}{c^2_2}} \TD(v,p)^{-1} \frac{1}{\sqrt{1+c_2^{-1}}} \sqrt{\mu} \mPdiag^{-1} 
 \\
 &\preceq &
   \sqrt{\frac{c_3-1}{c^2_2}} \TD(v,p)^{-1}\frac{1}{\sqrt{1+c_2^{-1}}} \frac{\sqrt{\mu}}{\min_i p_i}  \mI
 \\
 &\stackrel{\eqref{eq:c4_choice}}{\preceq} &
   \sqrt{\frac{c_3-1}{c^2_2}}\frac{1}{\sqrt{1+c_2^{-1}}} c_1 \mI
 \\
 &\stackrel{(*)}{\preceq} & 
  \left(1-c_2-  \frac{1+c_2 }{ c_3-1}\right) c_1 \mI ,
 \end{eqnarray*}
 which is exactly \eqref{eq:h_suff_acc_imp}. Above, $(*)$ holds for choice $c_3=5$ and $c_2=\frac{1}{5}$. It remains to verify that~\eqref{eq:alfa_choice}, \eqref{eq:beta_choice}, \eqref{eq:sigma_choice} and~\eqref{eq:tau_choice} indeed correspond to our derivations.
 \qed
\end{proof}

We also mention, without a proof, that acceleration parameters can be chosen in general such that $c_1$ can be lower bounded by constant and therefore the rate from Theorem~\ref{t:imp_acc} coincides with the rate from Table~\ref{tab:CDcmp}. Corollary~\ref{cor:acc_imp} is in fact a weaker result of that type. 
 
 \subsubsection{Proof of Corollary~\ref{cor:acc_imp}}
It suffices to verify that one can choose $v=\diag(\mM)$ in~\eqref{eq:ESO} and that due to $p_i\propto \sqrt{\mM_{ii}}$ we have $c_1= 1$.

\subsection{Proof of Lemma~\ref{lem:mirror}\label{app:mirror}} 
 \begin{proof}
 Firstly \eqref{eq:z_update}, is equivalent to
\[
z^{k+1}=\argmin_z \psi^k(z)\eqdef \frac{1}{2}\| z-z^k\|_2^2+ \beta \langle g^k, z \rangle   +\frac{\beta \mu}{2}\|z-x^{k+1}\|_2^2.
\]
Therefore, we have for every $u$
\begin{align}
\nonumber
0&=\langle \nabla \psi^k(z^{k+1}),z^{k+1}-u \rangle \\
&=
\langle z^{k+1}-z^k, z^{k+1}-u\rangle +\beta \langle g^k, z^{k+1}-u  \rangle +\beta \mu \langle z^{k+1}-x^{k+1}, z^{k+1}-u\rangle. \label{eq:zk_plus_1_optimal}
\end{align}
Next, by generalized Pythagorean theorem we have
\begin{equation}\label{eq:pyt_z}
\langle z^{k+1}-z^k,z^{k+1}-u \rangle =\frac12 \|z^k-z^{k+1}\|_2^2-\frac12 \|z^k-u\|_2^2+\frac12 \|u-z^{k+1}\|_2^2
\end{equation}
and
\begin{equation}\label{eq:pyt_x}
\langle z^{k+1}-x^{k+1},z^{k+1}-u \rangle =\frac12 \|x^{k+1}-z^{k+1}\|_2^2-\frac12 \|x^{k+1}-u\|_2^2+\frac12 \|u-z^{k+1}\|_2^2.
\end{equation}
Plugging~\eqref{eq:pyt_z} and~\eqref{eq:pyt_x} into~\eqref{eq:zk_plus_1_optimal} we obtain
\begin{eqnarray*}
&&
\beta \langle g^k, z^{k}-u  \rangle -\frac{\beta \mu}{2}\|x^{k+1}-u\|_2^2
\\
&& \qquad \qquad
\leq 
\beta  \langle g^k, z^{k}-z^{k+1}  \rangle
 -\frac12\| z^k-z^{k+1}\|_2^2+\frac{1}{2}\|z^k-u\|_2^2-\frac{1+\beta \mu}{2}\| z^{k+1}-u\|_2^2
\\
&& \qquad \qquad
\stackrel{(*)}{\leq} 
\frac{\beta^2}{2}\| g^k \|_2^2+\frac{1}{2}\|z^k-u\|_2^2
-\frac{1+\beta \mu}{2}\| z^{k+1}-u\|_2^2.
\end{eqnarray*}
The step marked by $(*)$ holds due to Cauchy-Schwartz inequality. 
It remains to take the expectation conditioned on $x^{k+1}$ and use~\eqref{eq:unbiased_estimator}.
\hfill \qed
\end{proof}
 
\subsection{Proof of Lemma~\ref{lem:acc_grad}\label{app:grad}}
 \begin{proof}
 The shortest, although not the most intuitive, way to write the proof is to put matrix factor into norms. Apart from this trick, the proof is quite simple consists of applying smoothness followed by ESO:
\begin{eqnarray*}
\E{f(y^{k+1})} - f(x^{k+1})
&\stackrel{\eqref{eq:M_smooth}}{\leq}& 
-\alpha \E{\langle\nabla f(x^k), \mPdiag^{-1} g^k\rangle}+\frac{\alpha^2}{2} \E{\|\mPdiag^{-1} g^k\|_\mM^2} 
\\
&=&
-\alpha \|\nabla f(x^k)\|^2_{\mPdiag^{-1}}  +\frac{\alpha^2}{2} \E{\| g^k\|_{\mPdiag^{-1}\mM\mPdiag^{-1}}} 
\\
&\stackrel{\eqref{eq:almost_eso_g}}{\leq} &
-\alpha \|\nabla f(x^k)\|^2_{\mPdiag^{-1}}  +\alpha^2\|\nabla f(x^k) \|^2_{\mPdiag^{-1}(\Probmat \circ \mPdiag^{-1}\mM\mPdiag^{-1})\mPdiag^{-1}}
\\
&& \qquad 
+\alpha^2 \| h^k\|^2_{\mPdiag^{-1}(\Probmat \circ \mPdiag^{-1}\mM\mPdiag^{-1})\mPdiag^{-1}-\mPdiag^{-1}\mM\mPdiag^{-1}} 
\\
&= &
-\alpha \|\nabla f(x^k)\|^2_{\mPdiag^{-1}}  +\alpha^2\|\nabla f(x^k) \|^2_{\mPdiag^{-2}(\Probmat \circ\mM)\mPdiag^{-2}}
\\
&& \qquad 
+\alpha^2 \| h^k\|^2_{\mPdiag^{-2}(\Probmat \circ \mM)\mPdiag^{-2}-\mPdiag^{-1}\mM\mPdiag^{-1}} 
\\
&\stackrel{\eqref{eq:ESO}}{\leq} &
-\alpha \|\nabla f(x^k)\|^2_{\mPdiag^{-1}}  +\alpha^2\|\nabla f(x^k) \|^2_{\mVdiag \mPdiag^{-3}}
\\
&& \qquad 
+\alpha^2 \| h^k\|^2_{\mVdiag \mPdiag^{-3}-\mPdiag^{-1}\mM\mPdiag^{-1}} 
\\
&\leq&
-\left(\alpha - \alpha^2\max_{i}\frac{v_i}{p_{i}^2} \right)\|f(x^k)\|^2_{\mPdiag^{-1}}  
+\alpha^2 \| h^k\|^2_{\mVdiag\mPdiag^{-3}-\mPdiag^{-1}\mM\mPdiag^{-1}} .
\end{eqnarray*}
\hfill \qed
\end{proof}

\section{Subspace \texttt{SEGA}: a More Aggressive Approach \label{sec:subSEGA}}

In this section we describe a {\em more aggressive} variant of \texttt{SEGA}, one that exploits the fact that the gradients of $f$ lie in a lower dimensional subspace if this is indeed the case.

In particular, assume that $F(x) = f(x) + R(x)$ and \[f(x) = \phi(\mA x),\] where $\mA\in \R^{m\times n}$\footnote{Strong convexity is not compatible with the assumption that $\mA$ does not have full rank, so a different type of analysis using Polyak-\L{}ojasiewicz inequality is required to give a formal justification. However, we proceed with the analysis anyway to build the intuition why this approach leads to better rates.}. Note that $\nabla f(x)$ lies in $\Range{\mA^\top}$. There are situations where the dimension of $\Range{\mA^\top}$ is much smaller than $n$. For instance, this happens when $m\ll n$. However, standard coordinate descent methods still move around in directions $e_i\in \R^n$ for all $i$. We can modify the gradient sketch method to force our gradient estimate to lie in $\Range{\mA^\top}$, hoping that this will lead to faster convergence.

\subsection{The algorithm}

Let $x^k$ be the current iterate, and let  $h^k$ be the current estimate of the gradient of $f$. Assume that the sketch $\mS_k^\top \nabla f(x^k)$ is available. We can now define $h^{k+1}$ through the following modified sketch-and-project process:
\begin{eqnarray}
h^{k+1} &=& \arg \min_{h\in \R^{n}} \| h -  h^k\|_\mB^2 \notag \\
&& \text{subject to} \quad \mS_k^\top h = \mS_k^\top \nabla f(x^k), \label{eq:sketch-n-project2B}\\
&& \phantom{subject to} \quad h \in \Range{\mA^\top}.\notag
\end{eqnarray}
Before proceeding further, we note that there are such sketches and metric (as discussed in Section~\ref{sec:optimal}) which keep $h\in  \Range{\mA^\top}$ implicitly, and therefore one might omit the extra constraint in such case. In fact, the mentioned sketches also lead to a faster convergence, which is the main takeaway from this section.

Standard arguments reveal that the closed-form solution of \eqref{eq:sketch-n-project2B} is
\begin{equation}h^{k+1} = \mH\ \left(h^k - \mB^{-1}\mS_k(\mS_k^\top\mH\mB^{-1}\mS_k)^\dagger \mS_k^\top(\mH h^k - \nabla f(x^k)) \right), \label{eq:h^{k+1}3}\end{equation}
where  
\begin{align}\label{eq:H}
\mH \eqdef \mA^\top (\mA\mB \mA^\top)^\dagger \mA\mB
\end{align}
 is the projector onto $\Range{\mA^\top}$.
A quick sanity check reveals that this gives the same formula as \eqref{eq:h^{k+1}} in the case where $\Range{\mA^\top} = \R^n$. We can also write
\begin{equation} \label{eq:general_update_of_h}h^{k+1} = \mH h^k -   \mH\mB^{-1}\mZ_k(\mH h^k - \nabla f(x^k)) = \left(\mI - \mH\mB^{-1}\mZ_k\right)\mH h^k + \mH\mB^{-1}\mZ_k\nabla f(x^k),\end{equation}
where 
\begin{align}\label{eq:Z_k}
\mZ_k \eqdef \mS_k(\mS_k^\top\mH\mB^{-1}\mS_k)^\dagger \mS_k^\top.
\end{align}
 Assume that $\theta_k$ is chosen in such a way that
 \begin{equation*}
     \ED{\theta_k\mZ_k} = \mB.
 \end{equation*}
 Then, the following estimate of $\nabla f(x^k)$
 \begin{align}\label{eq:g^k_agressive}
     g^k\eqdef \mH h^k + \theta_k\mH\mB^{-1}\mZ_k (\nabla f(x^k) - \mH h^k)
 \end{align}
 is unbiased, i.e.\ $\ED{g^k} = \nabla f(x^k)$. After evaluating $g^k$, we perform the same step as in \texttt{SEGA}: \[x^{k+1} = \prox(x^k - \alpha g^k). \]
By inspecting \eqref{eq:sketch-n-project2B}, \eqref{eq:H} and \eqref{eq:g^k_agressive}, we get the following simple observation.
\begin{lemma} \label{lem:range_preserve}If $h^0\in \Range{\mA^\top}$, then $h^k, g^k \in \Range{\mA^\top}$ for all $k$.
\end{lemma}

Consequently, if $h^0 \in \Range{\mA^\top}$, \eqref{eq:h^{k+1}3}  simplifies to
\begin{equation} \label{eq:b98g9fd2}
h^{k+1} =  h^k - \mH\mB^{-1}\mS_k(\mS_k^\top\mH\mB^{-1}\mS_k)^\dagger \mS_k^\top(h^k - \nabla f(x^k)) 
\end{equation}
and \eqref{eq:g^k_agressive} simplifies to
\begin{equation} \label{eq:g^k2-xx}
 g^k \eqdef h^k + \theta_k\mH\mB^{-1}\mZ_k (\nabla f(x^k) - h^k).\end{equation}

\begin{example}[Coordinate sketch]\label{ex:coord_setup2}
Consider $\mB=\mI$ and the  choice of $\cD$ given by $\mS=e_i $ with probability $p_i>0$. Then we can choose the bias-correcting random variable as $\theta = \theta(s) = \frac{w_i}{p_i} $, where $w_i \eqdef \|\mH e_i\|_2^2 = e_i^\top \mH e_i$. Indeed, with this choice, \eqref{eq:unbiased} is satisfied. For simplicity, further choose $p_i = 1/n$ for all $i$. We then have
\begin{equation} \label{eq:b98g9fd2-coord} h^{k+1} = h^k -   \frac{e_{i}^\top  h^k - e_{i}^\top \nabla f(x^k)}{w_{i}} \mH e_{i} = \left(\mI - \frac{\mH e_{i} e_{i}^\top}{w_{i}}\right)  h^k + \frac{ \mH e_{i} e_{i}^\top}{w_{i}}\nabla f(x^k)\end{equation}
and \eqref{eq:g^k2-xx} simplifies to
\begin{equation} \label{eq:g^k2-xx-coord} g^k \eqdef (1-\theta_k) h^k + \theta_k h^{k+1} = h^k + n \mH e_{i} e_{i}^\top \left( \nabla f(x^k) -  h^k \right).\end{equation}
\end{example}

\subsection{Lemmas}
All theory provided in this subsection is, in fact, a straightforward generalization of our non-subspace results. The reader can recognize similarities in both statements and proofs with that of previous sections.
\begin{lemma}\label{lem:properties_of_Z_and_H}
    Define $\mZ_k$ and $\mH$ as in equations~\eqref{eq:Z_k} and~\eqref{eq:H}. Then $\mZ_k$ is symmetric, $\mZ_k \mH \mB^{-1}\mZ_k = \mZ_k$, $\mH^2=\mH$ and $\mH\mB^{-1} = \mB^{-1}\mH^\top $.
\end{lemma}
\begin{proof}
    The symmetry of $\mZ_k$ follows from its definition. The second statement is a corollary of the equations $((\mA_1\mA_2)^\dagger)^\top = (\mA_2^\top\mA_1^\top)^\dagger$ and $\mA_1^\dagger \mA_1 \mA_1^\dagger = \mA_1^\dagger$, which are true for any matrices $\mA_1, \mA_2$. 
Finally, the last two rules follow directly from the definition of $\mH$ and the property $\mA_1^\dagger \mA_1 \mA_1^\dagger = \mA_1^\dagger$. \hfill \qed
\end{proof}

\begin{lemma}\label{lem:bug79dv987gs}
    Assume $h^k\in \Range{\mA^\top}$. Then \[\ED{\| h^{k+1} - v\|_\mB^2} = \|h^k - v\|_{\mB - \ED{\mZ}}^2 + \|\nabla f(x^k) - v\|_{\ED{\mZ}}^2\] for any vector $v\in \Range{\mA^\top}$.
\end{lemma}
\begin{proof}
    By Lemma \ref{lem:properties_of_Z_and_H} we can rewrite $\mH\mB^{-1}$ as $\mB^{-1}\mH^\top$, so
\begin{eqnarray}
    \ED{\| h^{k+1}-v\|_\mB^2} 
    &\overset{\eqref{eq:general_update_of_h}}{=}& \ED{ \left\|h^k - \mH\mB^{-1}\mZ_k (h^k - \nabla f(x^k)) - v \right\|_\mB^2 } \nonumber\\
 &=&  \ED{ \left\| \left(\mI -\mH \mB^{-1}\mZ_k \right)(h^k - v) + \mH\mB^{-1}\mZ_k (\nabla f(x^k)-v)  \right\|_\mB^2 }\nonumber\\
 &=&  \ED{ \left\| \left(\mI - \mB^{-1}\mH^\top\mZ_k \right)(h^k - v) + \mH\mB^{-1}\mZ_k(\nabla f(x^k)-v)  \right\|_\mB^2 }\nonumber\\
 &=& \ED{ \left\| \left(\mI - \mB^{-1}\mH^\top\mZ_k \right)(h^k-v) \right\|_\mB^2} + \ED{\left\|\mH\mB^{-1}\mZ_k  ( \nabla f(x^k) - v)  \right\|_\mB^2 }\nonumber\\
 && \quad + 2 (h^k-v)^\top\ED{\left(\mI - \mB^{-1}\mH^\top\mZ_k \right)^\top \mB\mH\mB^{-1}\mZ_k}  (\nabla f(x^k) - v) \nonumber\\
 &=& (h^k-v)^\top \ED{ \left(\mI - \mB^{-1}\mH^\top\mZ_k \right)^\top \mB  \left(\mI - \mH\mB^{-1}\mZ_k \right)} (h^k-v) \nonumber\\
 && \quad + (\nabla f(x^k)-v)^\top \ED{\mZ_k\mB^{-1}\mH^\top\mB\mH\mB^{-1}\mZ_k} (\nabla f(x^k) -v)\nonumber\\
 && \quad + 2 (h^k-v)^\top\ED{ \mB\mH\mB^{-1}\mZ_k - \mZ_k\mH\mH\mB^{-1}\mZ_k}  (\nabla f(x^k) - v). \label{eq:last_term_in_h_minus_v}
 \end{eqnarray}
By Lemma~\ref{lem:properties_of_Z_and_H} we have 
\begin{align*}
	\mZ_k\mH\mH\mB^{-1}\mZ_k = \mZ_k\mH\mB^{-1}\mZ_k = \mZ_k,
\end{align*}
so the last term in~\eqref{eq:last_term_in_h_minus_v} is equal to 0. As for the other two, expanding the matrix factor in the first term leads to
\begin{eqnarray*}
      \left(\mI - \mB^{-1}\mH^\top\mZ_k \right)^\top \mB  \left(\mI - \mH\mB^{-1}\mZ_k \right)
     &=& \left(\mI - \mZ_k \mH\mB^{-1} \right) \mB  \left(\mI - \mH\mB^{-1}\mZ_k\right)\\
     &=& \mB - \mZ_k\mH\mB^{-1}\mB - \mB\mB^{-1}\mH^\top\mZ_k + \mZ_k\mH \mB^{-1}\mB \mH\mB^{-1}\mZ_k\\
     &=& \mB - \mZ_k\mH - \mH^\top\mZ_k + \mZ_k.
\end{eqnarray*}
 Let us mention that $\mH(h^k - v) = h^k - v$ and $(h^k - v)^\top\mH^\top=(h^k - v)^\top$ as both vectors $h^k$ and $v$ belong to $\Range{\mA^\top}$. Therefore,
 \begin{align*}
     (h^k-v)^\top \ED{ \mB - \mZ_k\mH - \mH^\top\mZ_k + \mZ_k} (h^k-v) = (h^k-v)^\top \left(\mB - \ED{\mZ_k}\right) (h^k-v).
 \end{align*}

 It remains to consider
 \begin{align*}
 	\ED{\mZ_k\mB^{-1}\mH^\top\mB\mH\mB^{-1}\mZ_k} = \ED{\mZ_k\mH\mB^{-1}\mB\mH\mB^{-1}\mZ_k} = \ED{\mZ_k}.
 \end{align*}
We, thereby, have derived
 \begin{eqnarray*}
 \ED{\| h^{k+1}-v\|_\mB^2} 
 &=&(h^k-v)^\top \left( \mB - \ED{\mZ_k}\right) (h^k - v)\\
 &&\quad + (\nabla f(x^k)-v)^\top \ED{\mZ_k\mB^{-1}\mZ_k} (\nabla f(x^k) - v)\\
 &=& \|h^k - v\|_{ \mB - \ED{\mZ_k}}^2 + \|\nabla f(x^k) - v\|_{\ED{\mZ}}^2.
\end{eqnarray*}
 \hfill \qed
 \end{proof}

\begin{lemma}\label{lem:nb98gdf9jf}
    Suppose $h^k\in\Range{\mA^\top}$ and $g^k$ is defined by~\eqref{eq:g^k_agressive}. Then
    \begin{align}\label{eq:lemma2_agressive}
        \ED{\|g^k - v\|^2_\mB} \le \|h^k - v\|^2_{\mC - \mB} + \|\nabla f(x^k) -v\|^2_{\mC}
    \end{align}
     for any $v\in\Range{\mA^\top}$, where
     \begin{align}\label{eq:mC}
       \mC\eqdef \ED{\theta^2\mZ}.
     \end{align}
\end{lemma}
\begin{proof}
    Writing $g^k - v = a + b$, where $a\eqdef (\mI - \theta_k\mH\mB^{-1}\mZ_k)(h^k - v)$ and $b\eqdef \theta_k \mH\mB^{-1}\mZ_k (\nabla f(x^k) - v)$, we get $\|g^k\|_\mB^2\le 2(\|a\|_\mB^2 + \|b\|_\mB^2)$.  By definition of $\theta_k$,
    \begin{eqnarray*}
        \ED{\|a\|_\mB^2} 
        &=& \ED{\|\left(\mI - \theta_k\mH\mB^{-1}\mZ_k\right)(h^k - v)\|_\mB^2}\\
        &=& (h^k - v)^\top \ED{\left( \mI - \theta_k\mZ_k\mB^{-1}\mH\right)\mB\left( \mI - \theta_k\mH\mB^{-1}\mZ_k\right)} (h^k - v)\\
        &=&  (h^k - v)^\top \ED{\left( \mB - \theta_k\mZ_k\mB^{-1}\mH\mB - \mB\theta_k\mH\mB^{-1}\mZ_k + \theta_k^2\mZ_k\mB^{-1}\mH\mB\mH\mB^{-1}\mZ_k \right)} (h^k - v).
    \end{eqnarray*}
According to Lemma~\ref{lem:properties_of_Z_and_H},  $\mH\mB^{-1}=\mB^{-1}\mH$ and $\mZ_k\mH\mB^{-1}\mZ_k=\mZ_k$, so
    \begin{eqnarray*}
    \ED{\|a\|_\mB^2} 
        &=&(h^k - v)^\top \ED{\left( \mB - \theta_k\mZ_k\mH - \theta_k\mH^\top\mZ_k + \theta_k^2\mZ_k\right)} (h^k - v)\\
        &=& \|h^k - v\|^2_{\ED{\theta^2\mZ} - \mB},
    \end{eqnarray*}
    where in the last step we used the assumption that $h^k$ and $v$ are from $\Range{\mA^\top}$ and $\mH$ is the projector operator onto $\Range{\mA^\top}$.
    
    Similarly, the second term in the upper bound on $g^k$ can be rewritten as
    \begin{eqnarray*}
        \ED{\|b\|_\mB^2}
        &=& \ED{\|\theta_k \mH\mB^{-1}\mZ_k (\nabla f(x^k) - v)\|^2_\mB}\\
        &=& (\nabla f(x^k) - v)^\top \ED{\theta_k^2\mZ_k \mB^{-1}\mH^\top\mB\mH\mB^{-1}\mZ_k} (\nabla f(x^k) - v)\\
        &=& \|\nabla f(x^k) - v\|^2_{\ED{\theta_k^2\mZ_k}}.
    \end{eqnarray*}
    Combining the pieces, we get the claim. \hfill \qed
\end{proof}

\subsection{Main result}

The main result of this section is:

\begin{theorem}\label{thm:main_agressive}
    Assume that $f$ is $\mmM$--smooth, $\mu$--strongly convex, and that $\alpha>0$ is such that
\begin{equation}
        \alpha\left(2(\mC - \mB) +\sigma \mu \mB\right) \le \sigma\ED{\mZ},\qquad \alpha \mC \le \frac{1}{2}\left(\mmM - \sigma \ED{\mZ}\right). \label{eq:general_bound_on_stepsize_agressive}        
    \end{equation}
    If we define $\Lgen^k \eqdef \|x^k - x^*\|^2_{\mB} + \sigma \alpha \|h^k - \nabla f(x^k)\|^2_{\mB}$, then
$
        \E{\Lgen^{k}} \le (1 - \alpha\mu)^k \Lgen^0.
$
\end{theorem}
\begin{proof}
Having established Lemmas~\ref{lem:properties_of_Z_and_H}, \ref{lem:bug79dv987gs} and
\ref{lem:nb98gdf9jf}, the proof follows the same steps as the proof of  Theorem~\ref{thm:main}. \qed
\end{proof}

\subsection{Optimal choice of $\mB$ and $\mS_k$ \label{sec:optimal}}

Let us now slightly change the value of $\theta_k$ that we use in the algorithm. Instead of seeking for $\theta_k$ giving $\ED{\theta_k\mZ_k}=\mB$, we will use the one that gives $\ED{\theta_k\mZ_k} = \mB\mH$. This will steal lead to $\ED{g^k} = \nabla f(x^k)$ and, if $f$ is strongly-convex, we can still show the convergence rate of Theorem~\ref{thm:main_agressive}. Although the strong convexity assumption is simplistic, the new idea results in a surprising finding.

Let $a_1, \dotsc, a_m$ be the columns of $\mA^\top$ and $\mU\in\mathbb{R}^{d \times n}$ be a matrix that transforms these columns into an orthogonal basis of $d\eqdef \text{Rank}(\mA)$ vectors. Set $\mB = \mU^\top\mU$. Then, $\langle a_i, a_j\rangle_\mB =  0$ for any $i\neq j$. Assume for simplicity, that $\|a_i\|_\mB\neq 0$ for $i \le d$ and $\|a_i\|_\mB=0$ for $i>d$. This is always true up to permutation of $a_1,\dotsc, a_m$. Choose also $\mS_k\in \mathbb{R}^n$ equal to $\xi_i\eqdef \tfrac{\mB a_i}{\|a_i\|_\mB}$ with $i$ sampled with probability $p_i > 0$, and $\theta_k = p_i^{-1}$. Clearly, one has
\begin{align*}
    \ED{\theta_k\mZ_k} = \sum_{i=1}^d p_ip_i^{-1} \xi_i (\xi_i^\top \mH\mB^{-1}\xi_i)^\dagger \xi_i^\top = \sum_{i=1}^d \xi_i \|a_i\|_\mB^2 (a_i^\top \mB\mH\mB^{-1}\mB a_i)^\dagger \xi_i^\top.
\end{align*}
Since $a_i$ lies in $\Range{\mA^\top}$, we have $\mH a_i = a_i$, which gives
\begin{align}\label{eq:ED_theta_mZ}
    \ED{\theta_k \mZ_k} = \sum_{i=1}^d \xi_i \|a_i\|_\mB^2 (a_i^\top \mB a_i)^\dagger \xi_i^\top = \sum_{i=1}^d \xi_i \xi_i^\top.
\end{align}
By definition of $\mB$,
\begin{align*}
    (\mA\mB\mA^\top)^\dagger = (\text{diag}(\|a_i\|^2_\mB))^\dagger = \sum_{i=1}^d \|a_i\|_\mB^{-2}e_i e_i^\top.
\end{align*}
Thus,
\begin{align*}
    \mB\mH = \mB\mA^\top (\mA\mB\mA^\top) \mA\mB = \sum_{i=1}^d \frac{(\mB a_i)^\top \mB a_i}{\|a_i\|_\mB^2} = \ED{\theta_k \mZ_k},
\end{align*}
so we have achieved our goal. Note that if $h^0\in \Range{\mA^\top}$, we have $h^k\in  \Range{\mA^\top}$ even without implicitly enforcing it in~\eqref{eq:sketch-n-project2B}. Therefore, the method can be seen as \emph{\texttt{SEGA} with a smart choice of both sketches and metric} in which we project. 

To show how the choice of $\mB$ and of the sketches provided above improves the rate, let us take a closer look at the conditions of Theorem~\ref{thm:main_agressive}. We have
\begin{eqnarray*}
    \mC
    \stackrel{\eqref{eq:mC}}{=} \ED{\theta^2 \mZ} \stackrel{\eqref{eq:ED_theta_mZ}}{=} \sum_{i=1}^d p_i p_i^{-2} \xi_i \xi_i^\top = \sum_{i=1}^d p_i^{-1} \xi_i \xi_i^\top.
\end{eqnarray*}
If we assume that $\sigma \le 2/ \mu$, then the first bound on $\alpha$ simplifies to
\begin{align*}
    \alpha(2(\mC - \mB) + \sigma\mu\mB)\le 2\alpha\mC \le \sigma \ED{\mZ} = \sigma\sum_{i=1}^d p_i \xi_i \xi_i^\top,
\end{align*}
where the second part needs to be verified by choosing $\alpha$ to be small enough. For this it is sufficient to take $\alpha\le \sigma\max p_i^{-2}$ as every summand $\xi_i \xi_i^\top$ in the expression for $\mC$ is positive definite. As for the second condition, it is enough to choose $\sigma \le \tfrac{\lambda_{\max}(\mQ)}{2\lambda_{\min}(\ED{\mZ})}$ and $\alpha\le \tfrac{\lambda_{\max}(\mQ)}{4\lambda_{\min}(\mC)}$. Note that $\xi_i\xi_i^\top\le \|\xi_i\|_2^2\mI$, so for uniform sampling with $p_i = \tfrac{1}{d}$ and uniform $\mQ$--smoothness with $\mQ = \tfrac{1}{L}\mI$ we get the following condition on $\alpha$:
\begin{align*}
    \alpha\le \min\left\{\frac{\sigma}{d^2}, \frac{1}{4L d\max_i \|\xi_i\|_2^2} \right\}.
\end{align*}
In particular, choosing $\sigma =\min \left\{\frac{2}{\mu},\frac{\lambda_{\max}(\mQ)}{2\lambda_{\min}(\ED{\mZ})} \right\} = \min \left\{\frac{2}{\mu}, \frac{d}{2L\max_i \|\xi_i\|_2^2} \right\}$, we get the requirement
\begin{align*}
    \alpha\le \min\left\{\frac{2}{\mu d^2}, \frac{1}{4L d\max_i \|\xi_i\|_2^2} \right\}.
\end{align*}
Typically, $d \ll \tfrac{1}{\mu}$, so the leading term in the maximum above is the second one and we get ${\cal O} \left(\tfrac{1}{d} \right)$ requirement instead of previous ${\cal O} \left(\tfrac{1}{n} \right)$.

\subsection{The conclusion of subspace \texttt{SEGA}}
Let us recall that $g^k = h^k + \theta_k\mB^{-1}\mZ_k(\nabla f(x^k) - h^k)$. A careful examination shows that when we reduce $\theta_k$ from $\cO(n)$ to $\cO(d)$, we put more trust in the value of $h^k$ with the benefit of reducing the variance of $g^k$. This insight points out that a practical implementation of the algorithm may exploit the fact that $h^k$ learns the gradient of $f$ by using smaller $\theta_k$.

It is also worth noting that \texttt{SEGA} is a stationary point algorithm regardless of the value of $\theta_k$. Indeed, if one has $x^k = x^*$ and $h^k = \nabla f(x^*)$, then $g^k = \nabla f(x^*)$ for any $\theta_k$. Therefore, once we get a reasonable $h^k$, it is well grounded to choose $g^k$ to be closer to $h^k$. This argument is also supported by our experiments.

Finally, the ability to take bigger stepsizes is also of high interest. One can think of extending other methods in this direction, especially if interested in applications with a small rank of matrix~$\mA$.

\section{Simplified Analysis of \texttt{SEGA} 1 \label{sec:simple_SEGA}}

In this section we consider the setup from Example~\ref{ex:coord_setup} with $\mB=\mI$ uniform probabilities: $p_i=1/n$ for all $i$. We now state the main complexity result.

\begin{theorem} \label{thm:simple} Let $\mB=\mI$ and choose $\cD$ to be the uniform distribution over unit basis vectors in $\R^n$. Choose $\sigma>0$ and define \[\Lgen^k \eqdef \|x^k-x^*\|_2^2 + \sigma \alpha \|h^k\|_2^2,\] where $\{x^k,h^k\}_{k\geq 0}$ are the iterates of the  gradient sketch  method.  If the stepsize  satisfies 
\begin{equation}\label{eq:alpha_bound} 0<\alpha \leq \min\left\{ \frac{1-\frac{L\sigma}{n}}{2Ln}, \frac{1}{n\left(\mu + \tfrac{2(n-1)}{\sigma}\right)} \right\},\end{equation}
then
$\ED{\Lgen^{k+1}} \leq (1-\alpha \mu) \Lgen^{k}.$
This means that 
\[k \geq \frac{1}{\alpha \mu} \log \frac{1}{\epsilon} \quad \Rightarrow \quad \E{\Lgen^k} \leq \epsilon \Lgen^0.\]

\end{theorem}

In particular, if we let $\sigma = \frac{n}{2L}$, then  $\alpha = \tfrac{1}{(4L+\mu)n}$ satisfies \eqref{eq:alpha_bound}, and we have the iteration complexity \[n\left(4 + \frac{1}{\kappa}\right)\kappa \log \frac{1}{\epsilon} = \tilde{\cO}(n\kappa),\] where $\kappa\eqdef  \tfrac{L}{\mu}$ is the condition number.

This is the same complexity as \texttt{NSync} \cite{NSync} under the same assumptions on $f$. \texttt{NSync} also needs just access to partial derivatives. However, \texttt{NSync} uses variable stepsizes, while \texttt{SEGA} can do the same with  {\em fixed} stepsizes. This is because \texttt{SEGA} {\em learns} the direction $g^k$ using  past information.

\subsection{Technical Lemmas}

Since $f$ is $L$--smooth, we have
\begin{equation}\label{eq:L-smooth_inequality}
\|\nabla f(x^k)\|_{2}^2 \leq 2L(f(x^k) - f(x^*)).
\end{equation}
On the other hand, by $\mu$--strong convexity of $f$ we have \begin{equation}\label{eq:8998sgjfbif}f(x^*) \geq f(x^k) + \langle \nabla f(x^k), x^*-x^k \rangle + \frac{\mu}{2}\|x^*-x^k\|_{2}^2.\end{equation}

\begin{lemma} \label{lem:1} The variance of $g^k$ as an estimator of $\nabla f(x^k)$ can be bounded as follows:
\begin{equation} \label{eq:bu9808hf9}
\ED{\|g^k\|_{2}^2} \leq 4Ln (f(x^k) - f(x^*)) + 2(n-1)\|h^k \|_{2}^2.
\end{equation}
\end{lemma}

\begin{proof}  
In view of \eqref{eq:8h0h09ffs}, we first write \[g^k = \underbrace{h^k - \frac{1}{p_i}  e_{i}^\top  h^k e_{i}}_{a} +  \underbrace{ \frac{1}{p_i} e_{i}^\top \nabla f(x^k)  e_{i}}_{b},\]
and note that $p_i=1/n$ for all $i$. Let us bound the expectation of each term individually. The first term is equal to
\begin{eqnarray*}
\ED{\|a\|_{2}^2} &=& \ED{\left\|h^k - n e_{i}^\top  h^k e_{i}\right\|_{2}^2}\\
&=& \ED{\left\| \left(\mI - n e_{i} e_{i}^\top \right) h^k  \right\|_{2}^2}\\
&=&(h^k)^\top \ED{ \left(\mI - n e_{i} e_{i}^\top \right)^\top \left(\mI - n e_{i} e_{i}^\top \right)} h^k\\
&=& (n-1)\|h^k\|_{2}^2.
\end{eqnarray*}
The second term can be bounded as
\begin{eqnarray*}
\ED{\|b\|_{2}^2} &=&  \ED{\left\| n  e_{i}^\top \nabla f(x^k)  e_{i} \right\|_{2}^2}\\
&=& n^2 \sum_{i=1}^n \frac{1}{n} (e_i^\top \nabla f(x^k))^2 \\
&=& n \|\nabla f(x^k)\|_2^2 \\
&=& n \|\nabla f(x^k) - \nabla f(x^*)\|_2^2 \\
&\overset{\eqref{eq:L-smooth_inequality}}{\leq} & 2Ln (f(x^k) - f(x^*)),
\end{eqnarray*}
where in the last step we used $L$--smoothness of $f$. It remains to combine the two bounds.

\end{proof}

\begin{lemma} \label{lem:2_easy} For all $v\in \R^n$ we have
\begin{equation}\label{eq:h_decomp_easy}
\ED{\|h^{k+1} \|_2^2} = \left( 1-\frac{1}{n}\right) \|h^k \|_2^2 +\frac{1}{n} \|\nabla f(x^k) - v\|_2^2.
\end{equation}
\end{lemma}
\begin{proof}
We have
\begin{eqnarray*}
\ED{\| h^{k+1}\|_2^2} &\overset{\eqref{eq:988fgf}}{=}& \ED{ \left\|h^k + e_{i_k}^\top (\nabla f(x^k) - h^k) e_{i_k} \right\|_2^2 }\\
 &=&  \ED{ \left\| \left(\mI - e_{i_k} e_{i_k}^\top \right)h^k + e_{i_k} e_{i_k}^\top \nabla f(x^k)  \right\|_2^2 }\\
 &=& \ED{ \left\| \left(\mI - e_{i_k} e_{i_k}^\top \right)h^k \right\|_2^2} + \ED{\left\|e_{i_k} e_{i_k}^\top \nabla f(x^k)  \right\|_2^2 }\\
 &=& (h^k)^\top \ED{\left(\mI - e_{i_k} e_{i_k}^\top \right)^\top \left(\mI - e_{i_k} e_{i_k}^\top \right)} h^k  (\nabla f(x^k))^\top \ED{(e_{i_k} e_{i_k}^\top)^\top e_{i_k} e_{i_k}^\top} \nabla f(x^k)\\
 &=&(h^k)^\top \ED{\mI - e_{i_k} e_{i_k}^\top } h^k + (\nabla f(x^k))^\top \ED{e_{i_k} e_{i_k}^\top} \nabla f(x^k)\\
 &=& \left(1-\frac{1}{n}\right) \|h^k\|_2^2 + \frac{1}{n} \|\nabla f(x^k)\|_2^2.
\end{eqnarray*}

\end{proof}

\subsection{Proof of Theorem~\ref{thm:simple}}

 We can now write
\begin{eqnarray*}
\ED{\|x^{k+1}-x^*\|_{2}^2} &=& \ED{\|x^k - \alpha g^k - x^*\|_{2}^2}\\
&=& \|x^k -x^*\|_2^2 + \alpha^2 \ED{\|g^k\|_2^2} - 2\alpha \langle \ED{g^k}, x^k - x^* \rangle \\
&\overset{\eqref{eq:unbiased_estimator}}{=}& \|x^k -x^*\|_2^2 + \alpha^2 \ED{\|g^k\|_2^2} - 2\alpha \langle \nabla f(x^k), x^k - x^* \rangle\\
&\overset{\eqref{eq:8998sgjfbif}}{\leq} & (1-\alpha \mu)\|x^k -x^*\|_2^2  + \alpha^2 \ED{\|g^k\|_2^2} - 2\alpha (f(x^k)-f(x^*)).
\end{eqnarray*}
Using Lemma~\ref{lem:1}, we can further estimate 
\begin{eqnarray*}
\ED{\|x^{k+1}-x^*\|_2^2} &\leq & (1-\alpha \mu)\|x^k -x^*\|_2^2  \\
&&\qquad  + 2\alpha(2Ln \alpha -1) (f(x^k) - f(x^*)) + 2(n-1)\alpha^2\|h^k \|_2^2 .
\end{eqnarray*}
Let us now add $\sigma \alpha \ED{ \|h^{k+1}\|_2^2}$ to both sides of the last inequality. Recalling the definition of the Lyapunov function, and applying Lemma~\ref{lem:2}, we get
\begin{eqnarray*}
\ED{\Lgen^{k+1}} &\leq & (1-\alpha \mu)\|x^k -x^*\|_2^2  + 2\alpha(2Ln \alpha -1) (f(x^k) - f(x^*)) + 2(n-1)\alpha^2\|h^k \|_2^2  \\
&& \quad + \sigma \alpha \left(1-\frac{1}{n} \right) \|h^k\|_2^2 +  \frac{\sigma \alpha}{n}\|\nabla f(x^k)\|_2^2\\
&\overset{\eqref{eq:L-smooth_inequality}}{\leq}& (1-\alpha \mu)\|x^k -x^*\|_2^2 +   2\alpha\underbrace{\left(2Ln \alpha + \frac{L\sigma}{n} -1\right)}_{\text{I}} (f(x^k) - f(x^*)) \\
&&\qquad  +\underbrace{ \left( 1-\frac{1}{n} + \frac{2(n-1)\alpha}{\sigma}\right)}_{\text{II}} \sigma \alpha \|h^k\|_2^2.
\end{eqnarray*}
Let us choose $\alpha$ so that $\text{I} \leq 0$ and $\text{II} \leq 1-\alpha \mu$. This leads to the bound \eqref{eq:alpha_bound}. For any $\alpha > 0$ satisfying this bound we therefore have
$
\ED{\Lgen^{k+1}} \leq  (1-\alpha \mu) \Lgen^k,
$
as desired. Lastly, as we have freedom to choose $\sigma$, let us pick it so as to maximize the upper bound on the stepsize.

\section{Simplified  Analysis of \texttt{SEGA} II} \label{sec:analysis-samenorm}

In this section we consider the setup from Example~\ref{ex:coord_setup} with arbitrary non-uniform probabilities: $p_i>0$ for all $i$. We provide a simplified analysis of \texttt{SEGA} in this scenario. However, we will do this under slightly different assumptions. In particular, we shall assume that smoothness and strong convexity of $f$ are measured with respect to the same norm.

In this setup, as we shall see, uniform probabilities are optimal. That is, uniform probabilities are identical to the importance sampling probabilities. We note that this would be the case even for standard coordinate descent under these assumptions, as follows from the results in \cite{NSync}. 

Let $\mG=\diag(g_1,\dots,g_n)\succ 0$ and assume that
\[\|\nabla f(x) - \nabla f(y)\|_{\mG^{-1}} \leq L \|x-y\|_{\mG}\]
and\footnote{Note that in the strong convexity inequality below the scalar product is without any additional metric unlike in other sections.}
\[ f(x) \geq f(y) + \langle \nabla f(y), x-y \rangle + \frac{\mu}{2} \|x-y\|_{\mG}^2\]
for all $ x,y\in \R^n$. These two assumptions combined lead to the following inequalities:
\[f(y) + \langle \nabla f(y), x-y \rangle + \frac{\mu}{2} \|x-y\|_{\mG}^2  \leq f(x) \leq f(y) + \langle \nabla f(y), x-y \rangle + \frac{L}{2} \|x-y\|_{\mG}^2 .\]

We define $g^k$ as before, but change the method to:
 \begin{equation} \label{eq:method_G}\boxed{x^{k+1} = x^k - \alpha \mG^{-1} g^k} \end{equation}

We now state the main complexity result.

\begin{theorem} \label{thm:simple999}
Choose $\sigma>0$ and define $\Lgen^k \eqdef \|x^k-x^*\|_{\mG}^2 + \sigma \alpha \|h^k\|_{\diag\left(\frac{1}{g_i p_i}\right)}^2$, where $\{x^k,h^k\}_{k\geq 0}$ are the iterates of the  gradient sketch  method.  If the stepsize  satisfies 
\begin{equation}\label{eq:alpha_bound99} 0<\alpha \leq \min_i \left\{ p_i\left(\frac{1}{\mu+L} - \frac{\sigma}{2}\right), \frac{p_i}{\frac{2}{\sigma}(1-p_i) + \frac{2L\mu}{\mu+L}}\right\},\end{equation}
then
$\ED{\Lgen^{k+1}} \leq \left(1-\alpha \mu \frac{2L}{\mu+L}\right) \Lgen^{k}.$
This means that 
\[k \geq \frac{L+\mu}{2\alpha L \mu} \log \frac{1}{\epsilon} \quad \Rightarrow \quad \E{\Lgen^k} \leq \epsilon \Lgen^0.\]
In particular, if we choose $g_i=1$ and $p_i=\tfrac{1}{n}$ for all $i$, then if we set $\sigma=\tfrac{1}{2L}$, we can choose stepsize $\alpha=\frac{3L-\mu}{4Ln(L+\mu)}$, and obtain the rate  $\frac{2L+2\mu}{3L-\mu}n\left(\frac{L}{\mu}+1\right) \log \frac{1}{\epsilon} \leq 2n\left(\frac{L}{\mu}+1\right) \log \frac{1}{\epsilon}$.

\end{theorem}


\subsection{Two lemmas}

\begin{lemma} \label{lem:1xxx} Let $d_1,\dots,d_n>0$. The variance of $g^k$ as an estimator of $\nabla f(x^k)$ can be bounded as follows:
\begin{equation} \label{eq:bu9808hf9xxx}
\ED{\|g^k\|_{\diag(d_i)}^2} \leq  2\|h^k \|_{\diag\left(d_i\frac{1-p_i}{p_i} \right)}^2 + 2 \|\nabla f(x^k)\|^2_{\diag\left(\frac{d_i}{p_i} \right)}.
\end{equation}
\end{lemma}

\begin{proof}  
In view of \eqref{eq:8h0h09ffs}, we first write \[g^k = \underbrace{h^k - \frac{1}{p_i}  e_{i}^\top  h^k e_{i}}_{a} +  \underbrace{ \frac{1}{p_i} e_{i}^\top \nabla f(x^k)  e_i}_{b}.\]
Let us bound the expectation of each term individually. The first term is equal to
\begin{eqnarray*}
\ED{\|a\|_{\mG^{-1}}^2} &=& \ED{\left\|h^k - \frac{1}{p_i} e_{i}^\top  h^k e_{i}\right\|_{\diag(d_i)}^2}\\
&=& \ED{\left\| \left(\mI - \frac{1}{p_i} e_{i} e_{i}^\top \right) h^k  \right\|_{\diag(d_i)}^2}\\
&=&(h^k)^\top \ED{ \left(\mI - \frac{1}{p_i} e_{i} e_{i}^\top \right)^\top \diag(d_i) \left(\mI - \frac{1}{p_i} e_{i} e_{i}^\top \right)} h^k\\
&=&(h^k)^\top  \ED{\left(\diag(d_i) - \frac{2d_i}{p_i } e_{i} e_{i}^\top + \frac{d_i}{p_i^2 }e_i e_i^\top\right)}h^k \\
&=& \sum_{i=1}^n d_i\left( \frac{1}{p_i} - 1\right) (h^k_i)^2.
\end{eqnarray*}
The second term can be bounded as
\[
\ED{\|b\|_{\diag(d_i)}^2} =  \ED{\left\| \frac{1}{p_i}  e_{i}^\top \nabla f(x^k)  e_{i} \right\|_{\diag(d_i)}^2}
= \sum_{i=1}^n \frac{d_i}{p_i} (e_i^\top \nabla f(x^k))^2.
\] It remains to combine the two bounds. \qed

\end{proof}

\begin{lemma} \label{lem:2xxx} For all $v\in \R^n$ and $d_1,\dots,d_n>0$ we have
\begin{equation}\label{eq:xnaosa}
\ED{\|h^{k+1} - v\|_{\diag(d_i)}^2} = \|h^k - v\|_{\diag(d_i(1-p_i))}^2 + \|\nabla f(x^k) - v\|_{\diag(d_i p_i)}^2.
\end{equation}
\end{lemma}

\begin{proof}
We have
\begin{eqnarray*}
\ED{\| h^{k+1}-v\|_{\diag(d_i)}^2} &\overset{\eqref{eq:988fgf}}{=}& \ED{ \left\|h^k + e_{i}^\top (\nabla f(x^k) - h^k) e_{i} - v \right\|_{\diag(d_i)}^2 }\\
 &=&  \ED{ \left\| \left(\mI - e_{i} e_{i}^\top \right)(h^k - v) + e_{i} e_{i}^\top (\nabla f(x^k)-v)  \right\|_{\diag(d_i)}^2 }\\
 &=& \ED{ \left\| \left(\mI - e_{i} e_{i}^\top \right)(h^k-v) \right\|_{\diag(d_i)}^2} + \ED{\left\|e_{i} e_{i}^\top (\nabla f(x^k) - v)  \right\|_{\diag(d_i)}^2 }\\
 &=& (h^k-v)^\top \ED{\left(\mI - e_{i} e_{i}^\top \right)^\top \diag(d_i)\left(\mI - e_{i} e_{i}^\top \right)} (h^k-v) \\
 && \qquad + (\nabla f(x^k)-v)^\top \ED{(e_{i} e_{i}^\top)^\top \diag(d_i)e_{i} e_{i}^\top} (\nabla f(x^k) -v)\\
 &=&(h^k-v)^\top \ED{\diag(d_i) - d_i e_{i} e_{i}^\top } (h^k - v) \\
 && \qquad + (\nabla f(x^k)-v)^\top \ED{d_i e_{i} e_{i}^\top} (\nabla f(x^k) - v)\\
 &=&  \|h^k - v\|_{\diag(d_i(1-p_i))}^2 +  \|\nabla f(x^k) - v\|_{\diag(d_i p_i)}^2.
\end{eqnarray*}
\qed
\end{proof}

\subsection{Proof of Theorem~\ref{thm:simple}}

\begin{proof}
Since $f$ is $L$--smooth and  $\mu$--strongly convex, we  have the inequality
\[ \langle \nabla f(x) - \nabla f(y), x - y\rangle \geq \frac{\mu L}{\mu + L} \|x-y\|_{\mG}^2 + \frac{1}{\mu+L} \|\nabla f(x) - \nabla f(y)\|_{\mG^{-1}}^2.\]
In particular, we will use it for $x=x^k$ and $y=x^*$:
\begin{equation}\label{eq:bun09s88yh08s} \langle \nabla f(x^k)  , x^* - x^k\rangle \leq - \frac{\mu L}{\mu + L} \|x-x^*\|_{\mG}^2 - \frac{1}{\mu+L} \|\nabla f(x^k) \|_{\mG^{-1}}^2.\end{equation}
We can now write
\begin{eqnarray*}
\ED{\|x^{k+1}-x^*\|_{\mG}^2} &\overset{\eqref{eq:method_G}}{=}& \ED{\|x^k - \alpha \mG^{-1} g^k - x^*\|_{\mG}^2}\\
&=& \|x^k -x^*\|_{\mG}^2 + \alpha^2 \ED{\| \mG^{-1} g^k\|_{\mG}^2} - 2\alpha \langle \ED{ g^k}, x^k - x^* \rangle \\
&\overset{\eqref{eq:unbiased_estimator}}{=}& \|x^k -x^*\|_{\mG}^2 + \alpha^2 \ED{\|g^k\|_{\mG^{-1}}^2} + 2\alpha \langle \nabla f(x^k), x^* - x^k \rangle\\
&\overset{\eqref{eq:bun09s88yh08s}}{\leq} & \left(1-\alpha \mu \tfrac{2L}{\mu+L}\right)\|x^k -x^*\|_{\mG}^2  + \alpha^2 \ED{\|g^k\|_{\mG^{-1}}^2} - \tfrac{2\alpha}{\mu+L} \|\nabla f(x^k)\|_{\mG^{-1}}^2.
\end{eqnarray*}
Using Lemma~\ref{lem:1xxx} to bound $\ED{\|g^k\|_{\mG^{-1}}^2}$, we can further estimate 
\begin{eqnarray*}
\ED{\|x^{k+1}-x^*\|_{\mG}^2} &\leq & \left(1-\alpha \mu \tfrac{2L}{\mu+L}\right)\|x^k -x^*\|_{\mG}^2   + 2\alpha^2 \|\nabla f(x^k)\|^2_{\diag\left(\tfrac{1}{p_i g_i}\right)} \\
&& \qquad - \tfrac{2\alpha}{\mu+L} \|\nabla f(x^k)\|_{\mG^{-1}}^2  + 2\alpha^2 \|h^k \|_{\diag\left(\tfrac{1-p_i}{p_i g_i}\right)}^2.
\end{eqnarray*}
Let us now add $\sigma \alpha \ED{ \|h^{k+1}\|_{\diag\left(\tfrac{1}{g_i p_i}\right)}^2}$ to both sides of the last inequality. Recalling the definition of the Lyapunov function, and applying Lemma~\ref{lem:2xxx} with $v=0$ and $d_i=\tfrac{1}{g_ip_i }$, we get

\begin{eqnarray*}
\ED{\Lgen^{k+1}} &\leq & \left(1-\alpha \mu \tfrac{2L}{\mu+L}\right)\|x^k -x^*\|_{\mG}^2   + 2\alpha^2 \|\nabla f(x^k)\|^2_{\diag\left(\frac{1}{p_i g_i}\right)} + \left( \alpha \sigma - \tfrac{2\alpha}{\mu+L}\right) \|\nabla f(x^k)\|_{\mG^{-1}}^2 \\
&& \qquad + (2\alpha^2+\alpha \sigma ) \|h^k \|_{\diag\left(\frac{1-p_i}{p_i g_i}\right)}^2\\
&\leq & \left(1-\alpha \mu \tfrac{2L}{\mu+L}\right)\|x^k -x^*\|_{\mG}^2   +  \sigma \alpha \|h^k \|_{\diag\left(\left(\frac{2\alpha}{\sigma}+1 \right)\frac{1-p_i}{p_i g_i}\right)}^2 \\
&&\qquad  + \|\nabla f(x^k)\|^2_{\diag\left(\frac{2\alpha^2 }{p_i g_i} +\frac{\sigma \alpha}{g_i} -\frac{2\alpha}{(\mu+L)g_i}\right)} .
\end{eqnarray*}
If we now choose $\alpha>0$ such that
\[\frac{2\alpha }{p_i } +\sigma  -\frac{2}{\mu+L} \leq 0, \qquad \left(\frac{2\alpha}{\sigma}+1 \right)(1-p_i) \leq 1-\alpha \mu \frac{2L}{\mu+L},\]
then 
we get the recursion
\[\ED{\Lgen^{k+1}} \leq  \left(1-\alpha \mu \tfrac{2L}{\mu+L}\right) \Lgen^k \le (1 - \alpha\mu)\Lgen^k. \]
\qed
\end{proof}

\section{Extra Experiments} \label{sec:extra_exp}

\subsection{Evolution of Iterates: Extra Plots\label{sec:evolution_extra}}

Here we show some additional plots similar to Figure~\ref{fig:trajectory}, which we believe help to build intuition about how the iterates of  \texttt{SEGA} behave. We also include plots for \texttt{biasSEGA}, which uses biased estimators of the gradient instead. We found that the iterates of \texttt{biasSEGA} often behave in a more stable way, as could be expected given the fact that they enjoy lower variance. However, we do not have any theory supporting the convergence of \texttt{biasSEGA}; this is left for future research.

\begin{figure}[!h]
\centering
\begin{minipage}{0.35\textwidth}
\centering
\includegraphics[width = \textwidth ]{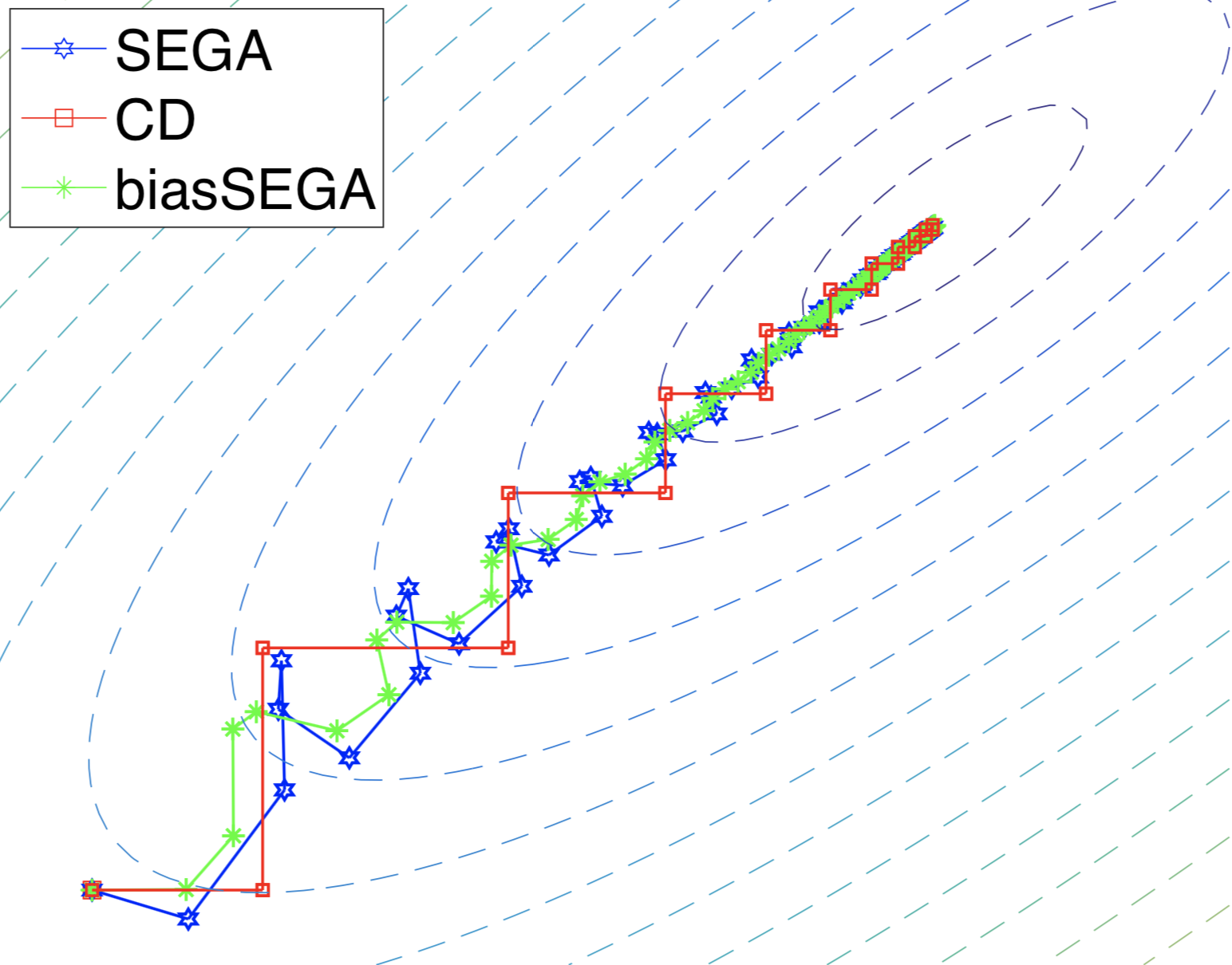}
\caption{Evolution of iterates of \texttt{SEGA}, \texttt{CD} and \texttt{biasSEGA} (updates made via $h^{k+1}$ instead of $g^k$).}
\end{minipage}
\hskip 1cm
\begin{minipage}{0.35\textwidth}
\centering
\includegraphics[width = \textwidth ]{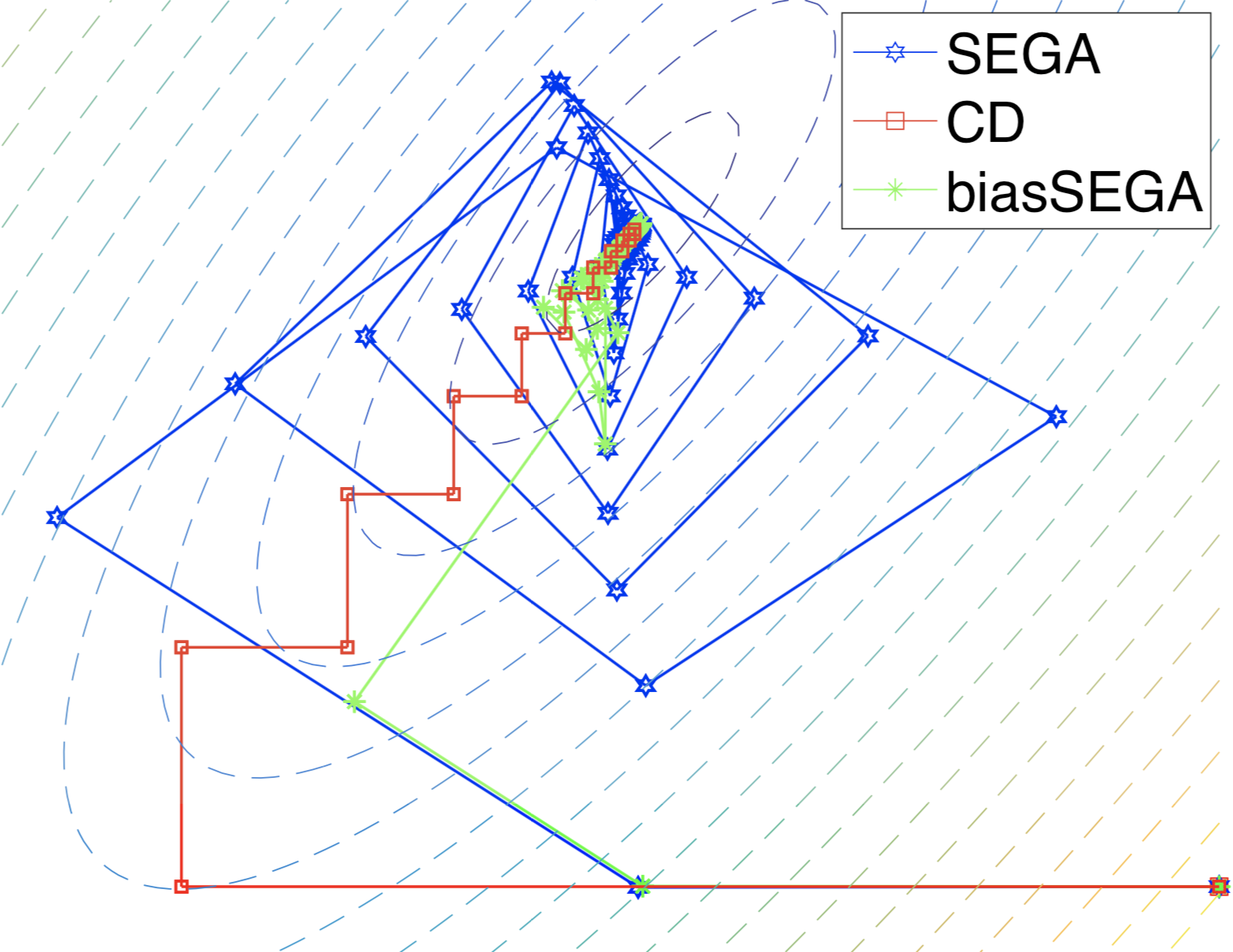}
\caption{Iterates of \texttt{SEGA}, \texttt{CD} and \texttt{biasSEGA} (updates made via $h^{k+1}$ instead of $g^k$). Different starting point.}
\end{minipage}
\\
\hskip 1cm
\begin{minipage}{0.35\textwidth}
\centering
\includegraphics[width = \textwidth ]{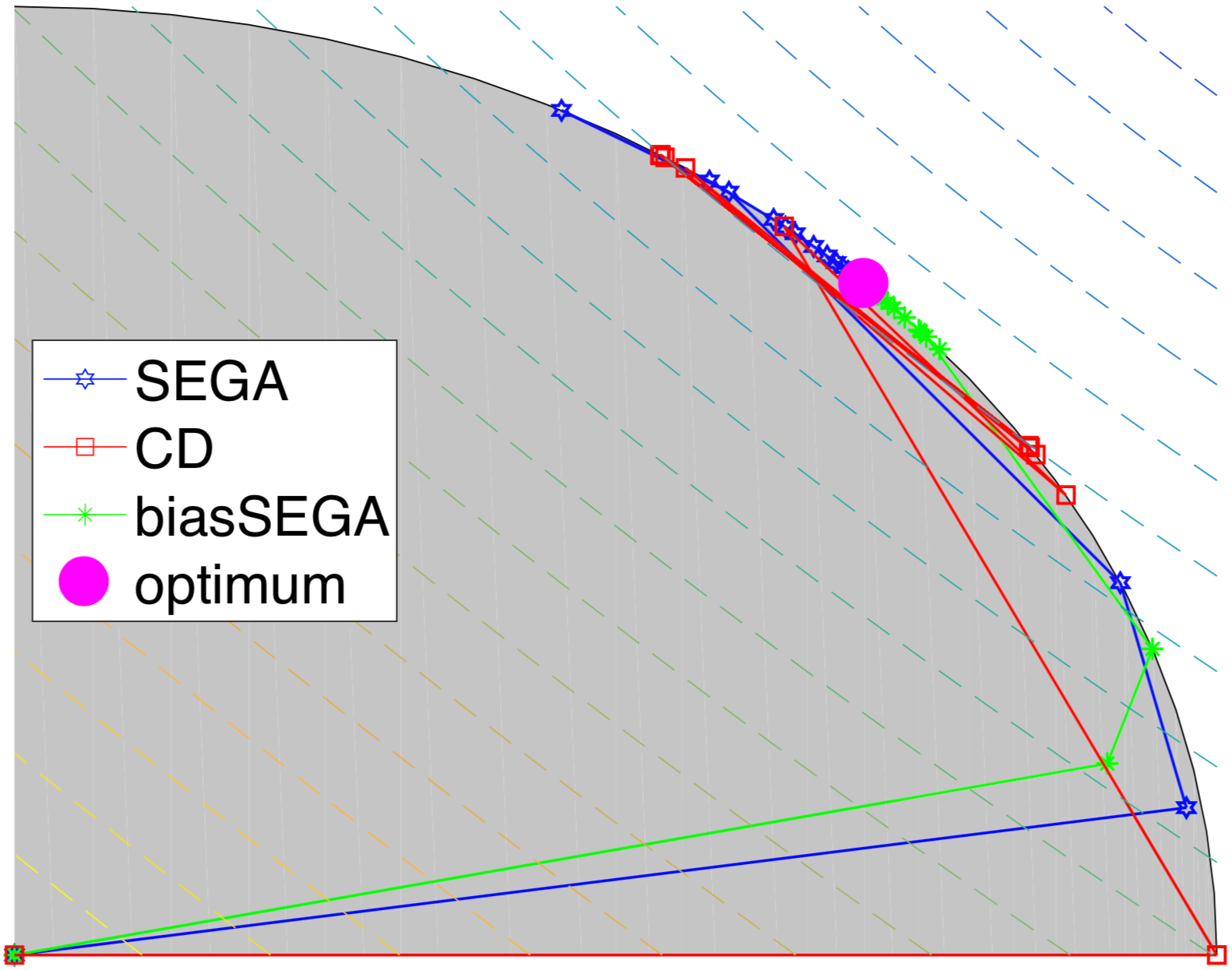}
\caption{Iterates of projected \texttt{SEGA}, projected \texttt{CD} (which do not converge) and projected \texttt{biasSEGA} (updates made via $h^{k+1}$ instead of $g^k$). The constraint set is represented by the shaded region.}
\end{minipage}
\end{figure}

\clearpage

\subsection{Experiments from Section~\ref{sec:experiments} with empirically optimal stepsize \label{sec:exp_optimal}}
In the experiments in Section~\ref{sec:experiments}, we worked with quadratic functions of the form \[f(x)\eqdef \frac12 x^\top \mM x-b^\top x,\] where $b$ is a random vector with independent entries from $\cN(0,1)$ and $\mM\eqdef\mU\Sigma \mU^\top$ according to Table~\ref{tab:problem} for $\mU$ obtained from QR decomposition of random matrix with independent entries from $\cN(0,1)$. For each problem, the starting point was chosen to be a vector with independent entries from $\cN(0,1)$. 

\begin{table}[!h]
\centering
\small
\begin{tabular}{|c|c|}
\hline
Type  &   $\Sigma$  \\
\hline
\hline
1 &  Diagonal matrix with first  $n/2$ components equal to 1 and the rest equal to $n$\\ \hline
2&  Diagonal matrix with first  $n-1$ components equal to 1 and the remaining one equal to $n$\\ \hline
3&  Diagonal matrix with $i$--th component equal to $i$\\ \hline
4&   Diagonal matrix with components coming from uniform distribution over $[0,1]$\\ \hline
\end{tabular}
\caption{Spectrum of $\mM$.}
\label{tab:problem}
\end{table}

The results are provided in Figures~\ref{fig:pgd_comp2}-\ref{fig:aggressive2}. They include zeroth-order experiments and the subspace version of \texttt{SEGA}.

\begin{figure}[!h]
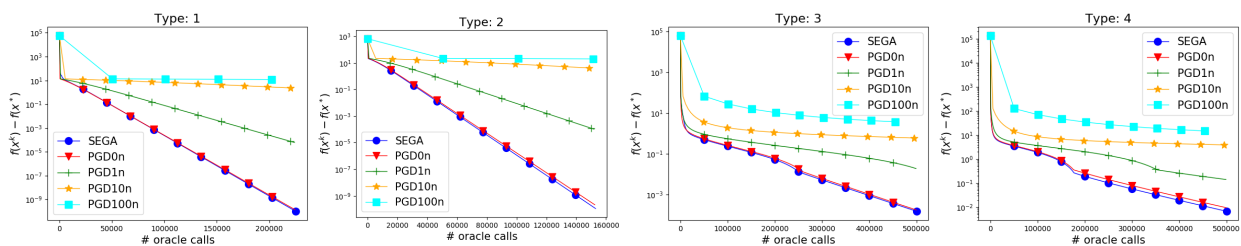

\centering
\begin{minipage}{0.25\textwidth}
  \centering
\includegraphics[width =  \textwidth ]{SEGAth_oracle_1.png}
\end{minipage}%
\begin{minipage}{0.25\textwidth}
  \centering
\includegraphics[width =  \textwidth ]{SEGAth_oracle_2.png}
\end{minipage}%
\begin{minipage}{0.25\textwidth}
  \centering
\includegraphics[width =  \textwidth ]{SEGAth_oracle_3.png}
\end{minipage}%
\begin{minipage}{0.25\textwidth}
  \centering
\includegraphics[width =  \textwidth ]{SEGAth_oracle_4.png}
\end{minipage}%
\caption{\footnotesize Counterpart to Figure~\ref{fig:pgd_comp} -- convergence illustration of \texttt{SEGA} and \texttt{PGD}. The indicator ``Xn'' in the label stands for the setting when the cost of solving linear system is $Xn$ times higher comparing to the oracle call. Recall that a linear system is solved after each $n$ oracle calls. Empirically best stepsizes were used both \texttt{PGD} and \texttt{SEGA}. }\label{fig:pgd_comp2}
\end{figure}

\begin{figure}[!h]
\begin{minipage}{0.25\textwidth}
  \centering
\includegraphics[width =  \textwidth ]{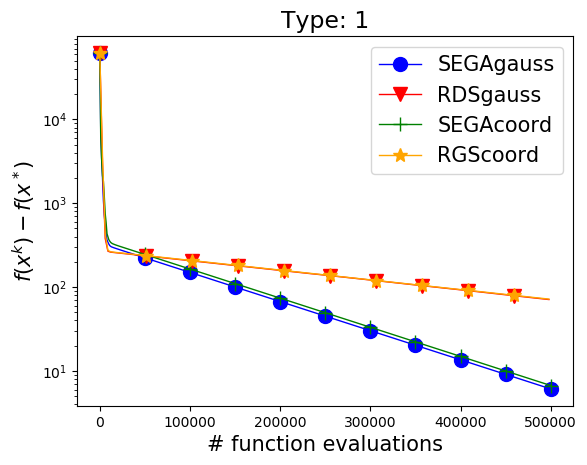}
\end{minipage}%
\begin{minipage}{0.25\textwidth}
  \centering
\includegraphics[width =  \textwidth ]{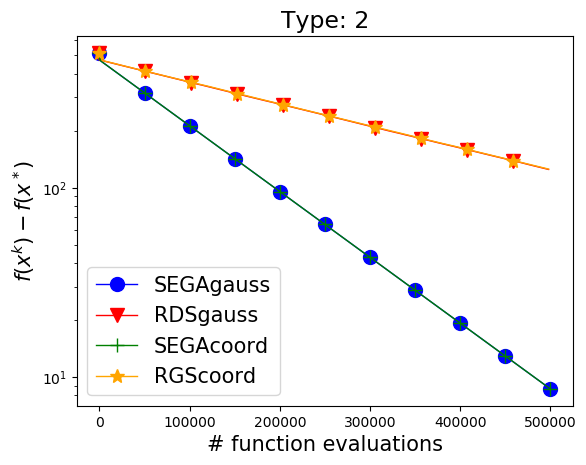}
\end{minipage}%
\begin{minipage}{0.25\textwidth}
  \centering
\includegraphics[width =  \textwidth ]{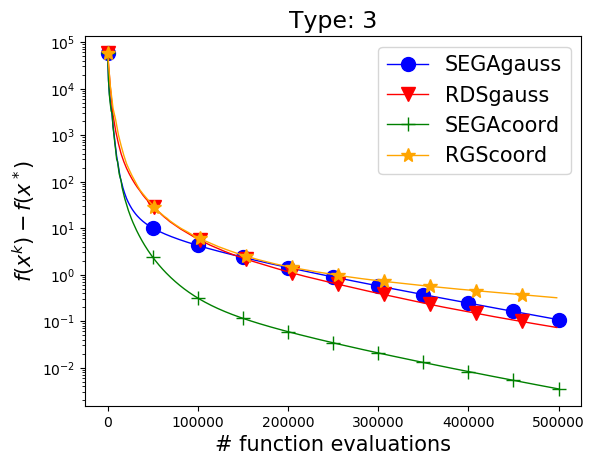}
\end{minipage}%
\begin{minipage}{0.25\textwidth}
  \centering
\includegraphics[width =  \textwidth ]{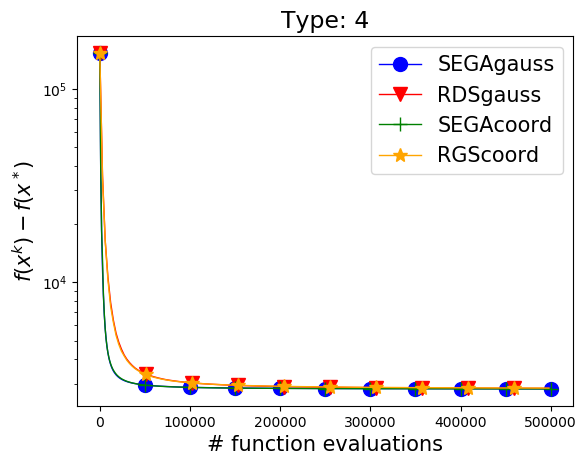}
\end{minipage}%
\caption{\footnotesize Counterpart to Figure~\ref{fig:DFO} -- comparison of \texttt{SEGA} and randomized direct search for a various problems. Empirically best stepsizes were used for both methods.}\label{fig:DFO2}
\end{figure}

\begin{figure}[!h]
\centering
\begin{minipage}{0.25\textwidth}
  \centering
\includegraphics[width =  \textwidth ]{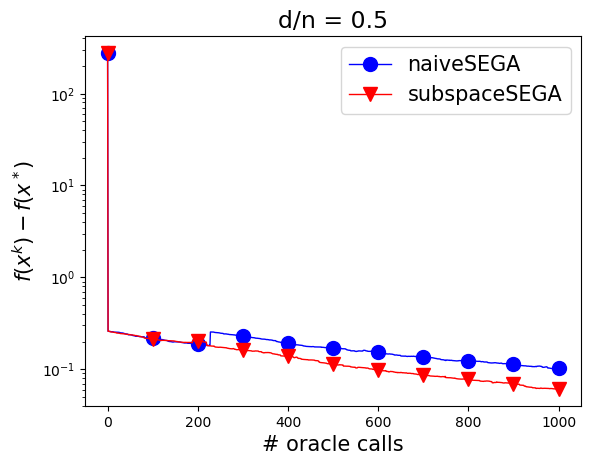}
\end{minipage}%
\begin{minipage}{0.25\textwidth}
  \centering
\includegraphics[width =  \textwidth ]{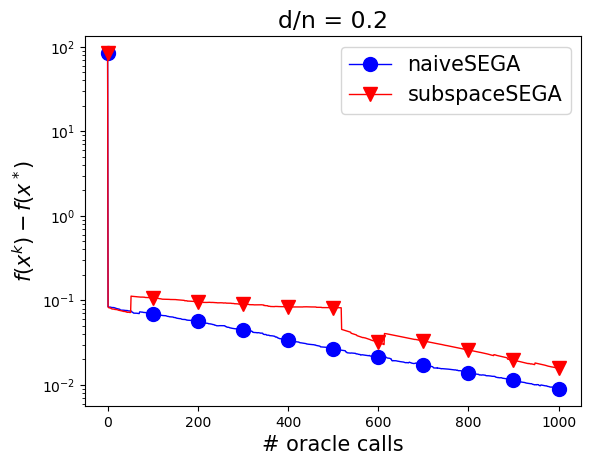}
\end{minipage}%
\begin{minipage}{0.25\textwidth}
  \centering
\includegraphics[width =  \textwidth ]{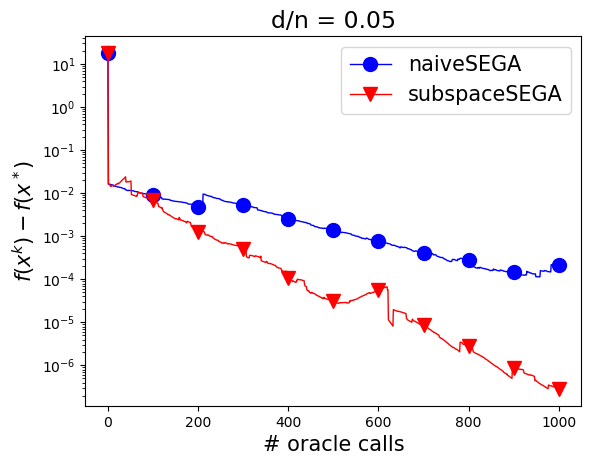}
\end{minipage}%
\begin{minipage}{0.25\textwidth}
  \centering
\includegraphics[width =  \textwidth ]{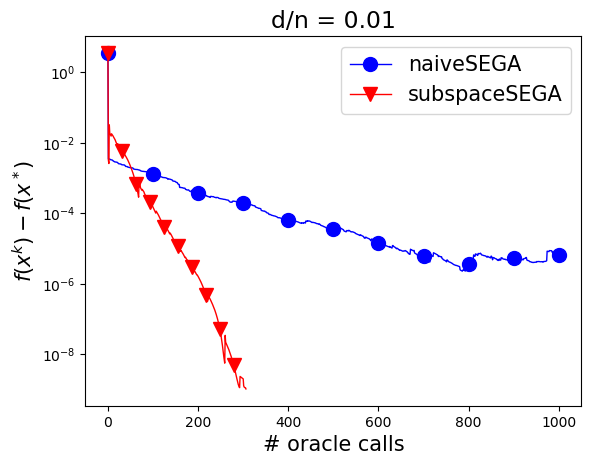}
\end{minipage}%
\caption{\footnotesize Counterpart to Figure~\ref{fig:aggressive} -- comparison of \texttt{SEGA} with sketches from a correct subspace versus naive \texttt{SEGA}. Optimal (empirically) stepsize chosen.}\label{fig:aggressive2}
\end{figure}

\subsection{Experiment: comparison with randomized coordinate descent}
In this section we numerically compare the results from Section~\ref{sec:CD} to analogous results for coordinate descent (as indicated in Table~\ref{tab:CDcmp}).  We consider the ridge regression problem on LibSVM~\cite{chang2011libsvm} data, for both primal and dual formulation. For all methods, we have chosen parameters as suggested from theory Figure~\ref{fig:cd_cmp} shows the results. We can see that in all cases, \texttt{SEGA} is slower to the corresponding coordinate descent method, but still is competitive. We however observe only constant times difference in terms of the speed, as suggested by Table~\ref{tab:CDcmp}.

\begin{figure}[!h]
\centering
\begin{minipage}{0.25\textwidth}
  \centering
\includegraphics[width =  \textwidth ]{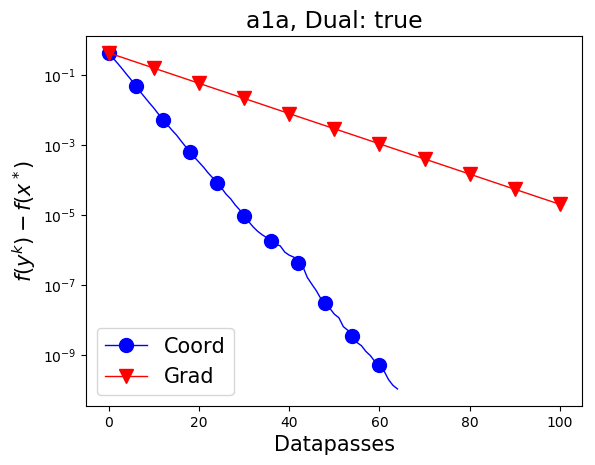}
\end{minipage}%
\begin{minipage}{0.25\textwidth}
  \centering
\includegraphics[width =  \textwidth ]{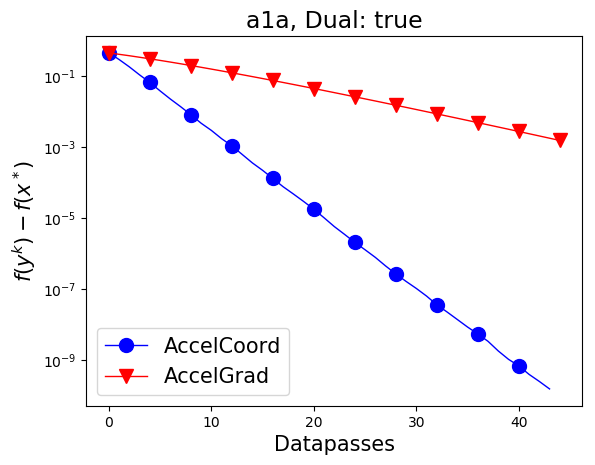}
\end{minipage}%
\begin{minipage}{0.25\textwidth}
  \centering
\includegraphics[width =  \textwidth ]{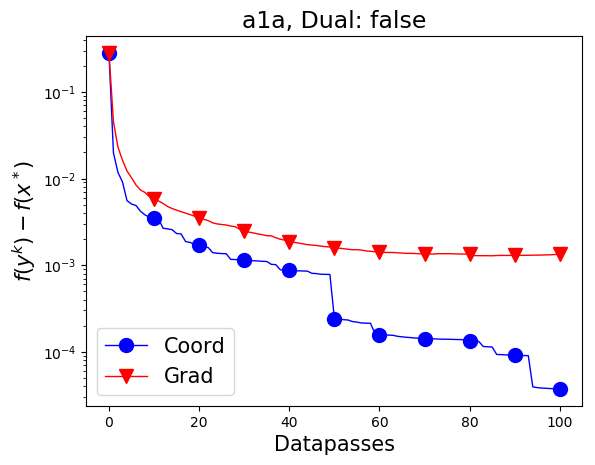}
\end{minipage}%
\begin{minipage}{0.25\textwidth}
  \centering
\includegraphics[width =  \textwidth ]{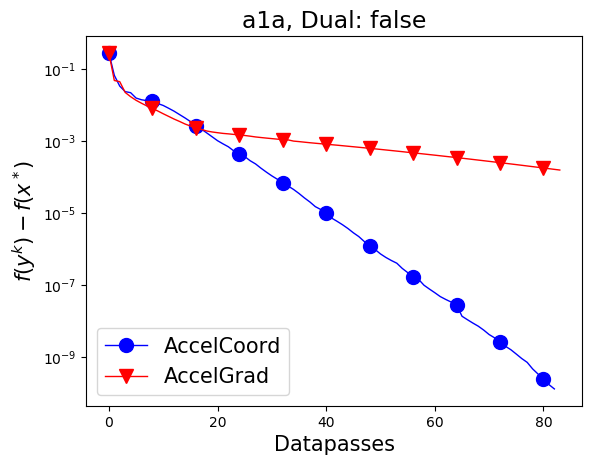}
\end{minipage}%
\\
\begin{minipage}{0.25\textwidth}
  \centering
\includegraphics[width =  \textwidth ]{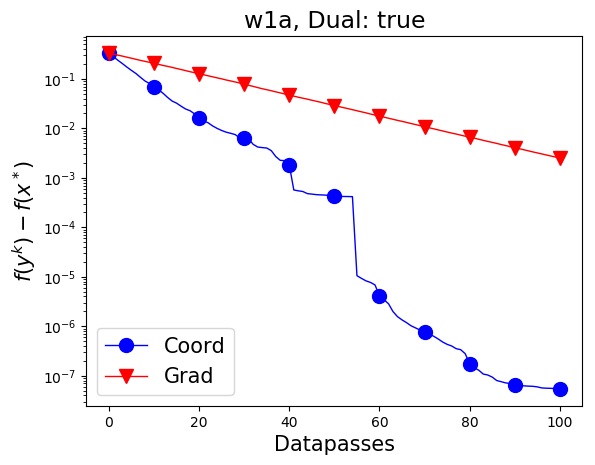}
\end{minipage}%
\begin{minipage}{0.25\textwidth}
  \centering
\includegraphics[width =  \textwidth ]{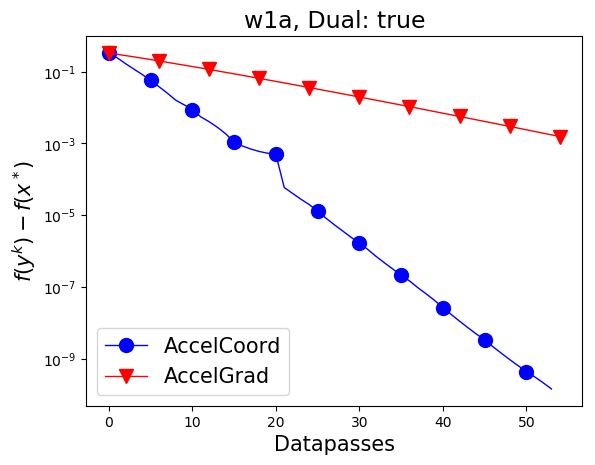}
\end{minipage}%
\begin{minipage}{0.25\textwidth}
  \centering
\includegraphics[width =  \textwidth ]{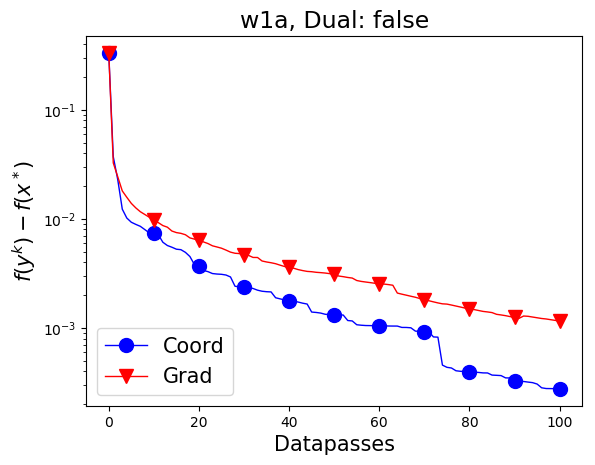}
\end{minipage}%
\begin{minipage}{0.25\textwidth}
  \centering
\includegraphics[width =  \textwidth ]{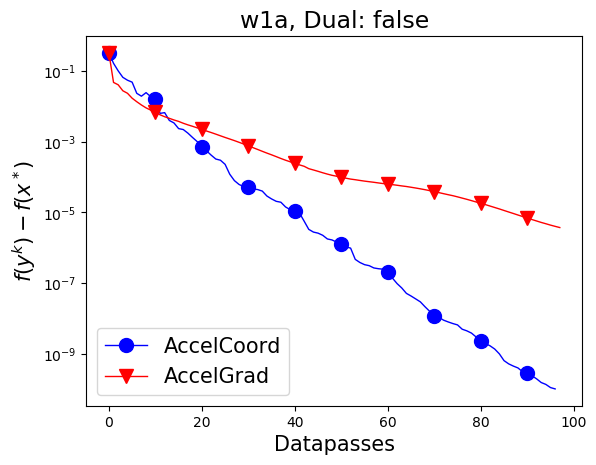}
\end{minipage}%
\\
\begin{minipage}{0.25\textwidth}
  \centering
\includegraphics[width =  \textwidth ]{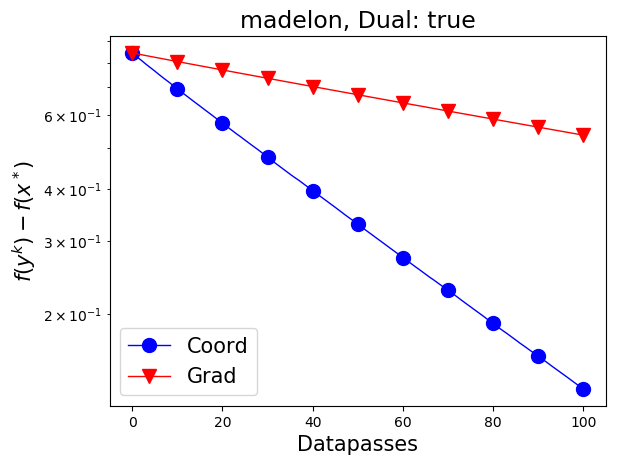}
\end{minipage}%
\begin{minipage}{0.25\textwidth}
  \centering
\includegraphics[width =  \textwidth ]{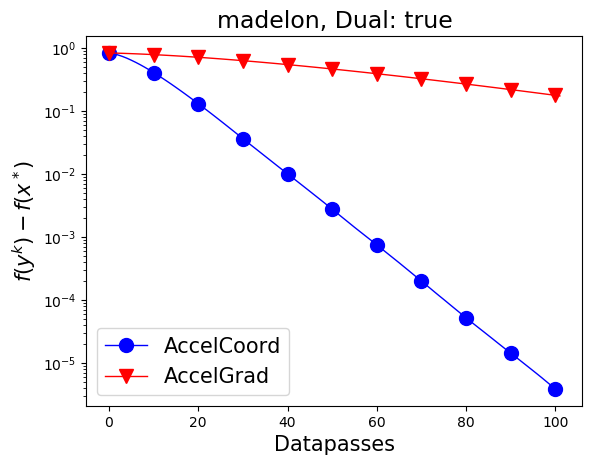}
\end{minipage}%
\begin{minipage}{0.25\textwidth}
  \centering
\includegraphics[width =  \textwidth ]{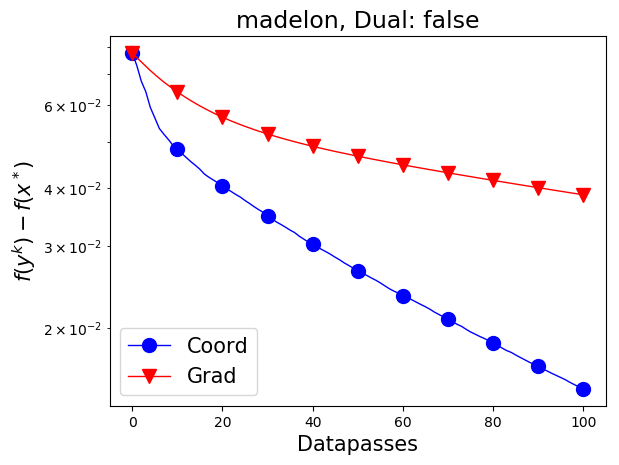}
\end{minipage}%
\begin{minipage}{0.25\textwidth}
  \centering
\includegraphics[width =  \textwidth ]{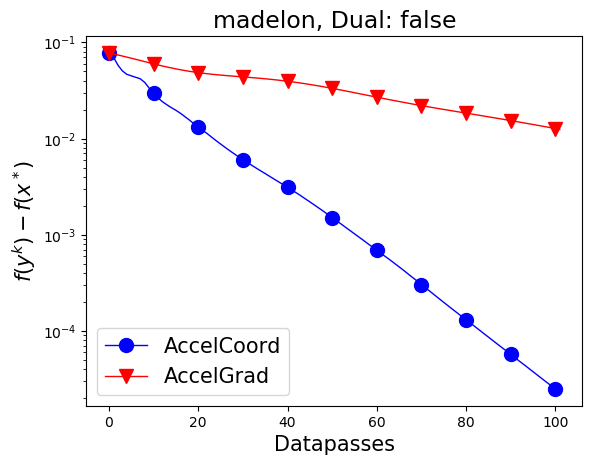}
\end{minipage}%
\\
\begin{minipage}{0.25\textwidth}
  \centering
\includegraphics[width =  \textwidth ]{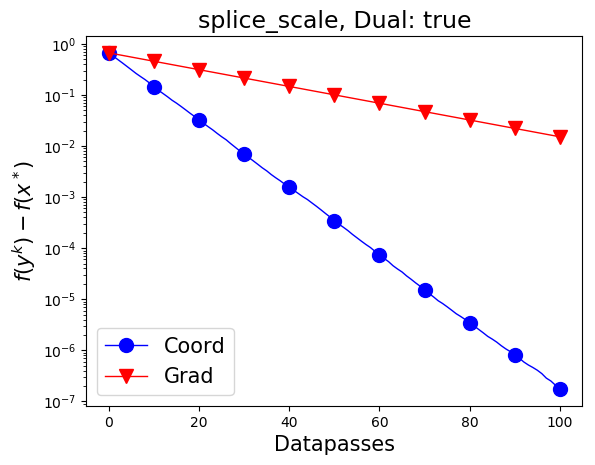}
\end{minipage}%
\begin{minipage}{0.25\textwidth}
  \centering
\includegraphics[width =  \textwidth ]{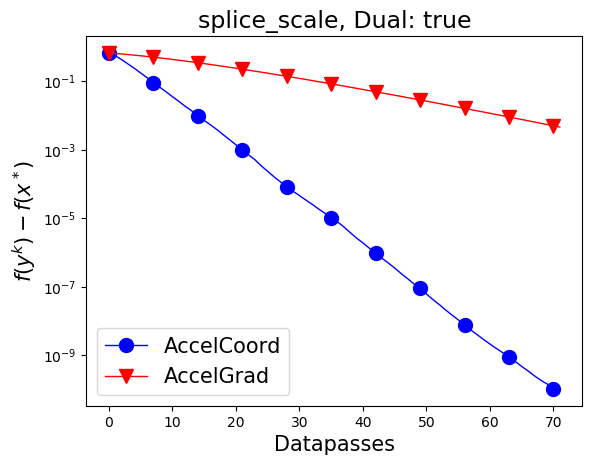}
\end{minipage}%
\begin{minipage}{0.25\textwidth}
  \centering
\includegraphics[width =  \textwidth ]{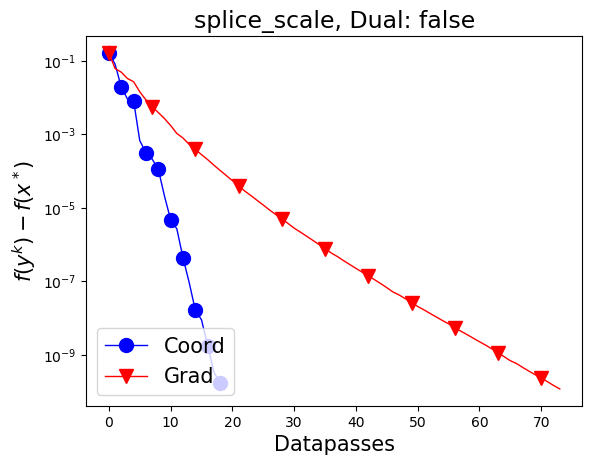}
\end{minipage}%
\begin{minipage}{0.25\textwidth}
  \centering
\includegraphics[width =  \textwidth ]{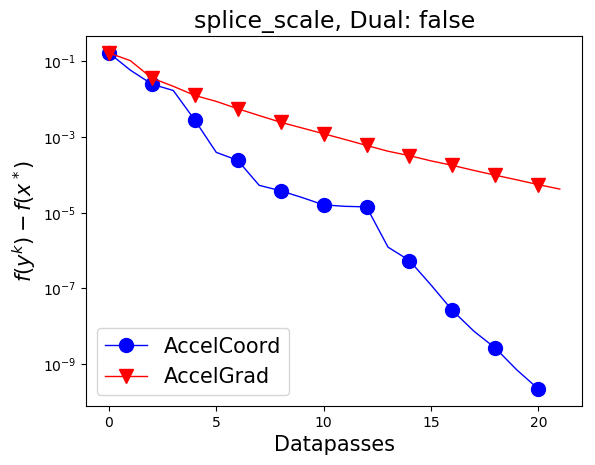}
\end{minipage}%
\\
\begin{minipage}{0.25\textwidth}
  \centering
\includegraphics[width =  \textwidth ]{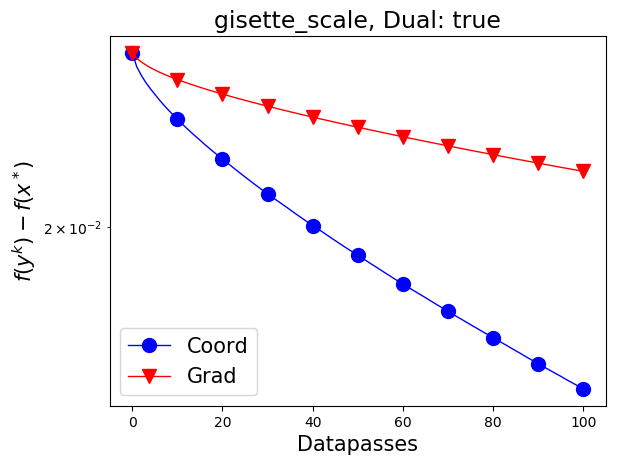}
\end{minipage}%
\begin{minipage}{0.25\textwidth}
  \centering
\includegraphics[width =  \textwidth ]{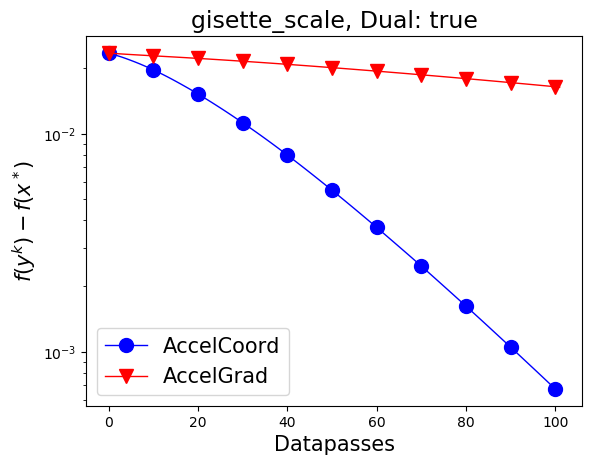}
\end{minipage}%
\begin{minipage}{0.25\textwidth}
  \centering
\includegraphics[width =  \textwidth ]{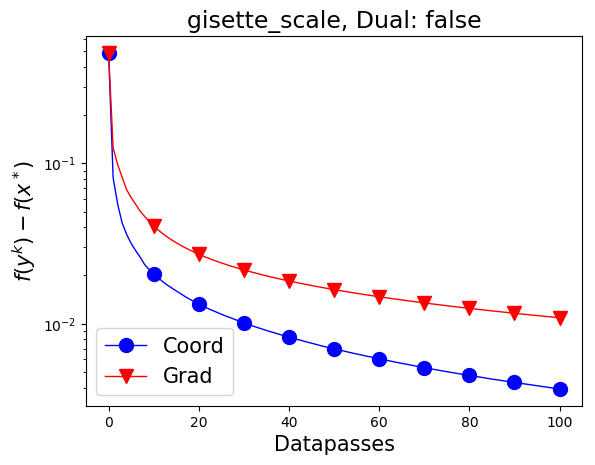}
\end{minipage}%
\begin{minipage}{0.25\textwidth}
  \centering
\includegraphics[width =  \textwidth ]{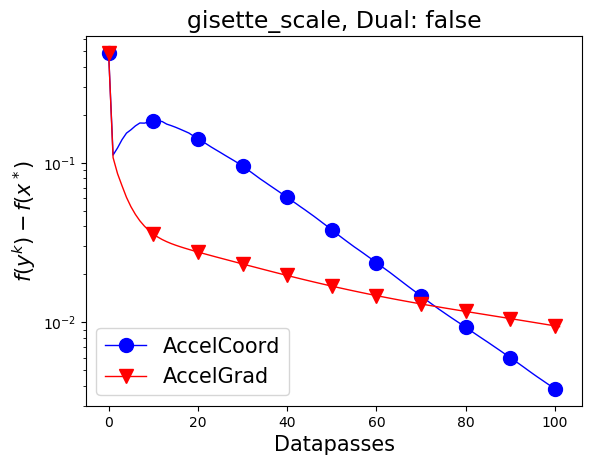}
\end{minipage}%
\caption{Comparison of \texttt{SEGA} and \texttt{ASEGA} with corresponding coordinate descent methods for $R=0$.}\label{fig:cd_cmp}
\end{figure}
\clearpage
\subsection{Experiment: large-scale logistic regression}
In this experiment, we set $\mB$ to be identity matrix and compare \texttt{CD} to \texttt{SEGA}  with coordinate sketches, both with uniform sampling and with similar stepsizes. The problem considered is logistic regression with $\ell_2$ penalty:
\[
	\min_{x\in\R^n} \frac{1}{m}\sum_{i=1}^m \log\left(1 + \exp(-b_i a_i^\top x) \right) + \frac{\mu}{2}\|x\|_2^2,
\]
where $a_i$ and $b_i$ are data-dependent.
 Clearly, this regularizer is separable, so we can easily apply both methods. The value of $\mu$ was chosen to be of order $\frac{1}{m}$ in both experiments. 
Here we use real-world large scale datasets from the LIBSVM \cite{chang2011libsvm} library, a summary can be found in Table~\ref{tab:logreg}.
To make it clear whether \texttt{CD} and \texttt{SEGA}  converge with the same speed if given similar stepsizes, we use stepsize $\frac{1}{L}$ for \texttt{CD} and $\frac{1}{dL}$ for \texttt{SEGA}. The results can be found in Figure~\ref{fig:logreg}.
\begin{table}[h]
\begin{center}
\begin{tabular}{ c c c c c c} 
 \hline
 Dataset & $m$ & $n$ & $L$ & $\mu$ \\ 
 \hline
 Epsilon & 400000  & 2000 & 0.25 & $2.5\cdot 10^{-5}$\\
 Covtype &  581012  & 54 & 21930585.	25 & $10^{-1}$\\
 \hline
\end{tabular}
\end{center}
\caption{Description of the datasets used in our logistic regression experiments. Constants $m$, $n$, $L$ and $\mu$ denote respectively the size of the training set, the number of features, the Lipschitz constant, and the value of $\ell_2$ penalty.}\label{tab:logreg}
\end{table}
\begin{figure}
\begin{center}
\includegraphics[width =  \textwidth/3 ]{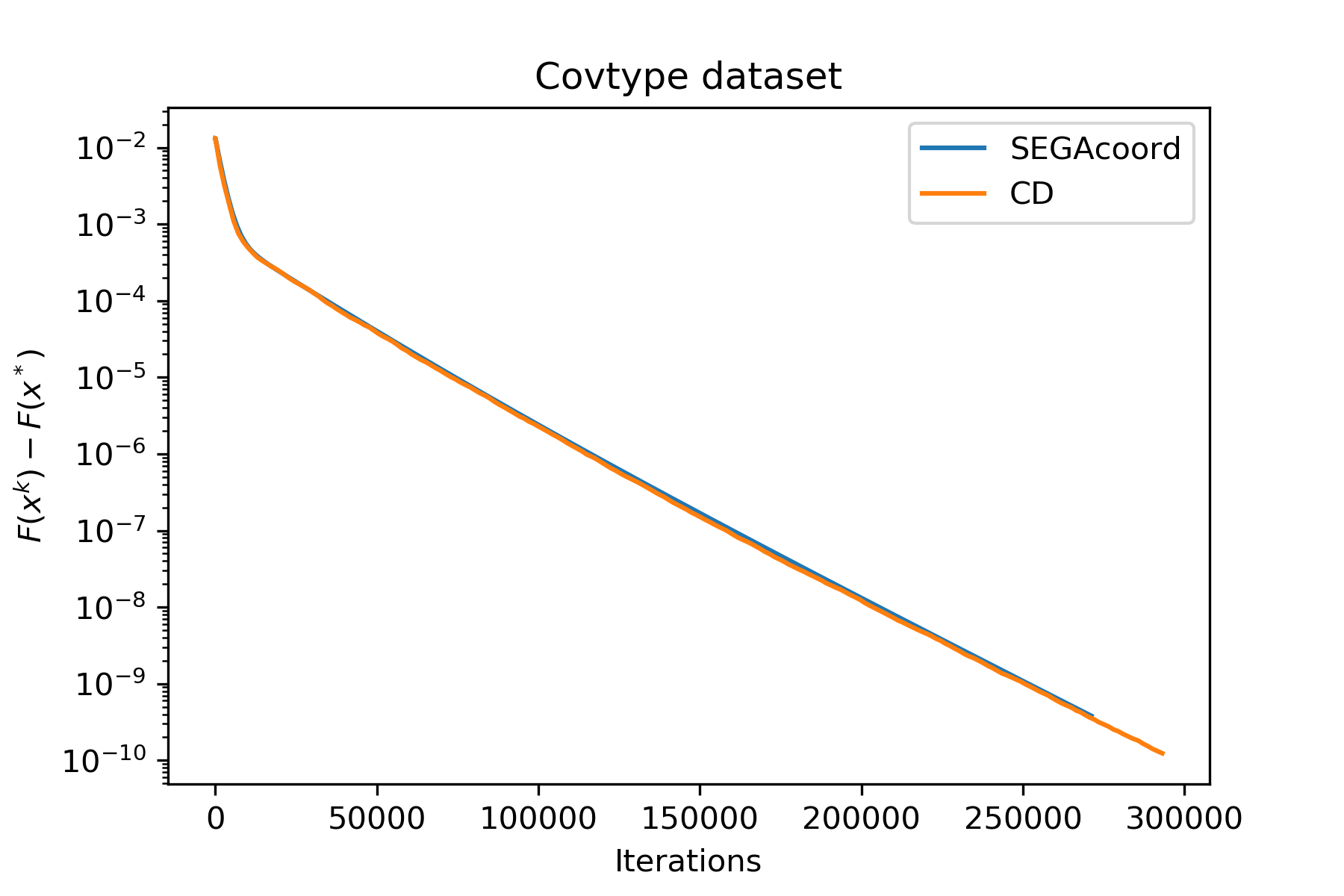}
\includegraphics[width =  \textwidth/3 ]{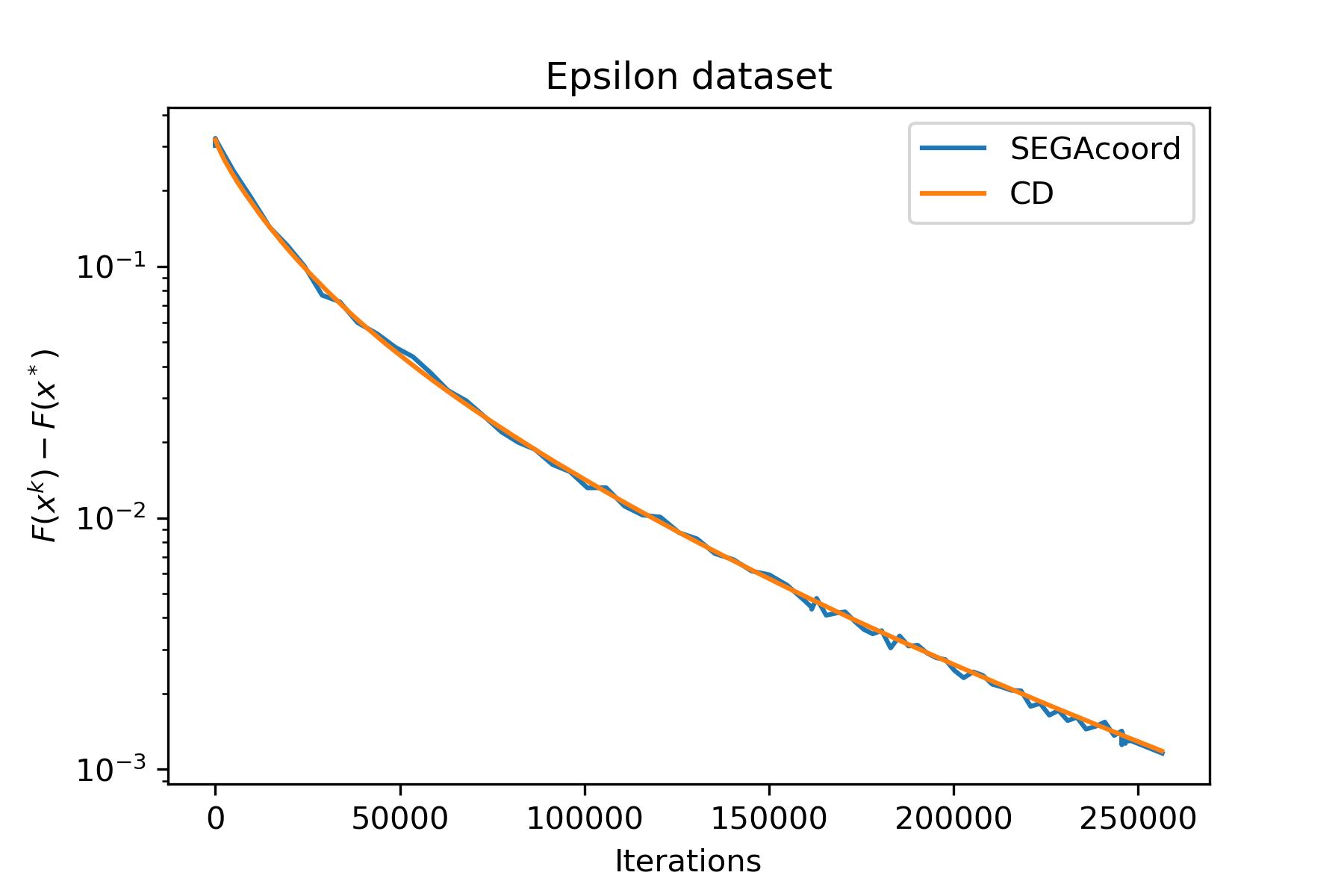}
\end{center}
\caption{Comparison of \texttt{SEGA}  with \texttt{CD} on logistic regression problem with similar stepsizes.} \label{fig:logreg}
\end{figure}

\newpage
\section{Frequently Used Notation \label{sec:notation}}

\begin{table}[!h]
\begin{center}

{
\footnotesize
\begin{tabular}{|c|l|c|}
\hline
\multicolumn{3}{|c|}{{\bf Basic} }\\
\hline
$\E{\cdot}$,  $\Prob{\cdot}$  & Expectation / Probability &\\
$\langle \cdot ,\cdot \rangle_\mB$,  $\| \cdot \|_\mB$ & Weighted inner product and  norm: $\langle x ,y \rangle_\mB =x^\top \mB y$; $\| x \|_\mB =\sqrt{\langle x ,x \rangle_\mB} $ &\\
$e_i$ & $i$-th vector from the standard basis & \\
$\mI$ & Identity matrix & \\
$\lambda_{\max} (\cdot), \lambda_{\min}(\cdot)$ & Maximal eigenvalue / minimal eigenvalue  & \\
 $f$ & Objective to be minimized over set $\R^{n}$ & \eqref{eq:main}\\
 $R$ & Regularizer & \eqref{eq:main}\\
$x^*$ & Global optimum  &\\
$L$ & Lipschitz constant for $\nabla f$ &\\
$\mQ$ & Smoothness matrix &\eqref{eq:M_smooth_inv} \\
$\mM$ & Smoothness matrix, equal to $\mQ^{-1}$ for $\mB=\mI$ &\eqref{eq:M_smooth} \\
$\mu$ & Strong convexity constant & \\

\hline
 \multicolumn{3}{|c|}{{\bf \texttt{SEGA}}}\\
 \hline
  $\cD$ & Distribution over sketch matrices $\mS$ &\\
 $\mS$ & Sketch matrix & \eqref{eq:sketch-n-project}\\
 $\ED{\cdot}$ & Expectation over the choice of $\mS$&\\
  $b$ & Random variable such that $\mS\in \R^{n\times b}$ & \\
  $\zeta(\mS, x)$ & Sketched gradient at $x$ & \eqref{eq:sketched_grad}\\
  $\mZ$ &  $\mS \left(\mS^\top \mB^{-1} \mS\right)^\dagger\mS^\top$ & \\
  $\theta$ & Random variable for which $\ED{\theta  \mZ} = \mB$ & \eqref{eq:unbiased}\\
  $\mC$ & $ \ED{ \theta^2 \mZ }$ & Thm~\ref{thm:main}\\
 $h, g$ & Biased and unbiased gradient estimators & \eqref{eq:h^{k+1}}, \eqref{eq:g^k}\\
  $\alpha$ & Stepsize & \\ 
    $\Lgen$ & Lyapunov function & Thm~\ref{thm:main}, \\ 
     $\sigma$ & Parameter for Lyapunov function & Thm~\ref{thm:main}, \ref{t:imp_nacc}\\ 
 \hline

 \multicolumn{3}{|c|}{{\bf Extra Notation for Section~\ref{sec:CD} }}\\
 \hline
  $p$, $\Probmat$ & Probability vector and matrix &\\
    $v$ & vector of ESO parameters &\eqref{eq:ESO}\\
  $\mPdiag,\mVdiag$ & $\diag(p),\diag(v)$ & \\
  $\gamma$& $\alpha - \alpha^2\max_{i}\{\tfrac{v_{i}}{p_{i}}\}-\sigma$& Thm~\ref{t:imp_nacc} \\
$y,z$ & Extra sequences of iterates for \texttt{ASEGA} & \\
$\tau,\beta$ & Parameters for \texttt{ASEGA} & \\
  $\Lnacc, \Lacc$ & Lyapunov functions & Thm~\ref{t:imp_nacc}, \ref{t:imp_acc}\\ 
   $ \TD(v,p)$ & $\max_i \frac{\sqrt{v_i}}{p_i}$ & \\
 \hline

\hline
\end{tabular}
}

\end{center}
\caption{Summary of frequently used notation.}
\label{tbl:notation}
\end{table}

\end{document}